\documentclass[11pt]{amsart}        
\usepackage{amsmath,amssymb,amsthm}
\usepackage{graphicx} 
\usepackage{stmaryrd}	
\usepackage{dsfont}	
\usepackage{tikz}
\usepackage{mathrsfs} 
\usepackage{nicefrac}

\newtheorem{df}{Definition}[section]
\newtheorem{prop}{Proposition}[section]

\newtheorem{theo}{Theorem}[section]
\newtheorem{lem}{Lemma}[section]
\newtheorem{cor}{Corollary}[section]
\newtheorem{conj}{Conjecture}[section]
\newtheorem{rqqq}{Remark}[section]
\numberwithin{equation}{section}
\title[Local limits of conditioned marked Galton Watson trees.]{Local limits of conditioned marked Galton Watson trees.}
\date{\today}
\author{Romain Abraham}
\address{Romain Abraham. Universit\'{e} d'Orl\'{e}ans, 	Universit\'e de Tours, CNRS, Institut Denis Poisson, UMR 7013, Orl\'{e}ans, France.}
\email{romain.abraham@univ-orleans.fr}

\author{Sonia Boulal}
\address{Sonia Boulal. Universit\'{e} d'Orl\'{e}ans, 	Universit\'e de Tours, CNRS, Institut Denis Poisson, UMR 7013, Orl\'{e}ans, France.}
\email{sonia.boulal@univ-orleans.fr}

\author{Pierre Debs}
\address{Pierre Debs. Universit\'{e} d'Orl\'{e}ans, 	Universit\'e de Tours, CNRS, Institut Denis Poisson, UMR 7013, Orl\'{e}ans, France.}
\email{pierre.debs@univ-orleans.fr}



\thanks{This work is supported by the ANR project ``Rawabranch" number ANR-23-CE40-0008.}

\newcommand{\N}{\mathbb{N}}
\newcommand{\NN}{\mathbb{N}^*}
\newcommand{\Z}{\mathbb{Z}}
\newcommand{\F}{\mathscr{F}}
\newcommand{\G}{\mathcal{G}}
\newcommand{\Le}{\mathscr{L}}
\newcommand{\Lq}{\mathcal{L}}

\newcommand{\U}{\mathcal{U}}

\newcommand{\p}{\mathbb{P}}
\newcommand{\T}{\mathbb{T}}
\newcommand{\Th}{\mathbb{T}^{(h)}}
\newcommand{\E}{\mathbb{E}}
\newcommand{\indic}{\mathds{1}}
\newcommand{\bt}{\mathbf{t}}
\newcommand{\bp}{\mathbf{p}}
\newcommand{\bq}{\mathbf{q}}
\begin{document}

	\begin{abstract}
		We consider a Galton-Watson tree where each node is mar\-ked independently of each others with a probability depending on its out-degree. We give a complete picture of the local convergence of critical or sub-critical
		marked Galton-Watson trees conditioned on having a large number of marks. 
		In the critical and sub-critical generic case, the limit is a random marked tree with an
		infinite spine, named marked Kesten’s tree. We focus also on the non-generic
		case, where the local limit is a random marked tree with a node with infinite out-degree. This
		case corresponds to the so-called marked condensation phenomenon. 
	\end{abstract}
	\maketitle
	\section{Introduction}
	The Galton-Watson process was introduced by Bienaymé in 1845 and independently by Galton in 1873 in order to study the disappearance of surnames.
	\smallbreak
	This process is a very simple model of population growth where all individuals give birth independently of each others to a random number of children according to the distribution $\bp$. In other words, to each generation each individual could have $k$ children with probability $\bp(k)$. Thus $\bp$ is called the offspring distribution.\\
	Let $\bp:=(\bp(n))_{n\in\N}$ be an offspring distribution satisfying:
	\begin{equation}\label{condp}
		\bp(0)>0~,~\bp(0)+\bp(1)<1~ \text{and}~ \bp \text{ has a moment of order 2.}
	\end{equation}
	We denote by $\mu(\bp):=\sum_{n\ge 0}n\bp(n)$, its mean. If $\mu(\bp)<1$ (resp. $\mu(\bp)=1$, $\mu(\bp)>1$), we say that the offspring distribution is sub-critical (resp. critical, super-critical). In the sub-critical and critical case we have almost surely population extinction.
	\smallbreak
	In 1986, Neveu introduced the notion of Galton-Watson tree (see \cite{neveu_arbres_1986}). This tree is a random genealogical tree, noted $\tau$, that describes the population growth associated with the offspring distribution $\bp$. 
	\smallbreak
	Conditioning critical or sub-critical Galton-Watson trees comes from the work of Kesten, in 1986, \cite{kesten_subdiffusive_1986}. In the sub-critical or critical cases, the tree is almost surely finite, but Kesten considered in \cite{kesten_subdiffusive_1986} the local limit of a sub-critical or critical tree conditioned to have height greater than $n$. When $n$ goes to infinity, this conditioned tree converges in distribution to an infinite tree called here Kesten's tree. This tree has an infinite spine. A random number of independent Galton-Watson trees, with the same offspring distribution $\bp$, are grafted onto this spine. This limit tree can be seen as a Galton-Watson tree conditioned on non-extinction.
	\smallbreak
	Since then, other ways of conditioning a critical Galton-Watson tree to be large have been considered: large total progeny (see Kennedy \cite{kennedy_galton-watson_1975} and Geiger and Kaufman \cite{geiger_shape_2004}), large number of leaves (see Curien and Kortchemski \cite{curien_random_2014}). In the sub-critical and critical cases, Rizzolo \cite{rizzolo_scaling_2015} introduces the conditioning on having a large number of individuals with the number of offsprings belonging to a given set $\mathcal{A}$, it also appears in the paper of Kortchemski \cite{kortchemski_invariance_2012}. 
	All these cases are contained in the general result of Abraham-Delmas \cite{abraham_local_2013} and the limiting tree in the critical case is always the Kesten's tree associated with $\bp$.
	In 2017, in \cite{abraham_local_2017}, Abraham, Bouaziz and Delmas generalized this approach by marking the nodes randomly where, conditionnally on the tree, the nodes are marked independently of each others with a probability that may depend on their out-degree.  More precisely, independently of each other, each individual gives birth to $k$ children and is
	\begin{align*}
		\bullet& \text{ marked with probability }	\bp_{0}(k,1):=\bp(k)\bq(k);\\
		\bullet& \text{ unmarked with probability }	\bp_{0}(k,0):=\bp(k)(1-\bq(k))
	\end{align*}
	where $\bq:=(\bq(k))_{k\ge0}$ is a sequence of numbers in $[0,1]$ and is called the mark function. The probability distribution $\bp_0$ on $\N\times\{0,1\}$ is called here the marking-reproduction law of the associated marked Galton-Watson tree (MGW). We always suppose that the condition 
	\begin{equation}\label{condq}
		\exists k\in\N^*,\ \bp(k)\bq(k)>0
	\end{equation}
	holds so a  MGW contains at least a mark with positive probability. Then they condition the tree on having a large number of marked nodes and prove a local convergence in the critical case of this conditioned marked tree toward Kesten's tree (as for the other conditionnings) as the number of marked nodes tends to infinity. \\
	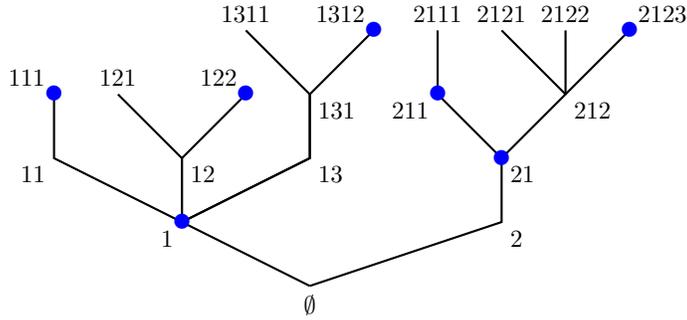
\begin{figure}[h!]
		\centering
		\begin{tikzpicture}[scale=0.85,every node/.style={scale=0.85}]
			\coordinate (0) at (0,0);
			\coordinate (1) at (-2,1);
			\coordinate (2) at (3,1);
			\coordinate (11) at (-4,2);
			\coordinate (12) at (-2,2);
			\coordinate (13) at (0,2);
			\coordinate (111) at (-4,3);
			\coordinate (121) at (-3,3);
			\coordinate (122) at (-1,3);
			\coordinate (131) at (0,3);
			\coordinate (1311) at (-1,4);
			\coordinate (1312) at (1,4);
			\coordinate (21) at (3,2);
			\coordinate (211) at (2,3);
			\coordinate (212) at (4,3);
			\coordinate (2111) at (2,4);
			\coordinate (2121) at (3,4);
			\coordinate (2122) at (4,4);
			\coordinate (2123) at (5,4);
			\draw [thick](0) -- (2) -- (21)--(211)--(2111);
			\draw [thick](0) -- (1)-- (13)--(131)--(1312);
			\draw [thick](21) -- (212)--(2121);
			\draw [thick](1) -- (12)--(121);
			\draw [thick](1) -- (11)-- (111);
			\draw [thick](1) -- (12)-- (122);
			\draw [thick](1) -- (13)-- (131)--(1311);
			\draw [thick](212) -- (2123);
			\draw [thick](212) -- (2122);
			\draw (0) node[below]{$\emptyset$};
			\draw (1) node{\textcolor{blue}{\LARGE{$\bullet$}}}node[below left]{1};
			\draw (2) node[below right]{2};
			\draw (11) node[below left]{11};
			\draw (12) node[below right]{12};
			\draw (13) node[below right]{13};
			\draw (111) node{\textcolor{blue}{\LARGE{$\bullet$}}}node[above left]{111};
			\draw (121) node[above]{121};
			\draw (122) node{\textcolor{blue}{\LARGE{$\bullet$}}}node[above left]{122};
			\draw (131) node[below right]{131};
			\draw (1311) node[above]{1311};
			\draw (1312) node{\textcolor{blue}{\LARGE{$\bullet$}}}node[above left]{1312};
			\draw (21) node{\textcolor{blue}{\LARGE{$\bullet$}}}node[below right]{21};
			\draw (211) node{\textcolor{blue}{\LARGE{$\bullet$}}}node[below left]{211};
			\draw (212) node[below right]{212};
			\draw (2111) node[above]{2111};
			\draw (2121) node[above]{2121};
			\draw (2122) node[above]{2122};
			\draw (2123) node{\textcolor{blue}{\LARGE{$\bullet$}}}node[above right]{2123};
		\end{tikzpicture}
		\caption{A marked tree $\bt^{*}$ with 5 generations.}
		\label{examplemarkedtree}
	\end{figure}
	\newpage
The situation in the sub-critical case is more involved (see Janson \cite{janson_simply_2011} when conditioning on the total progeny, and Abraham-Delmas \cite{abraham_local_2014} when conditioning the number of nodes with offspring in a given set $\mathcal{A}$). Indeed two limits may appear: a Kesten's tree associated with a modified offspring distribution in the so-called generic case, or a condensation  tree ( a tree with a node with infinite out-degree) in the non-generic case.

The main goal of the paper is to study the sub-critical case when conditioning on the number of marks. To this aim, we first need to introduce a one-parameter family of modified mark-reproduction law. 

For everey $\theta>0$,
	\begin{equation}\label{ptheta}
		\forall k \geq 0, ~\bp_\theta(k):=\theta^{k-1} \bp(k) \left(c_\theta \bq(k)+1-\bq(k)\right),
	\end{equation}
	where $c_\theta$ is a normalizing constant given by:
	\begin{equation}\label{ctheta}
		c_\theta:=\dfrac{1-\E\left[\theta^{X-1}(1-\bq(X))\right]}{\E\left[\theta^{X-1}\bq(X)\right]},
	\end{equation}
	and $X$ is a random variable distributed according to $\bp$. We also define a new mark function denoted by $\bq_\theta$ such that for all $k\ge0$:
	\begin{equation}\label{qtheta}
		\bq_{\theta}(k) := \dfrac{c_\theta \bq(k)}{c_\theta \bq(k) +1-\bq(k)}.
	\end{equation}\\
	Let $I$ be the set of positive $\theta$ such that $\bp_\theta$ defines a probability distribution on $\N$. If $\bp$ is sub-critical, according to Lemma \ref{exthetacrit} either there exists a unique $\theta_c\in I$ such that $\bp_{\theta_c}$ is critical (generic case), or $\theta_s:=\max I \in I$ and $\bp_{\theta_s}$ is sub-critical (non-generic case). 

For technical reasons, we need some additional assumptions on $\bp$ and $\bq$.
	We assume that there exists $\alpha> 2$ and a slowly varying function (SV) $ \Le$ such that, for all $k\ge 1$,
	\begin{equation}\label{probpalpha}
		\bp(k)=\Le(k)k^{-(1+\alpha)}.
	\end{equation}
	Moreover, we assume that $\bq$ admits a finite limit at infinity, so we have:
	\begin{equation}\label{limitq}
		\lim_{n\rightarrow+\infty}\bq(n)=:\ell_\bq \in[0,1].
	\end{equation}
	If $\ell_\bq=1$, we also assume that  the mark function $\bq$ satisfies for $k\in\NN$ ,  
	\begin{equation}\label{defqlim1}
		1-\bq(k)=k^{-\beta}\Lq(k) 
	\end{equation}
	with $\beta\ge 2$ and $\Lq$ is a SV function.\\
	For a marked tree $\bt^*$, we present our two main results. 
	We denote by $M(\bt^*)$ the number of marked of this tree, and by $\rho(\bp)$ the convergence radius of the generative function associated to $\bp$.
	We denote by $\tau^{*}_K(\bp,\bq)$ a marked Kesten's tree and $\tau_C^{*}(\bp,\bq)$ a marked condensation tree (see Sub-section \ref{kestcondtree}). Finally, we denote by $\mathrm{dist}(T)$ the distribution of the random variable $T$.
	\begin{theo}\label{theocondconvloc}
		Let $\tau^{*}(\bp,\bq)$ be a marked Galton-Watson tree with offspring distribution $\bp$ satisfying \eqref{condp}, and mark function $\bq$ satisfying \eqref{condq}.\\
				In the critical case we have 
		$$ \mathrm{dist}(\tau^{*}(\bp,\bq)| M(\tau^*(\bp,\bq))=n)\underset{n\to +\infty}{\to}\mathrm{dist}(\tau_K^{*}(\bp,\bq)).$$
		Generic sub-critical case ( $\exists \theta_c\in I$ s.t. $\bp_{\theta_c}$ critical), if $\bp_{\theta_c}$ admits a moment of order $2$ (always true for $\theta_c< \rho(\bp)$) we have 
		$$ \mathrm{dist}(\tau^{*}(\bp,\bq)| M(\tau^*(\bp,\bq))=n)\underset{n\to +\infty}{\to} \mathrm{dist}(\tau_K^{*}(\bp_{\theta_c},\bq_{\theta_c})).$$
		Non-generic sub-critical case ( $\forall \theta\in I$, $\bp_{\theta}$ is sub-critical), if $\bq$ satisfied \eqref{limitq} and, if $\ell_\bq=1$ \eqref{defqlim1} we have 
		$$ \mathrm{dist}(\tau^{*}(\bp,\bq)| M(\tau^*(\bp,\bq))=n)\underset{n\to +\infty}{\to}\mathrm{dist}(\tau_C^{*}(\bp,\bq)).$$
	\end{theo}

	\smallbreak
	The paper is organized as follows: in Section \ref{techback} we introduce the set of discrete marked trees and define the MGWs. We also explain the construction of Kesten's tree and condensation tree, and some convergence criterions which is the key to prove the convergence. \\
	In Section \ref{markcond}, we observe the behavior of the MGW, when we condition with the total number of marks. Sub-section \ref{resultcond} is devoted to prove some results about the biased law.
	We explain properties about generic and non-generic distributions in Sub-section \ref{gennongencase}. In the Sub-section \ref{Model}, we construct a model, that permits to prove our results.\\ Section \ref{condconvresult} is devoted to the proof of Theorem \ref{theocondconvloc}. \\We prove in the appendix, Section \ref{appendix}, lemmas that we used in the previous section.
	\section{Technical background}\label{techback}
	\subsection{The set of discrete trees}\label{settree}
	Let $\N$ be the set of nonnegative integers and $\NN$ the set of positive integers. We recall Neveu's formalism \cite{neveu_arbres_1986}  for ordered rooted trees. We denote by
	\[\mathcal{U}=\bigcup_{n\geq 0}(\NN)^n,\]
	 the set of finite sequences of positive integers with the convention $(\NN)^0=\{\emptyset\}$. For $n\geq 0$ and $u=(u_1,...,u_n)\in\mathcal{U}$, let $\left|  u \right| =n$ be the length of $u$ and $\left|  u \right|_{\infty} =\max\{\left|  u \right|,u_1, u_2,...,u_{\left|  u \right|}\}$ with the convention $\left|  \emptyset \right|=\left|  \emptyset \right|_{\infty}=0$. We call $\left|  u \right|_{\infty}$ the norm of $u$ although it is not a norm since $\mathcal{U}$ is not even a vector space. If $u$ and $v$ are two sequences of $\mathcal{U}$, we denote by $uv$ their concatenation,
	with the convention that $uv=u$ if $v=\emptyset$ and $uv=v$ if $u=\emptyset$. The set of ancestors of $\emptyset$ is $A_\emptyset=\{\emptyset\}$ and of $u\neq \emptyset$ is:
	\begin{equation}\label{ancestor}
		A_u=\{v\in\mathcal{U}; \text{ there exists }w\in \mathcal{U},w\neq \emptyset,\text{ such that } u=vw\}.
	\end{equation}
	The most recent common ancestor of a subset $s$ of $\mathcal{U}$, denoted by $MRCA(s)$, is the unique element $v$ of $\bigcap_{u\in s}A_u$ with maximal length $\left|  v \right|$. For $u,v\in\U$, we denote by $u<v$, the lexicographic order on $\U$, i.e. $u<v$ if either $u\in A_v$ or, if we set $w=MRCA(\{u,v\})$, then $u=wiu'$ and $v=wjv'$ for some $i,j\in\NN$ with $i<j$.
	\bigbreak
	A tree $\bt$ is a subset of $\U$ that satisfies:
	\begin{itemize}
		\item $\emptyset \in \bt$;
		\item if $u\in \bt$, then $A_u \subset \bt$;
		\item for every $u\in \bt$, there exists $k_u(\bt)\in\N\cup\{+\infty\}$ such that, for every positive integer $i$, $ui\in \bt \iff 1\leq i \leq k_u(\bt)$.
	\end{itemize}
	The integer $k_u(\bt)$ represents the number of offsprings of the vertex $u\in \bt$. If $k_u(\bt)=0$, the vertex $u\in \bt$ is called a leaf and if $k_u(\bt)=+\infty$, $u$ is said infinite. By convention, we shall set $k_u(\bt)=-1$ if $u\notin \bt$. The vertex $\emptyset$ is called the root of $\bt$.
	\\
	Its height and its {\it norm}, which can be infinite, are respectively defined by:
	\begin{align*}
		H(\bt)&=\sup\{\left|  u\right|,~u\in \bt\};\\
		H_\infty(\bt)&=\sup\{\left|  u\right|_\infty,~u\in \bt\}=\max(H(\bt),\sup\{k_u(\bt),u\in \bt\}).
	\end{align*}
	\medbreak
	We denote by: 
	\begin{itemize}
		\item $\T_{\infty}$, the set of trees;
		\item $\mathcal{L}_0(\bt):=\{u\in \bt, k_u(\bt)=0\}$, the set of leaves of $\bt\in \T_{\infty}$;
		\item $\T_0:=\{\bt\in \T_{\infty}, \left| \bt \right| < + \infty\}$,
		the subset of finite trees, where $\left|  \bt \right|$ is the cardinal of $\bt$;
		\item $\Th_{\infty}:=\{\bt\in \T_{\infty}, H_{\infty}(\bt)<h\}$,
		the subset of finite trees with norm less than $h\in \N$, $H_{\infty}(\bt)=\sup\{\left| u \right|_{\infty}, u \in \bt\}$;
		\item $\T:=\{\bt\in\T_{\infty}; k_u(\bt)<+{\infty}~ \forall u \in \bt\}$, 
		the subset of trees with no infinite vertex;
		\item $\T_1:=\{\bt\in\T_{\infty}; \underset{n\to +\infty}{\lim} |MRCA(\{u\in\bt;|u|=n\})|=+\infty\}$, the subset of trees with a unique infinite spine;
		\item $\T_2$, the subset of trees with no infinite spine and with exactly one infinite vertex;
		\item $\bt\circledast_x \bt':=\bt\bigcup\{x\zeta(v,k_x(\bt)),~v\in \bt' \setminus \{\emptyset \}\}$,
		the tree obtained by grafting the tree ${\bt'}\in\T_{\infty}$ at $x\in \bt$ on ``the right" of $\bt\in \T_\infty$ , with $\zeta(v,k):=\left(v_1+k, v_2,...,v_n\right)$ if $v=\left(v_1,v_2,...,v_n\right)\in \U$ with $k\in\NN$ and $n>0$;
		\item $\T(\bt,x):=\{\bt\circledast_x \bt', \bt' \in\T_{\infty}\}$
		the set of trees obtained by grafting a tree at $x\in \bt$ on ``the right" of $\bt\in \T_0$;
		\item $\T_{+}(\bt,x,k):=\{s\in \T(\bt,x), k_x(s)\geq k\}$
		the subset of $\T(\bt,x)$ such that the number of offspring of $x\in s$ is $k$ or more.
		
	\end{itemize}

Let $\bt$ be a tree, for $u\in \bt$, we define the sub-tree $\mathcal{S}_u(\bt)$ of $\bt$  ``above" $u$ as:\\
$\mathcal{S}_u(\bt): =\{v\in\mathcal{U}; ~uv\in \bt\}$.\\
For $u\in \bt\setminus\mathcal{L}_0(\bt)$, we also define the forest $\mathcal{F}_u(\bt)$ ``above" $u$ as:\\
$\mathcal{F}_u(\bt): =\left(S_{ui}(\bt);~ 1\leq i\leq k_u(\bt)\right)$.\\
For $u\in \bt\setminus\{\emptyset\}$, we also define the sub-tree $\mathcal{S}^u(\bt)$ of $\bt$ ``below" $u$ as:\\
$\mathcal{S}^u(\bt): =\{v\in \bt;~u\notin A_v\}$.
\smallbreak
	For $h\in\N$, the restriction function $r_{h,\infty}$ from $\T_\infty$ to $\T_\infty$ is defined by : $$r_{h,\infty}(\bt):=\{u\in \bt,\left|  u \right|_\infty \leq h\},$$ and the restriction function $r_h$ from $\T$ to $\T$ is defined by : $$r_{h}(\bt):=\{u\in \bt,\left|  u \right| \leq h\}.$$
	We endow the set $\T_{\infty}$ (resp. $\T)$, with the ultra-metric distance 
	\[d_\infty(\bt,\bt'):=2^{-\max\{h\in\N,r_{h,\infty}(\bt)=r_{h,\infty}(\bt')\}},\]\[~(\text{resp. } d(\bt,\bt'):=2^{-\max\{h\in\N,r_{h}(\bt)=r_{h}(\bt')\}}).\]
	A sequence $(\bt_n)_{n\in\N}$ of trees converges to a tree $\bt$ with respect to the distance $d_\infty$ (resp. $d$) if and only if, for every $h\in\N$,
	\[r_{h,\infty}(\bt_n)=r_{h,\infty}(\bt)~~(\text{ resp. } r_{h}(\bt_n)=r_{h}(\bt))~\text{for } n \text{ large enough.}\]
	The Borel $\sigma$-field associated with the distance $d_\infty$ (resp. $d$) is the smallest $\sigma$-field containing the singletons for which the restrictions functions $(r_{h,\infty})_{h\in\N}$ (resp. $(r_{h})_{h\in\N}$) are measurable. With this distance, the restriction function are contractant. According to \cite{jonsson_condensation_2011}, $\T_0$ is dense in $\T_{\infty}$ (resp. $\T_0\cap\T \in \T)$ and since,  $(\T_{\infty},d_\infty)$ (resp. $(\T,d)$) is complete, we get that $(\T_{\infty},d_\infty)$ (resp. $(\T,d)$) is a Polish metric space. Moreover, according to \cite{jonsson_condensation_2011}, $(\T_{\infty},d_\infty)$ is compact.
	
	\subsection{Galton-Watson trees}
	Let $\bp=(\bp(n))_{n\in\N}$ be a probability distribution on the set of the non-negative integers 
	satisfying \eqref{condp}.
	Let $g(z):=\E[z^X]$ be the generating function of $X$, where $X$ is a random variable with distribution $\bp$. We denote by $\rho(\bp)$ the radius of convergence of $g$ and we write $\rho$ when it is clear from the context.
	A $\T$-valued random variable $\tau(\bp)$ (noted $\tau$ if it is clear) is a Galton-Watson tree (GW) with offspring distribution $\bp$, if the distribution of $k_\emptyset(\tau)$ is $\bp$ and for $n\in\NN$, conditionally on $\{k_\emptyset(\tau)=n\}$, the sub-trees $(\mathcal{S}_1(\tau),\cdots,\mathcal{S}_n(\tau))$ are independent and distributed as the original tree $\tau$.
	The distribution of a GW tree is characterized by:
	\begin{equation}\label{probaTrh}
		\forall h \ge 1, \forall \bt\in\Th, ~	\p(r_{h}(\tau) = \bt)= \prod_{u\in\bt, |u|<h} { \bp(k_u(\bt))}.
	\end{equation}
 In particular, 
 \begin{equation}\label{probaT}
 			\forall \bt \in \T_0, ~\p(\tau = \bt)= \prod_{u\in \bt} { \bp(k_u(\bt))}.
 		\end{equation}

	For a tree $\bt$, we have:
	\begin{equation}\label{sumku}
		\underset{u\in \bt}{\sum} k_u(\bt) = \left|  \bt \right|-1.
	\end{equation}
	We recall that, if:
	\begin{itemize}
		\item $\mu(\bp)<1$ we are in the sub-critical case;
		\item $\mu(\bp)=1$ we are in the critical case;
		\item $\mu(\bp)>1$  we are in the super-critical case.
	\end{itemize}
	
	We say that $\bp$ is aperiodic if $\{k\in\N;~\bp(k)>0\}\subset d\N$ implies $d=1$.
	\smallbreak
	Let $\p_k$ be the distribution of the forest $\tau^{(k)}=(\tau_1,\cdots,\tau_k)$ of i.i.d GW with offspring distribution $\bp$. We set:
	\begin{equation}\label{cardforest}
		\left|  \tau^{(k)} \right| = \sum_{j=1}^k\left|  \tau_j \right|.
	\end{equation}
	\subsection{The set of marked discrete trees}\label{markproc}
	A marked tree $\bt^*$ is defined by a tree $\bt\in \T_\infty$ and a mark $\eta_u\in\{0,1\}$ for every node $u\in\bt$, that is
	\[
	\bt^*=\bigl(\bt,(\eta_u)_{u\in\bt}\bigr).
	\]
	A node $u\in\bt$ is said to be marked if $\eta_u=1$ and unmarked if $\eta_u=0$. Throughout the paper, the notations defined to $\bt$ and used for $\bt^{*}$ are the same.
	For instance $\T_0^{*}:=\{(\bt,(\eta_u(\bt))_{ u\in \bt}), \bt\in \T_0\}$.
	We denote by $\T_\infty^*$ the set of marked trees.\\
	We also denote by $M(\bt^*):=\sum_{u\in \bt}\eta_u$ the number of marked vertices. 
	\smallbreak
	For every $h\in \N$, we define the restriction functions
	\[
	r_h^*:\T^*\longrightarrow\T^*,\qquad\text{and}\qquad r_{h,\infty}^*:\T_\infty^*\longrightarrow\T_\infty^*
	\]
	by, for $\bt^*=(\bt, (\eta_u)_{u\in\bt})\in\T^*$,
	\begin{equation*}
		r_h^*(\bt^*)=\bigl(r_h(\bt), (\eta_u^h)_{u\in r_h(\bt)}\bigr),\qquad
		r_{h,\infty}^*(\bt^*)=\bigl(r_{h,\infty}(\bt), (\eta_u^h)_{u\in r_{h,\infty}(\bt)}\bigr)
	\end{equation*}
	with 
	\[
	\eta_u^h=\begin{cases}
		\eta_u & \text{if } \left|  u \right|\le h-1,\\
		0 & \text{if }\left|  u \right|=h.
	\end{cases}
	\]	
	
	We can endow the set $\T_\infty^*$ (resp.  $\T^*$) with the $\sigma$-field $\F$ generated by the family of sets $\left(\left\{\bt^*\in\T_\infty^*,\ u\in\bt\right\}\right)_{u\in\U}$ $\left(\text{resp. }\left(\left\{\bt^*\in\T^*,\ u\in\bt\right\}\right)_{u\in\U}\right)$ and hence define probability measures on $\left(\T_\infty^*,\F\right)$ (resp.  $\left(\T^*,\F\right)$).
	
	We also endow $\T_\infty^*$  (resp.  $\T^*$) with the filtration $(\F_n)_{n\ge 0}$ where $\F_n$ is the $\sigma$-field generated by the restiction function $r_{n,\infty}^*$ (resp. $r_{n}^*$ ). Notice that $\F=\bigvee_{n\ge 0} \F_n$ as, for every $u\in\U$,
	\[
	\left\{\bt^*\in\T_\infty^*,\ u\in\bt\right\}\in\F_{|u|}~\left(\text{resp. }\left\{\bt^*\in\T^*,\ u\in\bt\right\}\in\F_{|u|}\right).
	\]
%
	Let $\bt^{*}$, $\bt'^{*}$ be two marked trees and $x\in \bt$. 
	If $x$ is marked in $\bt^{*}$, $x$ is marked in $\bt^{*}\circledast_x \bt'^{*}$, we forget the mark on the root of $\bt'^{*}$. For $x\in \bt\setminus\{\emptyset\}$, $x$ is not marked on $\mathcal{S}^x(\bt^{*})$:
	\begin{equation*}
		\mathcal{S}^x(\bt^{*})=\{(u,\eta_u(\bt)), u\in \mathcal{S}^x(\bt)\setminus \{x\}\}\bigcup \{(x,0)\}.
	\end{equation*} 

	\subsubsection{Marked Galton-Watson trees}
	
	Let $\bp_0=(\bp_0(k,\eta))_{(k,\eta)\in\N\times \{0,1\}}$ be a probability distribution on $\N\times\{0,1\}$. There exists a unique probability measure $\p_{\bp _0}$ on $(\T^*,\F)$ such that, for all $h\in\NN$ and all $\bt^*=(\bt,(\eta_u)_{u\in\bt})\in\T_\infty^*$ (resp. $\T^*$), if $\tau^*$ denotes the canonical random variable on $\T_\infty^*$  (resp. $\T^*$),
\begin{align*}
	\p_{\bp_0}(r_h^*(\tau^*)=r_h^*(\bt^*))&=\prod_{u\in r_{h-1}(\bt)} { \bp_0(k_u(\bt),\eta_u)},
%
\end{align*}
	We say that the r.w. $\tau^*$ under $\p_{\bp_0}$ is a marked Galton-Watson tree (MGW) with reproduction-marking law $\bp_0$.
	As $\tau^{*}\in \T_0^{*}$, we have for all marked tree $\bt^*\in\T_0^{*}$,
	\begin{align}\label{probaTmark}
		\p(\tau^{*}=\bt^{*})&=\prod_{u\in \bt}\bp_0(k_u(\bt),\eta_u)\\
		&=\p(\tau=t)\prod_{u\in \bt,\eta_u=1}\bq(k_u(\bt))\prod_{u\in \bt,\eta_u=0}\left(1-\bq(k_u(\bt))\right)\nonumber
		.
	\end{align}
	Equivalently, the distribution of a marked Galton-Watson tree with re\-pro\-duc\-tion-marking law $\bp_0$ is also characterized by an offspring distribution $\bp$ and a mark function $\bq:\N\longrightarrow [0,1]$ with
	\[
	\bp(k)=\bp_0(k,0)+\bp_0(k,1),\quad \bq(k)=
	\frac{\bp_0(k,1)}{\bp(k)} \text{ if }\bp(k)\ne 0
	\]
	(the value of $\bq(k)$ has no importance if $\bp(k)=0$), or equivalently
	\begin{equation}\label{def_p0}
		\bp_0(k,1)=\bp(k)\bq(k),\qquad \bp_0(k,0)=\bp(k)(1-\bq(k)).
	\end{equation}
	This approach (giving $\bp$ and $\bq$) consists in first picking a Galton-Watson tree with offspring distribution $\bp$ then conditionally on the tree, adding marks on every node, independently of each others, with probabily $\bq(k)$ where $k$ is the out-degree of the node.
	
	We will then write $\p_{\bp_0}$ or $\p_{\bp,\bq}$ depending on the adopted point of view, or simply $\p$ when the context is clear.

	\medskip
	Under $\p_{\bp,\bq}$, the random tree $\tau^*$ satisfies the so-called branching property that is
	\begin{itemize}
		\item The random variable $k_\emptyset(\tau^*)$ is distributed according to the probability distribution $\bp$;
		\item Conditionally on $\{k_\emptyset(\tau^*)=j\}$, the root is marked with probability $\bq(j)$;
		\item Conditionally on $\{k_\emptyset(\tau^*)=j\}$, the $j$ sub-trees attached to the root are i.i.d. random marked trees distributed according to the probability $\p_{\bp,\bq}$.
	\end{itemize}

	\subsection{Kesten's tree and condensation tree}\label{kestcondtree}
	
	Let $\tau_l(\bp)$ denote the random tree defined by:
	\begin{enumerate}
		\item there are two types of nodes: normal and special;
		\item the root is special;
		\item normal nodes have offspring distribution $\bp$;
		\item special nodes have for offspring distribution the biased distribution $\check{\bp}$ on $\N \bigcup \{+\infty\}$ defined by:
		$$\check{\bp}(k) :=
		\left\{
		\begin{array}{ll}
			k\bp(k)& \text{if}~k\in \N,\\
			1-\mu & \text{if} ~k=+\infty.
		\end{array}
		\right. $$
		\item The offsprings of all the nodes are independent of each others;
		\item all the children of a normal node are normal;
		\item when a special node gets a finite number of children, one of them is selected uniformly at random and is special while the others are normal;
		\item When a special node gets an infinite number of children, all of them are normal.
	\end{enumerate}
	$\tau_l(\bp)$ represents a limit tree for the rest of the document.\smallbreak
	We can notice that:
	\begin{itemize}
		\item if $\bp$ is critical, $\tau_l(\bp)=\tau_K(\bp)$;
		\item if $\bp$ is sub-critical, $\tau_l(\bp)=\tau_C(\bp)$.
	\end{itemize}
	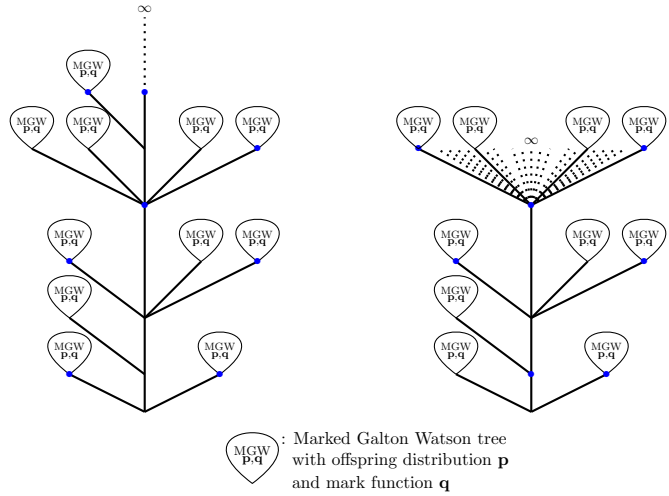
\begin{figure}[!h]
		\begin{tikzpicture}[scale=0.5,every node/.style={scale=0.5}]
		\coordinate (0) at (0,0);
		\coordinate (1) at (-2,1);
		\coordinate (3) at (2,1);
		\coordinate (2) at (0,1);
		\coordinate (21) at (-2,2.5);
		\coordinate (22) at (0,2.5);
		\coordinate (223) at (1.5,4);
		\coordinate (224) at (3,4);
		\coordinate (221) at (-2,4);
		\coordinate (222) at (0,4);
		\coordinate (2222) at (0,5.5);
		\coordinate (22223) at (0,7);
		\coordinate (22221) at (-3,7);
		\coordinate (22222) at (-1.5,7);
		\coordinate (22224) at (1.5,7);
		\coordinate (22225) at (3,7);
		\coordinate (222231) at (-1.5,8.5);
		\coordinate (222232) at (0,8.5);
		\coordinate (2222321) at (0,10.5);
		\draw [thick](0) -- (2)--(22)--(222)--(2222)--(22223)--(222232);
		\draw [thick](0) -- (1);
		\draw [thick](0) -- (3);
		\draw [thick](2) -- (21);
		\draw [thick](22) -- (221);
		\draw [thick](22) -- (223);
		\draw [thick](22) -- (224);
		\draw [thick](2222) -- (22221);
		\draw [thick](2222) -- (22222);
		\draw [thick](2222) -- (22224);
		\draw [thick](2222) -- (22225);
		\draw [thick](22223) -- (222231);
		\draw [thick,dotted](222232) -- (2222321);
		\draw (2222)  node{\textcolor{blue}{\large{$\bullet$}}};
		\draw (222232) node{\textcolor{blue}{\large{$\bullet$}}};
		\draw (1)..controls +(-2,1.5) and +(2,1.5)..(1);
		\draw (-2,1.5) node[above]{\scriptsize{MGW}};
		\draw (-2,1.5) node{\scriptsize{$\bp$,$\bq$}};
		\draw (1) node{\textcolor{blue}{\large{$\bullet$}}};
		\draw (3)..controls +(-2,1.5) and +(2,1.5)..(3);
		\draw (2,1.5) node[above]{\scriptsize{MGW}};
		\draw (2,1.5) node{\scriptsize{$\bp$,$\bq$}};
		\draw (3) node{\textcolor{blue}{\large{$\bullet$}}};
		\draw (21)..controls +(-2,1.5) and +(2,1.5)..(21);
		\draw (-2,3) node[above]{\scriptsize{MGW}};
		\draw (-2,3) node{\scriptsize{$\bp$,$\bq$}};
		\draw (221)..controls +(-2,1.5) and +(2,1.5)..(221);
		\draw (-2,4.5) node[above]{\scriptsize{MGW}};
		\draw (-2,4.5) node{\scriptsize{$\bp$,$\bq$}};
		\draw (221) node{\textcolor{blue}{\large{$\bullet$}}};
		\draw (223)..controls +(-2,1.5) and +(2,1.5)..(223);
		\draw (1.5,4.5) node[above]{\scriptsize{MGW}};
		\draw (1.5,4.5) node{\scriptsize{$\bp$,$\bq$}};
		\draw (224)..controls +(-2,1.5) and +(2,1.5)..(224);
		\draw (3,4.5) node[above]{\scriptsize{MGW}};
		\draw (3,4.5) node{\scriptsize{$\bp$,$\bq$}};
		\draw (224) node{\textcolor{blue}{\large{$\bullet$}}};
		\draw (22221)..controls +(-2,1.5) and +(2,1.5)..(22221);
		\draw (-3,7.5) node[above]{\scriptsize{MGW}};
		\draw (-3,7.5) node{\scriptsize{$\bp$,$\bq$}};
		\draw (22222)..controls +(-2,1.5) and +(2,1.5)..(22222);
		\draw (-1.5,7.5) node[above]{\scriptsize{MGW}};
		\draw (-1.5,7.5) node{\scriptsize{$\bp$,$\bq$}};
		\draw (22224)..controls +(-2,1.5) and +(2,1.5)..(22224);
		\draw (1.5,7.5) node[above]{\scriptsize{MGW}};
		\draw (1.5,7.5) node{\scriptsize{$\bp$,$\bq$}};
		\draw (22225)..controls +(-2,1.5) and +(2,1.5)..(22225);
		\draw (3,7.5) node[above]{\scriptsize{MGW}};
		\draw (3,7.5) node{\scriptsize{$\bp$,$\bq$}};
		\draw (22225) node{\textcolor{blue}{\large{$\bullet$}}};
		\draw (222231)..controls +(-2,1.5) and +(2,1.5)..(222231);
		\draw (-1.5,9) node[above]{\scriptsize{MGW}};
		\draw (-1.5,9) node{\scriptsize{$\bp$,$\bq$}};
		\draw (222231) node{\textcolor{blue}{\large{$\bullet$}}};
		\draw (2222321) node[above]{$\infty$};
		\end{tikzpicture}
			\begin{tikzpicture}[scale=0.5,every node/.style={scale=0.5}]
			\coordinate (0) at (0,0);
			\coordinate (1) at (-2,1);
			\coordinate (3) at (2,1);
			\coordinate (2) at (0,1);
			\coordinate (21) at (-2,2.5);
			\coordinate (22) at (0,2.5);
			\coordinate (223) at (1.5,4);
			\coordinate (224) at (3,4);
			\coordinate (221) at (-2,4);
			\coordinate (222) at (0,4);
			\coordinate (2222) at (0,5.5);
			\coordinate (22223) at (0,7);
			\coordinate (22221) at (-3,7);
			\coordinate (22222) at (-1.5,7);
			\coordinate (22224) at (1.5,7);
			\coordinate (22225) at (3,7);
			\draw [thick](0) -- (2)--(22)--(222)--(2222);
			\draw [thick](0) -- (1);
			\draw [thick](0) -- (3);
			\draw [thick](2) -- (21);
			\draw [thick](22) -- (221);
			\draw [thick](22) -- (223);
			\draw [thick](22) -- (224);
			\draw [thick](2222) -- (22221);
			\draw [thick](2222) -- (22222);
			\draw [thick](2222) -- (22224);
			\draw [thick](2222) -- (22225);	
			\draw [thick, dotted](2222) -- (22223);
			\draw [thick, dotted ](2222) -- (-2,7);
			\draw [thick, dotted ](2222) -- (-2.5,7);
			\draw [thick, dotted ](2222) -- (-2.25,7);
			\draw [thick, dotted ](2222) -- (-1.75,7);
			\draw [thick, dotted ](2222) -- (-1,7);
			\draw [thick, dotted ](2222) -- (-1.25,7);
			\draw [thick, dotted ](2222) -- (-0.5,7);
			\draw [thick, dotted ](2222) -- (-0.25,7);
			\draw [thick, dotted ](2222) -- (2,7);
			\draw [thick, dotted ](2222) -- (2.5,7);
			\draw [thick, dotted ](2222) -- (2.25,7);
			\draw [thick, dotted ](2222) -- (1.75,7);
			\draw [thick, dotted ](2222) -- (1,7);
			\draw [thick, dotted ](2222) -- (1.25,7);
			\draw [thick, dotted ](2222) -- (0.5,7);
			\draw [thick, dotted ](2222) -- (0.25,7);
			\draw (2) node{\textcolor{blue}{\large{$\bullet$}}};
			\draw (2222) node{\textcolor{blue}{\large{$\bullet$}}};
			\draw (22221) node{\textcolor{blue}{\large{$\bullet$}}};
			\draw (1)..controls +(-2,1.5) and +(2,1.5)..(1);
			\draw (-2,1.5) node[above]{\scriptsize{MGW}} node{\scriptsize{$\bp$,$\bq$}};
			\draw (3)..controls +(-2,1.5) and +(2,1.5)..(3);
			\draw (2,1.5) node[above]{\scriptsize{MGW}} node{\scriptsize{$\bp$,$\bq$}};
			\draw (3) node{\textcolor{blue}{\large{$\bullet$}}};
			\draw (21)..controls +(-2,1.5) and +(2,1.5)..(21);
			\draw (-2,3) node[above]{\scriptsize{MGW}} node{\scriptsize{$\bp$,$\bq$}};
			\draw (221)..controls +(-2,1.5) and +(2,1.5)..(221);
			\draw (-2,4.5) node[above]{\scriptsize{MGW}} node{\scriptsize{$\bp$,$\bq$}};
			\draw (221) node{\textcolor{blue}{\large{$\bullet$}}};
			\draw (223)..controls +(-2,1.5) and +(2,1.5)..(223);
			\draw (1.5,4.5) node[above]{\scriptsize{MGW}} node{\scriptsize{$\bp$,$\bq$}};
			\draw (224)..controls +(-2,1.5) and +(2,1.5)..(224);
			\draw (3,4.5) node[above]{\scriptsize{MGW}} node{\scriptsize{$\bp$,$\bq$}};
			\draw (224) node{\textcolor{blue}{\large{$\bullet$}}};
			\draw (22221)..controls +(-2,1.5) and +(2,1.5)..(22221);
			\draw (-3,7.5) node[above]{\scriptsize{MGW}} node{\scriptsize{$\bp$,$\bq$}};
			\draw (22222)..controls +(-2,1.5) and +(2,1.5)..(22222);
			\draw (-1.5,7.5) node[above]{\scriptsize{MGW}} node{\scriptsize{$\bp$,$\bq$}};
			\draw (22224)..controls +(-2,1.5) and +(2,1.5)..(22224);
			\draw (1.5,7.5) node[above]{\scriptsize{MGW}} node{\scriptsize{$\bp$,$\bq$}};
			\draw (22225)..controls +(-2,1.5) and +(2,1.5)..(22225);
			\draw (3,7.5) node[above]{\scriptsize{MGW}}node{\scriptsize{$\bp$,$\bq$}};
			\draw (22225) node{\textcolor{blue}{\large{$\bullet$}}};
			\draw (22223) node[above]{$\infty$};
		\end{tikzpicture}

		\centering
		\begin{tikzpicture}[scale=0.6,every node/.style={scale=0.6}]
			\coordinate (0) at (0,0);
			\draw (0)..controls +(-2,1.5) and +(2,1.5)..(0);
			\draw (0,0.5) node[above]{\scriptsize{MGW}} node{\scriptsize{$\bp$,$\bq$}};
			\draw (0.5,1) node[right]{: Marked Galton Watson tree};
			\draw (0.75,0.5) node[right]{with offspring distribution $\bp$};
			\draw (0.75,0) node[right]{and mark function $\bq$};
		\end{tikzpicture}
		
		\caption{Marked Kesten'tree and condensation tree}
	\end{figure}
		We denote by $\bp^{*}:=(\bp^{*}(n)=\nicefrac{n\bp(n)}{\mu})_{ n\in \N}$ the corresponding size-biased distribution of the Kesten's tree. We define an infinite random tree $\tau_K(\bp)$ (the Kesten's tree) whose distribution is described in \cite{abraham_local_2017} Sub-Section 3.2. This tree has one infinite spine and all its nodes have finite degrees. 
	We define an infinite random tree $\tau_C(\bp)$ (the condensation tree) whose described in \cite{abraham_local_2014} Sub-Section 3.2. This tree has exactly one node of infinite degree and no infinite spine. This tree has been considered in \cite{janson_simply_2011} and \cite{jonsson_condensation_2011}. 

	Now we consider the marked limit trees $\tau_K^{*}(\bp,\bq)$ and $\tau_C^{*}(\bp,\bq)$. In the both cases, the mark function stay the same as the original MGW with offspring distribution $\bp$ and mark function $\bq$.
	\subsubsection{Kesten's tree}
	In this subsection, we write $\tau_K^{*}$ for $\tau_K^{*}(\bp,\bq)$ and $\tau_K$ for $\tau_K(\bp)$.
		\begin{prop}\label{mfk}
Let $\tau_K^{*}$ be a marked Kesten's tree associated with a critical offspring distribution $\bp$ and a mark function $\bq$ satisfying \eqref{condp} and  \eqref{condq} respectively. Then we have, for all  marked tree $\bt^{*}\in\T_0^{*}$ and $x\in\mathcal{L}_0(\bt)$ :
			\begin{equation}\label{exptau^*K}
				\p(\tau_K^{*}\in \T(\bt^{*},x))=\dfrac{\p(\tau^{*} = \mathcal{S}^x(\bt^{*}))}{\bp(0)(1-\bq(0))}\E[X\alpha_{X,x}].
			\end{equation}
			Where $X$ is a random variable with distribution $\bp$, and for all $j\in\N$\\ $\alpha_{j,x}:=\bq(j)\eta_x(\bt)+(1-\bq(j))(1-\eta_x(\bt))$.
		\end{prop}
	\begin{proof}
	First, according to Sub-section 2.4 of \cite{abraham_local_2013}, we have  for $t\in \T_0$, $x\in \mathcal{L}_0(\bt)$: 
	\begin{equation}\label{exptauK}
		\p(\tau_K\in\T(\bt,x))=\dfrac{\p(\tau=\bt)}{\bp(0)},
	\end{equation}
		which implies:
		\begin{align*}
			\p(\tau_K^{*}\in \T(\bt^{*},x))	&=\dfrac{\p(\tau=\bt)}{\bp(0)}\prod_{\underset{\eta_u=1}{u\in \bt,~u\neq x}}{\bq(k_u(\bt))}\prod_{\underset{\eta_u=0}{u\in \bt,~u\neq x}}{\left(1-\bq\left(k_u(\bt)\right)\right)}\\
			&\times\sum_{j\ge0}j\bp(j)\alpha_{j,x}\\\\
			&=\dfrac{\p(\tau^{*}=\bt^{*})}{\bp(0)}\left(\dfrac{\eta_x}{\bq(0)}+\dfrac{1-\eta_x}{1-\bq(0)}\right)\E[X\alpha_{X,x}].
		\end{align*}
		And we easily conclude as for $x\in\mathcal{L}_0(\bt)$, we have:
		\[\p(\tau^{*} = \mathcal{S}^x(\bt^{*})) =\p(\tau^{*}=\bt^{*})\left(1-\eta_x+\dfrac{1-\bq(0)}{\bq(0)}\eta_x\right).\]
		
	\end{proof}
	\subsubsection{Condensation tree}
		In this subsection, we write $\tau_C^{*}$ for $\tau_C^{*}(\bp,\bq)$.
	
	If $\mu(\bp)<1$. For all marked tree $\bt^{*}\in \T_0^{*}$, $x\in \bt$, we set:
	\begin{align}\label{dfC}
		&C(\bt^{*},x):=\dfrac{\p(\tau^{*} = \mathcal{S}^x(\bt^{*}))}{\bp(0)(1-\bq(0))}\p_{k_x(\bt)}(\tau^{*}=\F_x(\bt^{*})).
	\end{align}

	\begin{prop}\label{mfc}
Let $\tau_C^{*}$ be a marked condensation tree associated with a sub-critical offspring distribution $\bp$ and a mark function $\bq$ satisfying \eqref{condp} and \eqref{condq} respectively. Then, we have for all $\bt^{*}\in \T_0^{*}$, $x\in \bt$, $k \in \N$:
			\begin{multline}\label{exptau^*C}
				\p\left(\tau^{*}_C\in \T_{+}(\bt^{*},x,k)\right)=C(\bt^{*},x)\left(\left(1-\mu(\bp)\right)\alpha_{\infty,x}\right.\\
				\left.+\E\left[\left(X-k_x(\bt)\right)_+ \alpha_{X,x} \indic_{\{X\geq k\}}\right]\right).
		\end{multline}
		with $C(\bt^*,x)$ defined in \eqref{dfC} and $z_{+}=\max(z,0)$ for $z\in \mathbb{R}$.
	\end{prop}
	\begin{proof}
		\begin{align*}
			&\p\left(\tau^{*}_C\in \T_{+}(\bt^{*},x,k)\right)=\prod_{\underset{u\neq x}{u\in \bt}}\bp(k_u(\bt))\prod_{\underset{\eta_u=1}{u\in \bt,~u\neq x}}{\bq(k_u(\bt))}\prod_{\underset{\eta_u=0}{u\in \bt,~u\neq x}}{\left(1-\bq\left(k_u(\bt)\right)\right)}\\
			&\times \left((1-\mu)\left(\bq(\infty)\eta_x+(1-\bq(\infty))(1-\eta_x)\right)\right.\\
			&\left.+\sum_{j\ge0}\alpha_{k+j,x}\check{\bp}(k+j)\frac{k+j-k_x(\bt)}{k+j}\right).
		\end{align*}
		Indeed, $x$ is a special node. Thereby, if $x$ has an infinity of children we obtain the term at the second line, otherwise we have the probability $\bp(k+j)(k+j)$ to have $k+j$ children, and we must to choose the special node knowing that it cannot belong to the $k_x(\bt)$ children from $\bt$, which give us the term $\dfrac{k+j-k_x(\bt)}{k+j}$. Thus we have:
		\begin{align*}
			\frac{\p\left(\tau^{*}_C\in \T_{+}(\bt^{*},x,k)\right)}{C(\bt^{*},x) }&=(1-\mu)\alpha_{\infty,x}+\sum_{j\ge0}\bp(k+j)\alpha_{k+j,x}(k+j-k_x(\bt))\\
			&=\left(1-\mu(\bp)\right)\alpha_{\infty,x}+\E\left[\left(X-k_x(\bt)\right)_+ \alpha_{X,x} \indic_{\{X\geq k\}}\right].
		\end{align*}
	\end{proof}
	\subsection{Convergence criterions}
	We adapt the convergence criterions seen in \cite{abraham_local_2013} and \cite{abraham_local_2014}.\smallbreak
	Let $(T_n)_{n\in\NN}$ and $T$ be $\T_{\infty}$-valued (resp. $\T$-valued) random variables. We denote by $\textrm{dist}(T)$ the distribution of the random variable $T$ (which is uniquely determined by the sequence of distributions of $r_{h,\infty}(T)$ (resp. $r_{h}(T)$) for every $h\geq 0$), and we denote :
	\[\textrm{dist}(T_n)\underset{n\to+\infty}{\longrightarrow}\textrm{dist}(T),\]
	for the convergence in distribution of the sequence $(T_n)_{n\in\NN}$ to $T$. Notice that this convergence in distribution is equivalent to the finite dimensional convergence in distribution of $(k_u(T_n))_{u\in \U}$ to $(k_u(T))_{u\in \U}$ when $n$ goes to infinity.
	\\
	We deduce from the portmanteau theorem that the sequence $(T_n)_{n\in\NN}$ converges in distribution to $T$ if and only if  for all $h\in \N$, $\bt\in \T_{\infty}^{(h)}$:
	$$\underset{n\to+\infty}{\lim}\p(r_{h,\infty}(T_n)=\bt)=\p(r_{h,\infty}(T)=\bt),$$
	$$\text{ (resp.} \underset{n\to+\infty}{\lim}\p(r_{h}(T_n)=\bt)=\p(r_{h}(T)=\bt)).$$
	From now on, we endow the sets of marked trees, with the ultra-metric distance, for  $\bt^{*},\bt'^{*}\in\T_{\infty}^{*}$ (resp. $\T^{*}$):
	$$d_\infty^{*}(\bt^{*},\bt'^{*}):=2^{-\max\{h\in\N,r_{h,\infty}^{*}(\bt^{*})=r_{h,\infty}^{*}(\bt'^{*})\}}$$
	$$\left(\text{ resp. } d^{*}(\bt^{*},\bt'^{*}):=2^{-\max\{h\in\N,r_{h}^{*}(\bt^{*})=r_{h}^{*}(\bt'^{*})\}} \right).$$
Consider the closed ball $B_\infty^{*}(\bt^{*},2^{-h})=\{\bt'^{*}\in\T_{\infty}^{*},d_\infty^{*}(\bt^{*},\bt'^{*})\leq 2^{-h}\}$ for some $\bt^{*}\in\T_{\infty}^{*}$ and $h\in\N$, we have $B_\infty^{*}(\bt^{*},2^{-h})={r_{h,\infty}^{*}}^{-1}(r_{h,\infty}^{*}(\bt^{*}))$. Since the distance is ultra-metric, the closed balls are open and the open balls are closed, and the intersection of two balls is either empty or one of them.
	We deduce that the family $(r_{h,\infty}^{*}(\bt^{*}),\bt^{*}\in {\T_{\infty}^{(h)}}^{*}, h \in\N)$ is a $\pi$-system, and Theorem 2.3 in \cite{billingsley_convergence_1999} implies that this family is convergence determining for the convergence in distribution.
	\begin{prop}
		For all $\bt^{*}\in\T_\infty^{*}$, $x\in \bt$ and $k\in\N$ we have $\T_{+}(\bt^{*},x,k)$ is closed and open.
	\end{prop}
	\begin{proof}
		
		\begin{itemize}
			\item Let $(\bt^{*}\circledast_xs_n^{*},~s_n^{*}\in\T_{\infty}^{*},~k_x(\bt\circledast_xs_n) \geq k)_{n\in\NN}$ be a sequence of $\T_{+}(\bt^{*},x,k)$ that converges to $z\in(\T_{\infty}^{*},d_\infty^{*})$.\\
			For all $n\in\NN$, $\mathcal{S}^x(\bt^{*}\circledast_xs_n^{*})=\mathcal{S}^x(\bt^{*})$, the forest that contains only the $k_x(\bt)$ first trees above $x$ on $\bt^{*}\circledast_xs_n^{*}$ equals $\F_x(\bt^{*})$ and the marking on $x$ is the same. Thereby, $z=\bt^{*}\circledast_xs^{*}$, with $s^{*}$ the limit of the sequence $(s_n^{*},n\in\NN)$. Thus $z\in\T_{+}(\bt^{*},x,k)$.\\
			$\T_{+}(\bt^{*},x,k)$ is closed.
			
			\item Let $s^{*}\in\T_{+}(\bt^{*},x,k)$ and $c_k=\max(k,H_\infty(\bt)-1)$. We have, according to the definition of $s^{*}$:
			\begin{align*}
				B_\infty^{*}(s^{*},2^{-c_k})&={r_{c_k,\infty}^{*}}^{-1}(r_{c_k,\infty}^{*}(s^{*}))\\
				&={r_{c_k,\infty}^{*}}^{-1}\bigl(r_{c_k,\infty}(s), (\eta_u^{c_k})_{u\in r_{c_k,\infty}(s)}\bigr)\subset \T_{+}(\bt^{*},x,k).
			\end{align*}
			 Thus, $\T_{+}(\bt^{*},x,k)$ is open.
		\end{itemize}
	\end{proof}
	We have the following two lemmas.\\
	According to lemma 2.1 in \cite{abraham_local_2013}:
	\begin{lem}\label{critconvT(t,x)}
		Let $(T_n^{*})_{n\in\NN}$ and $T^{*}$ be $\T^{*}$-valued random variables which belong a.s. to $\T_0^{*}\cup\T_1^{*}$. The sequence $(T_n^{*})_{n\in\NN}$ converges in distribution to $T^{*}$ if and only if for every $\bt^{*}\in \T_0^{*}\cap\T^{*}$ and every $x\in\mathcal{L}_0(\bt)$, we have:
		\begin{align}\label{critconvT}
			\underset{n\to+\infty}{\lim}& \p(T_n^{*}\in \T(\bt^{*},x))=\p(T^{*}\in\T(\bt^{*},x))\nonumber\\
			& \text{ and } \underset{n\to+\infty}{\lim} \p(T_n^{*}=\bt^{*})=\p(T^{*}=\bt^{*}).
		\end{align}
	\end{lem}
The proof is similar as the one of Abraham and Delmas in \cite{abraham_local_2013}.
	According to Lemma 2.2 in \cite{abraham_local_2014}:
	\begin{lem}\label{critconvT(t,x,k)}
		Let $(T_n^{*})_{n\in\NN}$ and $T^{*}$ be $\T_\infty^{*}$-valued random variables which belong a.s. to $\T_0^{*}\cup\T_2^{*}$. The sequence $(T_n^{*})_{n\in\NN}$ converges in distribution to $T^{*}$ if and only if for every $\bt^{*}\in \T_0^{*}$, $x\in \bt$ and $k\in\N$, we have:
		\begin{align}\label{critconvTinf}
			\underset{n\to+\infty}{\lim}& \p(T_n^{*}\in \T_+(\bt^{*},x,k))=\p(T^{*}\in\T_+(\bt^{*},x,k))\nonumber\\
			 &\text{ and } \underset{n\to+\infty}{\lim} \p(T_n^{*}=\bt^{*})=\p(T^{*}=\bt^{*}).
		\end{align}
	\end{lem}
The proof is similar as the one of Abraham and Delmas in \cite{abraham_local_2014}, see the survey.

	\section{Conditioning on the marks}\label{markcond}
	\subsection{Change of offspring distribution and conditioning}\label{resultcond}
	Let $\bp$ be a distribution on $\N$ satisfying \eqref{condp}, $X$ be a random variable with distribution $\bp$ and $\bq$ let be a mark function satisfying \eqref{condq}. Let $\tau^{*}$ be a MGW with offspring distribution $\bp$ and mark function $\bq$.
	\medbreak
	For every $\theta >0$,
	\begin{equation}
		\forall k \geq 0, ~\bp_\theta(k):=\theta^{k-1} \bp(k) \left(c_\theta \bq(k)+1-\bq(k)\right),
	\end{equation}
	where the normalizing constant $c_\theta$ is given by:
	\begin{equation}
		c_\theta:=\dfrac{1-\E\left[\theta^{X-1}(1-\bq(X))\right]}{\E\left[\theta^{X-1}\bq(X)\right]}.
	\end{equation}
	We denote by $I$ the set of $\theta$ such that $\bp_\theta$ defines a probability distribution on $\N$.
	We have $\theta\in I$ if and only if $\theta >0$ and:
	\begin{align}\label{condtheta}
			g(\theta)=\E[\theta^X]<+\infty\text{ and }  l(\theta)=\E[\theta^X(1-\bq(X))] < \theta.
	\end{align}
	Moreover, as $\bp$ satisfies \eqref{condp}, we can prove that for all $\theta\in I$, $\bp_\theta$ is an offspring distribution satisfying the first two conditions of \eqref{condp}. Indeed:
		\begin{itemize}
			\item [$\bullet$] as $c_\theta>0$, we have $ \bp_\theta(0)=\frac{1}{\theta}\bp(0)\left(c_\theta \bq(0)+1-\bq(0)\right) >0$;
			\item [$\bullet$] 
			there exists $k>1$ such that $\bp(k)>0$. Thereby, $\bp_\theta(k)>0$ and $\bp_\theta(0)+\bp_\theta(1)<1$.
	\end{itemize}
	The generating function $g_\theta$ of $\bp_\theta$ is given by:
	\begin{equation}\label{gtheta}
		g_\theta(z):=\frac{1}{\theta}\E\left[(z\theta)^X\left(c_\theta \bq(X)+1-\bq(X)\right)\right].
	\end{equation}
		We define also a modified mark function $\bq_{\theta}$ on $\mathbb N$ by:
	\begin{equation}\label{condqtheta}
			\bq_\theta(k)= \dfrac{c_\theta \bq(k)}{c_\theta \bq(k)+1-\bq(k)},
	\end{equation}
	and note that $\bq_\theta$ satisfies \eqref{condq}, and for all $k\in\mathbb N$:
	\begin{align}
		&	\bp_\theta(k)\bq_\theta(k)= c_\theta \theta^{k-1}\bp(k)\bq(k),\label{condqthetabis}\\
	&\bp_\theta(k)(1-\bq_\theta(k)) =\theta^{k-1}\bp(k)(1-\bq(k)).\label{condqthetater}
	\end{align}
Moreover, we have, for all $k\in\N$:
\[\bp_\theta(k)\bq_{\theta}(k)>0\Leftrightarrow\bp(k)\bq(k)>0.\]
Then, throughout the rest of the article, $\tau^*(\bp_\theta,\bq_\theta)$ (noted $\tau_\theta^*$) denotes a MGW with offspring distribution $\bp_\theta$ and mark function $\bq_\theta$.\\	
The introduction of these modified distribution is motivated by the following proposition:
	\begin{prop}\label{loicondid}
		Let $\tau^{*}$ a sub-critical MGW with offspring distribution $\bp$ satisfying \eqref{condp} and $\bq$ its mark function satisfying \eqref{condq}. For every $\theta\in I$, 
		the conditional distributions of $\tau^{*}$ given $\{M(\tau^*)=n\}$ and of $\tau_\theta^*$ given $\{ M(\tau_\theta^*)=n\}$ are the same.
	\end{prop}
	\begin{proof}
		For all $\bt^{*}\in\T_0^{*}$, using \eqref{condqthetabis}, \eqref{condqthetater} and \eqref{sumku}:
			\begin{align*}
			\p(\tau_\theta^*=\bt^{*})&=\prod_{\underset{\eta_u=1}{u\in \bt,}} \bp_\theta (k_u(\bt))\bq_\theta(k_u(\bt))\prod_{\underset{\eta_u=0}{u\in \bt,}} \bp_\theta (k_u(\bt))(1-\bq_\theta(k_u(\bt)))\\
			&=\prod_{{u\in \bt}} \theta^{k_u(\bt)-1}\bp(k_u(\bt))\prod_{\underset{\eta_u=1}{u\in \bt,}}c_\theta \bq(k_u(\bt))\prod_{\underset{\eta_u=0}{u\in \bt,}}(1-\bq(k_u(\bt)))\\
			&=\theta^{-1}{c_\theta}^{M(\bt^*)}\p(\tau^{*}=\bt^{*}).
		\end{align*}
		
		We deduce that
		\begin{align*}
			\p( M(\tau_\theta^*)=n)&=\sum_{\underset{M(\bt^*)=n}{\bt^{*}\in \T_0^{*}}} \p(\tau_\theta^*=\bt^{*})=\sum_{\underset{M(\bt^*)=n}{\bt^{*}\in \T_0^{*}}}  \theta^{-1}{c_\theta}^{M(\bt^*)}\p(\tau^{*}=\bt^{*})
			\\&=\theta^{-1}{c_\theta}^{n}\sum_{\underset{M(\bt^*)=n}{\bt^{*}\in \T_0^{*}}}  \p(\tau^{*}=\bt^{*})
			=\theta^{-1}{c_\theta}^{n}\p(M(\tau^*)=n).
		\end{align*}
		
		Finally, for every marked tree $\bt^{*}\in\T_0^{*}$, such that $M(\bt^*)=n$, we have:
		\begin{align*}
			\p(\tau_\theta^*=\bt^{*}| M(\tau_\theta^*)=n)&=\dfrac{\p(\tau_\theta^*=\bt^{*})}{\p( M(\tau_\theta^*)=n)}\\
			&=\dfrac{\p(\tau^{*}=\bt^{*})}{\p(M(\tau^*)=n)} = \p(\tau^{*}=\bt^{*}|M(\tau^*)=n).
		\end{align*}
	\end{proof}

	\subsection{Generic and non-generic distributions}\label{gennongencase}
	Here, we consider sub-critical or  critical $\bp$ and note that the mean of $\bp_\theta$ is equal to: 
	\begin{equation}\label{muptheta}
		\mu(\bp_\theta)=\E\left[X\theta^{X-1}\left(c_\theta \bq(X)+1-\bq(X)\right)\right].
	\end{equation}
	 If there exists $\theta\in I$ such that:
		\begin{equation}\label{mupthetaeq1}
			\mu(\bp_\theta)=1,
		\end{equation}
		$\bp$ is said generic and non-generic otherwise. \\
		The aim of this section is to find criteria for determining whether $\bp$ is generic or not.
			\begin{rqqq}
		The results of this section have been obtained in the particular case $\bq(k)=\indic_{k\in A}$ with $A\subset\N$ in \cite{abraham_local_2014}.
	\end{rqqq}	
	\begin{prop}
		The following composition rule holds:
		for all $\theta \in I$ and for all $\ell$ such that $\theta\ell \in I$, we  have :
		\begin{equation}\label{comprul}
			\bp_{\theta \ell} = (\bp_\theta)_\ell.
		\end{equation}
	\end{prop}
	\begin{proof}
		Let $\theta \in I$ and $\ell\in\mathbb{R}$ such that $\theta \ell \in I$. We consider $X$ a random variable with distribution $\bp$ and its mark function $\bq$, and $\bar{X}$ a random variable with distribution $\bp_\theta$ and its mark function $\bq_\theta$.
		\smallbreak
		For all $k\in\N$ we have:
		\[\bp_{\theta \ell}(k)=(\theta \ell)^{k-1}\bp(k)\left(c_{\theta \ell}\bq(k)+1-\bq(k)\right),~~c_{\theta \ell}=\dfrac{1-\E\left[(\theta \ell)^{X-1}(1-\bq(X))\right]}{\E\left[(\theta \ell)^{X-1}\bq(X)\right]},\]
		and 
		\[(\bp_\theta)_\ell(k)=\ell^{k-1}\bp_\theta (k) \left(\tilde{c}_\ell \bq_\theta(k)+1-\bq_\theta(k)\right),~~ \tilde{c}_{ \ell}=\dfrac{1-\E\left[ \ell^{\bar{X}-1}(1-\bq_\theta(\bar{X}))\right]}{\E\left[ \ell^{\bar{X}-1}\bq_\theta(\bar{X})\right]}.\]
		Using again \eqref{condqthetabis} and  \eqref{condqthetabis}, we have:
		\begin{align*}
			(\bp_\theta)_\ell(k)
			&=(\theta \ell)^{k-1}\bp(k)\left(\tilde{c}_\ell c_{\theta}\bq(k)+1-\bq(k)\right).
		\end{align*}
		We have our result if $\tilde{c}_\ell c_{\theta}=c_{\theta \ell}$, which is true as:
		\begin{align*}
			\E\left[ \ell^{\bar{X}-1}(1-\bq_\theta(\bar{X}))\right]
			&=\overset{+\infty}{\underset{k=0}{\sum}}(\theta \ell)^{k-1}\bp(k)(1-\bq(k))=\E\left[(\theta \ell)^{X-1}(1-\bq(X))\right],\\
			\E\left[ \ell^{\bar{X}-1}\bq_\theta(\bar{X})\right]
			& =c_\theta \overset{+\infty}{\underset{k=0}{\sum}}(\theta \ell)^{k-1}\bp(k)\bq(k)=c_\theta\E\left[(\theta \ell)^{X-1}\bq(X)\right].
		\end{align*}
	\end{proof}
		Notice that $I$ is an interval of $(0,+\infty)$ which contains 1 and let:
	\begin{equation}\label{theta*}
		\theta_s:=\sup{I} \in [1,\rho(\bp)].
	\end{equation}
	with $\rho(\bp)$ the radius of convergence of $g$, the generating function of \eqref{condtheta} and we define the function $l$ by:
	\begin{equation}\label{functionl}
		l(\theta)=\E\left[\theta^X(1-\bq(X))\right],
	\end{equation}
	and note that its radius of convergence $\rho_l(\bp,\bq)$ satisfies $\rho(\bp)\leq\rho_l(\bp,\bq)$.
	\begin{lem}\label{derivmuptpos}
		Let $\bp$ be a distribution on $\N$ satisfying the two first conditions of \eqref{condp} and $\bq$ satisfying \eqref{condq}. For $\theta \in \mathring I$, if $\mu(\bp_\theta)\leq 1$, we have:
		\begin{equation*}
			\dfrac{\mathrm{d}}{\mathrm{d}\theta}\mu(\bp_\theta)>0 .
		\end{equation*}
	\end{lem}
	\begin{proof}
		Since $\bp$ satisfies the two first conditions of \eqref{condp}, then $\bp_\theta$ satisfies the two first conditions of \eqref{condp} for all $\theta \in I$ such that $\theta<\theta_s$ (see Sub-section \ref{resultcond}). Thanks to the composition rule \eqref{comprul}, it is enough to prove that $\dfrac{\text{d}}{\text{d}\theta}\mu(\bp_\theta)>0$ at $\theta=1$ if $\mu(\bp)\leq 1$, with $\bp$ satisfying \eqref{condp} and $\rho(\bp)>1$. 
		\medbreak
	
		For $\kappa\in I$, we have:
		\begin{align*}
			\dfrac{\text{d}}{\text{d}\kappa}\mu(\bp_\kappa)&=\lim_{h\rightarrow0}\dfrac{\mu(\bp_{\kappa+h})-\mu(\bp_\kappa)}{h}=\lim_{h\rightarrow0}\dfrac{\mu(\bp_{\kappa(1+\frac{h}{\kappa})})-\mu(\bp_\kappa)}{h}\\
			&=\lim_{h\rightarrow0}\dfrac{\mu((\bp_\kappa)_{(1+\frac{h}{\kappa})})-\mu(\bp_\kappa)}{h}=\kappa^{-1}\left(\dfrac{\text{d}}{\text{d}\eta}\mu\left((\bp_\kappa)_\eta\right)\right)_{\mid_{\eta=1}}.
		\end{align*}	
	Therefore, we assume that if $\mu(\bp)\leq 1$, then:
	\begin{equation}\label{comprul1}
	\left(\dfrac{\text{d}}{\text{d}\theta}\mu(\bp_\theta)\right)_{\mid_{\theta=1}}>0. 
\end{equation}
		We apply the result in \eqref{comprul1}, with $\theta=\eta$ and $\bp=\bp_\kappa$.
		\smallbreak
		We recall that, $c_\theta=\dfrac{1-\E\left[\theta^{X-1}(1-\bq(X))\right]}{\E\left[\theta^{X-1}\bq(X)\right]}$.
		\smallbreak
		Let $\theta \in I$. To make the calculations easier, we write the following formula:
		\begin{align*}
		\mu(\bp_\theta)-\E\left[X\right]
		&=:\frac{\E\left[X\theta^X\bq(X)\right]d(\theta)+\E\left[\theta^X\bq(X)\right]h(\theta)}{\E\left[\theta^X \bq(X)\right]}=:\dfrac{f(\theta)}{\E\left[\theta^X \bq(X)\right]},
		\end{align*}
		where $d(\theta)=1-\E\left[\theta^{X-1}\right]$ and $h(\theta)=\E\left[X\theta^{X-1}\right]-\E\left[X\right]$. \\
		Moreover, we have $f(1)=d(1)=h(1)=0$ and the functions $f,d$ and $h$ are of class $\mathcal{C}^{\infty}$ on $[0,\rho(\bp))$ (we recall that $\rho(\bp)<\rho_l(\bp,\bq)$). Notice that 
		\begin{align*}
		\left(\dfrac{\text{d}}{\text{d}\theta}\mu(\bp_\theta)\right)_{\mid_{\theta=1}}&=\dfrac{f'(1)}{\E\left[\bq(X)\right]}=\frac{\E\left[X\bq(X)\right]d'(1)+\E\left[\bq(X)\right]h'(1)}{\E\left[\bq(X)\right]}\\
		&=\frac{\E\left[X\bq(X)\right]}{\E\left[\bq(X)\right]}(1-\E\left[X\right])+\E\left[X(X-1)\right].
		\end{align*}
		Since  $\E\left[X\right]\leq 1$ and $\E[X(X-1)]>0$ by assumption \eqref{condp}, we deduce that $f'(1)>0$, which concludes our demonstration.
	\end{proof}
	\begin{lem}\label{exthetacrit}
		Let $\bp$ be a sub-critical distribution on $\N$ satisfying \eqref{condp} and $\bq$ the associated mark function satisfying \eqref{condq}. Equation \eqref{mupthetaeq1}
		\begin{equation*}
			\mu(\bp_\theta)=1
		\end{equation*}
		has at most one solution in $I$, denoted by $\theta_c$ in case of existence. If there is no solution to \eqref{mupthetaeq1}, then we have $\theta_s$ belongs to $I$ and $\mu(\bp_{\theta_s})<1$. 
	\end{lem} 
	
	\begin{proof} We mimic the proof of Lemma 5.2 in \cite{abraham_local_2014}.
		
	Notice that $\gamma$ is continuous over $I$, as $X\in L^1$, and $\mathcal{C}^1$ on $\mathring I$.
	\\
		If $\mu(\bp)\leq1$, we satisfy Lemma \ref{derivmuptpos} and thus \eqref{mupthetaeq1} has at most one solution. Indeed, if it exists another $\theta$ satisfying \eqref{mupthetaeq1}. Then for all $\varepsilon>0$ and $\nu\in ]\theta-\varepsilon,\theta+\varepsilon[$, $\dfrac{\mathrm{d}}{\mathrm{d}\nu}\mu(\bp_\nu)<0$ and it is a contradiction with Lemma \ref{derivmuptpos}.\\
		First, if there is no solution to \eqref{mupthetaeq1}, this implies that $\mu(\bp)<1$ and thus:
		\begin{equation}\label{condmupthetaeq1nosol}
			\mu(\bp_\theta)<1\text{ for all }\theta \in I.
		\end{equation}
		Now we want to prove that $\theta_s\in I$ and  we only need to consider the case $\theta_s>1$. 
Since $\theta_s\leq \rho(\bp)$, we have $\rho(\bp)>1$. Since $\mu(\bp)<1$, the interval $J=\{\theta;~g(\theta)<\theta\}$ is non-empty, $1\in J$. On $J\cap I$, we deduce from \eqref{ctheta} that $c_\theta >1$ and then from \eqref{muptheta} that $\mu(\bp_\theta)>g'(\theta)$.  Thus we have $g'(\theta)<1$ on $J\cap I$, that implies $\theta \mapsto l(\theta)-\theta$ is decreasing and that $I\cap (1,+\infty) $ is a subset of $\bar{J}$, since for all $\theta\ge1$, $l(\theta)-\theta \le l(1)-1<0$. The properties of $g$ imply that $\bar{J}=\{\theta;~g(\theta)\leq\theta\}$.
\\
This clearly implies that $g(\theta_s)\leq\theta_s$, since $0<l(\theta_s)<g(\theta_s)$, then $l(\theta_s)<\theta_s$. 
		 Thereby, \eqref{condtheta} holds for $\theta_s$ that is $\theta_s \in I$. Then conclude using \eqref{condmupthetaeq1nosol}.
	\end{proof}
	\begin{df}\label{dfgen}
		Let $\bp$ be a sub-critical distribution on $\N$ satisfying \eqref{condp} and $\bq$ a mark function satisfying \eqref{condq}. If \eqref{mupthetaeq1} has a (unique) solution in $I$, then $\bp$ is called generic. If \eqref{mupthetaeq1} has no solution, then $\bp$ is called non-generic.
	\end{df}
	In the next lemma we write $\rho$ for $\rho(\bp)$.
	\begin{lem}\label{proprnongencase}
		Let $\bp$ be a distribution on $\N$ satisfying \eqref{condp}, such that $\mu(\bp)<1$ and $\bq$ a mark function satisfying \eqref{condq}.
		\begin{enumerate}
			\item  If $\rho=1$, then $\bp$ is non-generic.
			\item  If $\rho=+\infty$ or ($1<\rho < +\infty$ and $l'(\rho)\geq 1$), then $\bp$ is generic.
			\item  If $1<\rho<+\infty$ and $l'(\rho)<1$, then,
			\\$\bp$ is non-generic if, and only if:
			\begin{equation}\label{condnongensup}
				G(\theta_s)<\dfrac{1- l'(\theta_s)}{\theta_s - l(\theta_s)},~ \text{with}~G(s):=\dfrac{\E[Xs^{X-1}\bq(X)]}{\E[s^{X}\bq(X)]}.
			\end{equation}
		Moreover, if $ g'(\rho)<1$, we have  $\theta_s=\rho$.
		\end{enumerate}
	\end{lem}
	\begin{proof}
		Let $\theta \in I$, notice that:
		\begin{equation}\label{expmuptheta-1}
			\mu(\bp_\theta)-1 = G(\theta)(\theta-l(\theta)) -\left(1-l'(\theta)\right).
		\end{equation}
		\begin{enumerate}
			\item If $\rho=1$ , we have $\theta_s=1\in I$, and $\mu(\bp_1)=\mu(\bp)<1$. Thus, $\bp$ is non-generic.
			\item If $(1<\rho<+\infty~\text{and}~l'(\rho)\geq 1)$ or ($\rho = +\infty$ and there exists $k\in\NN, k\neq 1$, such that {$\bp(k)(1-\bq(k))>0$} ) , since $l$ is convex, $l'$ is non decreasing and then there exists $m\in (1,\rho)$ such that $l'(m)=1$. \\
			Thus, we have $l(1)\ge l(m)+l'(m)(1-m)$, since $l(1)=\E\left[1-\bq(X)\right]$ and $\E\left[\bq(X)\right]>0$ \eqref{condq}, we have $l(m)<m$. Since $1<m<\rho$, $g(m)<+\infty$ which implies that $m$ satisfies \eqref{condtheta}. \\
			This implies, thanks to \eqref{expmuptheta-1}, that $\mu(\bp_m)>1$.
			\\	
			If $\rho = +\infty$ and for all $k\in\NN, k\neq 1$, $\bp(k)(1-\bq(k))=0$. For all $\theta\in I$, 
			\[\mu(\bp_\theta)\ge1\Leftrightarrow (G(\theta)\theta-1)(1-l'(\theta))+G(\theta)(\bp(0)(1-\bq(0)))\ge0,\]
			with $1-l'(\theta)=1-\bp(1)(1-\bq(1))>0$ and $G(\theta)>0$.
			Furthermore, 
			\begin{align*}
				G(\theta)\theta \ge \dfrac{\underset{k>1}{\sum}\theta^k\bp(k)+\theta\bp(1)(1-\bq(1))}{\bp(0)(1-\bq(0))+\underset{k>1}{\sum}\theta^k\bp(k)+\theta\bp(1)(1-\bq(1))},
			\end{align*}
			since $\rho =\infty$, the right term tends to $1$ when $\theta$ tends to $+\infty$.
			Then there exists $m\le \theta_s$ such that $\bp_m$ is critical. \\
			Therefore, $\bp$ is generic, in any cases. 
			\item If $1<\rho<+\infty$ and $l'(\rho)<1$, 
			 $\bp$ is non-generic if $\mu(\bp_{\theta_s})<1$, that is:
			$$G(\theta_s)<\dfrac{1- l'(\theta_s)}{\theta_s - l(\theta_s)}.$$
			If  $g'(\rho)<1$, we have $l(\rho)<g(\rho)<\rho$. Thereby, $\rho$ satisfies \eqref{condtheta} and we have $\theta_s=\rho\in I$.
		\end{enumerate}
	\end{proof}
	\subsection{Marked Tree Decomposition}\label{Model}
	To prove our results, we need to construct a model adapted to the one introduced by Minami \cite{minami_number_2005}, Rizzolo \cite{rizzolo_scaling_2015} or Abraham and Delmas \cite{abraham_local_2014}. 
	\smallbreak
	In other words, let a marked tree $\bt^{*}\in \T_0^{*}$. We define a map $\phi$ from $ \{u\in \bt,\eta_u=1\}$ into the set $\underset{n\geq 1}{\bigcup}\T_0^n$  of forests of finite trees as follows.\\
	To begin with, for $u\in \bt$ we define $\mathcal{S}^\bq_u(\bt^*)$ the sub-tree rooted at $u$ with no marked 
	descendants by 
	$$\mathcal{S}^\bq_u(\bt^*):=\{w\in u\mathcal{S}_u(\bt),~A_w\cap A_u^c\cap\{v\in \bt,\eta_v=1\}=\emptyset\}.$$
	We also define $\mathcal{C}_u^\bq(\bt^*)$, the set of leaves of $\mathcal{S}^\bq_u(\bt^*)$ that belong to $\{u\in \bt,\eta_u=1\}$, in other words 
$$\mathcal{C}_u^\bq(\bt^*):=\{v\in\mathcal{S}^\bq_u(\bt),~k_v(\bt)=0,\eta_v=1\}.$$
	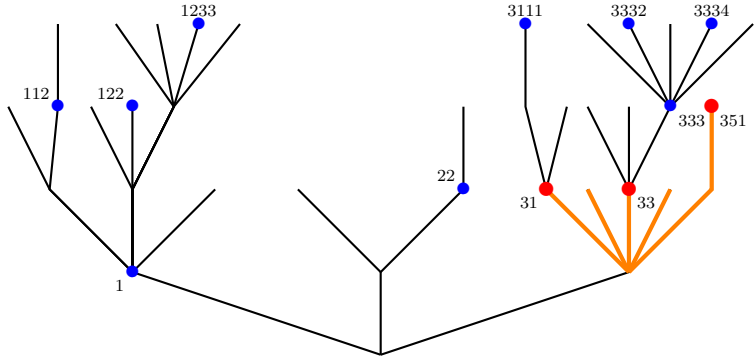
\begin{figure}[h!]
		\centering
		\begin{tikzpicture}[scale=1.1,every node/.style={scale=0.8}]
			\coordinate (0) at (0,0);
			\coordinate (1) at (-3,1);
			\coordinate (2) at (0,1);
			\coordinate (3) at (3,1);
			\coordinate (11) at (-4,2);
			\coordinate (12) at (-3,2);
			\coordinate (13) at (-2,2);
			\coordinate (111) at (-4.5,3);
			\coordinate (112) at (-3.9,3);
			\coordinate (121) at (-3.5,3);
			\coordinate (122) at (-3,3);
			\coordinate (123) at (-2.5,3);
			\coordinate (1121) at (-3.9,4);
			\coordinate (1231) at (-3.2,4);
			\coordinate (1232) at (-2.7,4);
			\coordinate (1233) at (-2.2,4);
			\coordinate (1234) at (-1.7,4);
			\coordinate (21) at (-1,2);
			\coordinate (22) at (1,2);
			\coordinate (221) at (1,3);
			\coordinate (31) at (2,2);
			\coordinate (32) at (2.5,2);
			\coordinate (33) at (3,2);
			\coordinate (34) at (3.5,2);
			\coordinate (35) at (4,2);
			\coordinate (311) at (1.75,3);
			\coordinate (312) at (2.25,3);
			\coordinate (3111) at (1.75,4);
			\coordinate (331) at (2.5,3);
			\coordinate (332) at (3,3);
			\coordinate (333) at (3.5,3);
			\coordinate (351) at (4,3);
			\coordinate (3331) at (2.5,4);
			\coordinate (3332) at (3,4);
			\coordinate (3333) at (3.5,4);
			\coordinate (3334) at (4,4);
			\coordinate (3335) at (4.5,4);
			\draw [thick](0) -- (2) -- (22)--(221);
			\draw [thick](0)-- (3) -- (33)--(333) --(3333);
			\draw [thick](0) -- (1)-- (11)--(112)--(1121);
			\draw [thick](2) -- (21);
			\draw [thick](1) -- (13);
			\draw [thick](1) -- (11)-- (111);
			\draw [thick](1) -- (12)-- (121);
			\draw [thick](1) -- (12)-- (122);
			\draw [thick](1) -- (12)-- (123)--(1231);
			\draw [thick](1) -- (12)-- (123)--(1232);
			\draw [thick](1) -- (12)-- (123)--(1233);
			\draw [thick](1) -- (12)-- (123)--(1234);
			\draw [thick](3) -- (32);
			\draw [thick](3) -- (31)-- (312);
			\draw [thick](3) -- (31)-- (311)--(3111);
			\draw [thick](3) -- (33)-- (331);
			\draw [thick](33)-- (332);
			\draw [thick](333)--(3331);
			\draw [thick](333)--(3332);
			\draw [thick](333)--(3334);
			\draw [thick](333)--(3335);
			\draw [thick](3) -- (34);
			\draw [ultra thick, orange](3) -- (35)-- (351);
			\draw [ultra thick, orange](3)-- (31);
			\draw [ultra thick, orange](3)-- (32);
			\draw [ultra thick, orange](3)-- (33);
			\draw [ultra thick, orange](3)-- (34);
			\draw (1)node{\textcolor{blue}{\Large{$\bullet$}}}node[below left]{\scriptsize{$1$}} ;
			\draw (112)node{\textcolor{blue}{\Large{$\bullet$}}}node[above left]{\scriptsize{$112$}};
			\draw (122)node{\textcolor{blue}{\Large{$\bullet$}}}node[above left]{\scriptsize{$122$}};
			\draw (1233)node{\textcolor{blue}{\Large{$\bullet$}}}node[ above]{\scriptsize{$1233$}};
			\draw (22)node{\textcolor{blue}{\Large{$\bullet$}}}node[above left]{\scriptsize{$22$}} ;
			\draw (31)node{\textcolor{red}{\LARGE{$\bullet$}}}node[below left]{\scriptsize{$31$}};
			\draw (3111)node{\textcolor{blue}{\Large{$\bullet$}}}node[above ]{\scriptsize{$3111$}};
			\draw (33)node{\textcolor{red}{\LARGE{$\bullet$}}}node[ below right]{\scriptsize{$33$}};
			\draw (333)node{\textcolor{blue}{\Large{$\bullet$}}}node[below right]{\scriptsize{$333$}} ;
			\draw (3332)node{\textcolor{blue}{\Large{$\bullet$}}}node[above]{\scriptsize{$3332$}};
			\draw (3334)node{\textcolor{blue}{\Large{$\bullet$}}}node[above ]{\scriptsize{$3334$}};
			\draw (351)node{\textcolor{red}{\LARGE{$\bullet$}}}node[ below right]{\scriptsize{$351$}};
		\end{tikzpicture}
		\caption{The sub-tree $\mathcal{S}^\bq_3(\bt^*)$ in bold and orange, and the elements of $\mathcal{C}_3^\bq(\bt^*)$ in red.}
	\end{figure}
	\\
	We set 
	$$\tilde{\mathcal{S}}^\bq_\emptyset(\bt^*) :=
	\left\{
	\begin{array}{ll}
		\mathcal{S}^\bq_\emptyset(\bt^*),& \text{if}~\eta_\emptyset=0\\
		\{\emptyset\}, & \text{if} ~\eta_\emptyset=1\\
	\end{array}
	\right. $$
	and we set  $\tilde{\mathcal{C}}_\emptyset^\bq(\bt^*)$, the set of leaves of $\tilde{\mathcal{S}}^\bq_\emptyset(\bt^*)$ that belong to $\{u\in \bt,\eta_u=1\}$.
	\smallbreak
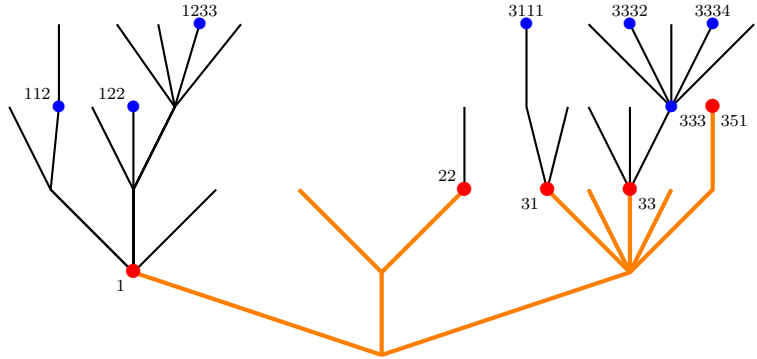
\begin{figure}[h!]
	\centering
	\begin{tikzpicture}[scale=1.1,every node/.style={scale=0.8}]
		\coordinate (0) at (0,0);
		\coordinate (1) at (-3,1);
		\coordinate (2) at (0,1);
		\coordinate (3) at (3,1);
		\coordinate (11) at (-4,2);
		\coordinate (12) at (-3,2);
		\coordinate (13) at (-2,2);
		\coordinate (111) at (-4.5,3);
		\coordinate (112) at (-3.9,3);
		\coordinate (121) at (-3.5,3);
		\coordinate (122) at (-3,3);
		\coordinate (123) at (-2.5,3);
		\coordinate (1121) at (-3.9,4);
		\coordinate (1231) at (-3.2,4);
		\coordinate (1232) at (-2.7,4);
		\coordinate (1233) at (-2.2,4);
		\coordinate (1234) at (-1.7,4);
		\coordinate (21) at (-1,2);
		\coordinate (22) at (1,2);
		\coordinate (221) at (1,3);
		\coordinate (31) at (2,2);
		\coordinate (32) at (2.5,2);
		\coordinate (33) at (3,2);
		\coordinate (34) at (3.5,2);
		\coordinate (35) at (4,2);
		\coordinate (311) at (1.75,3);
		\coordinate (312) at (2.25,3);
		\coordinate (3111) at (1.75,4);
		\coordinate (331) at (2.5,3);
		\coordinate (332) at (3,3);
		\coordinate (333) at (3.5,3);
		\coordinate (351) at (4,3);
		\coordinate (3331) at (2.5,4);
		\coordinate (3332) at (3,4);
		\coordinate (3333) at (3.5,4);
		\coordinate (3334) at (4,4);
		\coordinate (3335) at (4.5,4);
		\draw [thick](0) -- (2) -- (22)--(221);
		\draw [thick](0)-- (3) -- (33)--(333) --(3333);
		\draw [thick](0) -- (1)-- (11)--(112)--(1121);
		\draw [ultra thick, orange](2) -- (21);
		\draw [thick](1) -- (13);
		\draw [thick](1) -- (11)-- (111);
		\draw [thick](1) -- (12)-- (121);
		\draw [thick](1) -- (12)-- (122);
		\draw [thick](1) -- (12)-- (123)--(1231);
		\draw [thick](1) -- (12)-- (123)--(1232);
		\draw [thick](1) -- (12)-- (123)--(1233);
		\draw [thick](1) -- (12)-- (123)--(1234);
		\draw [thick](3) -- (32);
		\draw [thick](3) -- (31)-- (312);
		\draw [thick](3) -- (31)-- (311)--(3111);
		\draw [thick](3) -- (33)-- (331);
		\draw [thick](33)-- (332);
		\draw [thick](333)--(3331);
		\draw [thick](333)--(3332);
		\draw [thick](333)--(3334);
		\draw [thick](333)--(3335);
		\draw [thick](3) -- (34);
		\draw [ultra thick, orange](0)-- (1);
		\draw [ultra thick, orange](0)-- (2)--(22);
		\draw [ultra thick, orange](0)--(3) -- (35)-- (351);
		\draw [ultra thick, orange](3)-- (31);
		\draw [ultra thick, orange](3)-- (32);
		\draw [ultra thick, orange](3)-- (33);
		\draw [ultra thick, orange](3)-- (34);
		\draw (1)node{\textcolor{red}{\LARGE{$\bullet$}}}node[below left]{\scriptsize{$1$}} ;
		\draw (112)node{\textcolor{blue}{\Large{$\bullet$}}}node[above left]{\scriptsize{$112$}};
		\draw (122)node{\textcolor{blue}{\Large{$\bullet$}}}node[above left]{\scriptsize{$122$}};
		\draw (1233)node{\textcolor{blue}{\Large{$\bullet$}}}node[ above]{\scriptsize{$1233$}};
		\draw (22)node{\textcolor{red}{\LARGE{$\bullet$}}}node[above left]{\scriptsize{$22$}} ;
		\draw (31)node{\textcolor{red}{\LARGE{$\bullet$}}}node[below left]{\scriptsize{$31$}};
		\draw (3111)node{\textcolor{blue}{\Large{$\bullet$}}}node[above ]{\scriptsize{$3111$}};
		\draw (33)node{\textcolor{red}{\LARGE{$\bullet$}}}node[ below right]{\scriptsize{$33$}};
		\draw (333)node{\textcolor{blue}{\Large{$\bullet$}}}node[below right]{{\scriptsize{$333$}}} ;
		\draw (3332)node{\textcolor{blue}{\Large{$\bullet$}}}node[above]{\scriptsize{$3332$}};
		\draw (3334)node{\textcolor{blue}{\Large{$\bullet$}}}node[above ]{\scriptsize{$3334$}};
		\draw (351)node{\textcolor{red}{\LARGE{$\bullet$}}}node[ below right]{\scriptsize{$351$}};
	\end{tikzpicture}
	\caption{The sub-tree $\tilde{\mathcal{S}}^\bq_\emptyset(\bt^*)$ in bold and orange, and the elements of $\tilde{\mathcal{C}}_\emptyset^\bq(\bt^*)$ in red.}
	\label{exsubtreeSempty}
\end{figure}
	
	\smallbreak
	Let $\tilde{N}_\emptyset(\bt^*):=|\tilde{\mathcal{C}}_\emptyset^\bq(\bt^*)|$. Then the range of $\phi$ belongs to $\T_0^{\tilde{N}_\emptyset(\bt^*)}$. Moreover if $u_1<u_2<\cdots <u_{\tilde{N}_\emptyset(\bt^*)}$ are the elements of $\tilde{\mathcal{C}}_\emptyset^\bq(\bt^*)$ ranked in lexicographic order, we set for every $1\leq i \leq \tilde{N}_\emptyset(\bt^*)$, 
	$\phi(u_i):=\emptyset^{(i)}$, where $\emptyset^{(i)}$ denotes the root of the $i$-th tree in $\T_0^{\tilde{N}_\emptyset(\bt^*)}$.
	\smallbreak
	We then construct $\phi$ recursively: if $\eta_u=1$ 
	and $\phi(u)=v^{(i)}$ (which is an element of the $i$-th tree), we denote by $u_1<u_2<\cdots <u_k$ the elements of $\mathcal{C}_u^\bq(\bt^*)$ ranked in lexicographic order and we set for $1\leq j\leq k$, $\phi(u_j)=vj^{(i)}$. 
	\smallbreak
	Finally, we set $\F_\bq(\bt)=\phi(\bt)$.
\begin{figure}[h!]
	\begin{tikzpicture}[scale=0.9,every node/.style={scale=0.8}]
		\coordinate (0) at (0,0);
		\coordinate (1) at (-3,1);
		\coordinate (2) at (0,1);
		\coordinate (3) at (3,1);
		\coordinate (11) at (-4,2);
		\coordinate (12) at (-3,2);
		\coordinate (13) at (-2,2);
		\coordinate (111) at (-4.5,3);
		\coordinate (112) at (-3.9,3);
		\coordinate (121) at (-3.5,3);
		\coordinate (122) at (-3,3);
		\coordinate (123) at (-2.5,3);
		\coordinate (1121) at (-3.9,4);
		\coordinate (1231) at (-3.2,4);
		\coordinate (1232) at (-2.7,4);
		\coordinate (1233) at (-2.2,4);
		\coordinate (1234) at (-1.7,4);
		\coordinate (21) at (-1,2);
		\coordinate (22) at (1,2);
		\coordinate (221) at (1,3);
		\coordinate (31) at (2,2);
		\coordinate (32) at (2.5,2);
		\coordinate (33) at (3,2);
		\coordinate (34) at (3.5,2);
		\coordinate (35) at (4,2);
		\coordinate (311) at (1.75,3);
		\coordinate (312) at (2.25,3);
		\coordinate (3111) at (1.75,4);
		\coordinate (331) at (2.5,3);
		\coordinate (332) at (3,3);
		\coordinate (333) at (3.5,3);
		\coordinate (351) at (4,3);
		\coordinate (3331) at (2.5,4);
		\coordinate (3332) at (3,4);
		\coordinate (3333) at (3.5,4);
		\coordinate (3334) at (4,4);
		\coordinate (3335) at (4.5,4);
		\draw [thick](0) -- (2) -- (22)--(221);
		\draw [thick](0)-- (3) -- (33)--(333) --(3333);
		\draw [thick](0) -- (1)-- (11)--(112)--(1121);
		\draw [thick](2) -- (21);
		\draw [thick](1) -- (13);
		\draw [thick](1) -- (11)-- (111);
		\draw [thick](1) -- (12)-- (121);
		\draw [thick](1) -- (12)-- (122);
		\draw [thick](1) -- (12)-- (123)--(1231);
		\draw [thick](1) -- (12)-- (123)--(1232);
		\draw [thick](1) -- (12)-- (123)--(1233);
		\draw [thick](1) -- (12)-- (123)--(1234);
		\draw [thick](3) -- (32);
		\draw [thick](3) -- (31)-- (312);
		\draw [thick](3) -- (31)-- (311)--(3111);
		\draw [thick](3) -- (33)-- (331);
		\draw [thick](33)-- (332);
		\draw [thick](333)--(3331);
		\draw [thick](333)--(3332);
		\draw [thick](333)--(3334);
		\draw [thick](333)--(3335);
		\draw [thick](3) -- (34);
		\draw [thick](0)--(3) -- (35)-- (351);
		\draw (1)node{\textcolor{blue}{\Large{$\bullet$}}}node[below left]{\scriptsize{$1$}} ;
		\draw (112)node{\textcolor{blue}{\Large{$\bullet$}}}node[above left]{\scriptsize{$112$}};
		\draw (122)node{\textcolor{blue}{\Large{$\bullet$}}}node[above left]{\scriptsize{$122$}};
		\draw (1233)node{\textcolor{blue}{\Large{$\bullet$}}}node[ above]{\scriptsize{$1233$}};
		\draw (22)node{\textcolor{blue}{\Large{$\bullet$}}}node[above left]{\scriptsize{$22$}} ;
		\draw (31)node{\textcolor{blue}{\Large{$\bullet$}}}node[below left]{\scriptsize{$31$}};
		\draw (3111)node{\textcolor{blue}{\Large{$\bullet$}}}node[above ]{\scriptsize{$3111$}};
		\draw (33)node{\textcolor{blue}{\Large{$\bullet$}}}node[ below right]{\scriptsize{$33$}};
		\draw (333)node{\textcolor{blue}{\Large{$\bullet$}}}node[below right]{\scriptsize{$333$}} ;
		\draw (3332)node{\textcolor{blue}{\Large{$\bullet$}}}node[above]{\scriptsize{$3332$}};
		\draw (3334)node{\textcolor{blue}{\Large{$\bullet$}}}node[above ]{\scriptsize{$3334$}};
		\draw (351)node{\textcolor{blue}{\Large{$\bullet$}}}node[ below right]{\scriptsize{$351$}};
	\end{tikzpicture}
	\medbreak
	\begin{tikzpicture}[scale=1.4]
		\coordinate (1) at (-2,0);
		\coordinate (112) at (-2.5,1);
		\coordinate (122) at (-2,1);
		\coordinate (1233) at (-1.5,1);
		\coordinate (22) at (-1,0);
		\coordinate (31) at (0,0);
		\coordinate (33) at (1,0);
		\coordinate (3111) at (0,1);
		\coordinate (333) at (1,1);
		\coordinate (351) at (2,0);
		\coordinate (3332) at (0.5,2);
		\coordinate (3334) at (1.5,2);
		\draw [thick](-2.5,0)--(2.5,0);
		\draw [thick](1)--(112);
		\draw [thick](1)--(122);
		\draw [thick](1)--(112);
		\draw [thick](1)--(1233);
		\draw [thick](31)--(3111);
		\draw [thick](33)--(333)--(3332);
		\draw [thick](333)--(3334);
		\draw (1)node[below]{\footnotesize{$1$}} ;
		\draw (112)node[above left]{\footnotesize{$112$}};
		\draw (122)node[above ]{\footnotesize{$122$}};
		\draw (1233)node[ above right]{\footnotesize{$1233$}};
		\draw (22)node[below]{\footnotesize{$22$}} ;
		\draw (31)node[below ]{\footnotesize{$31$}};
		\draw (3111)node[above ]{\footnotesize{$3111$}};
		\draw (33)node[ below]{\footnotesize{$33$}};
		\draw (333)node[below right]{\footnotesize{$333$}} ;
		\draw (3332)node[above]{\footnotesize{$3332$}};
		\draw (3334)node[above ]{\footnotesize{$3334$}};
		\draw (351)node[ below ]{\footnotesize{$351$}};
	\end{tikzpicture}
	\caption{A marked tree $\bt^{*}$ and the forest $\F_\bq(\bt)$ with $5$ trees.}
	\label{exforest}
\end{figure}
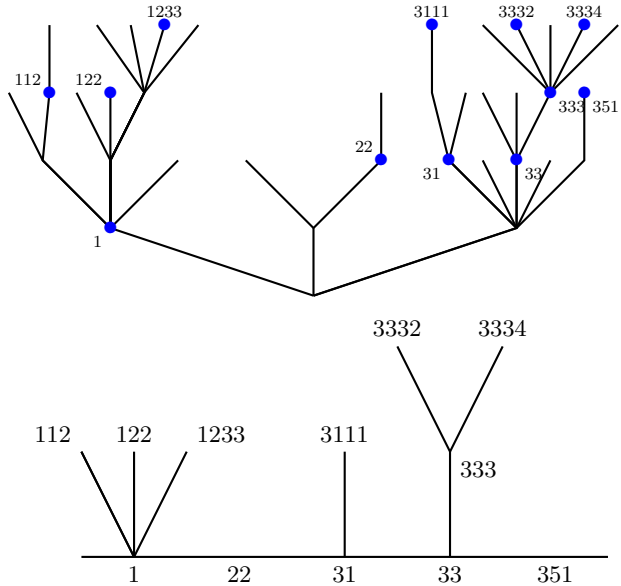
	\newpage
	Let $\tau^{*}$ be a MGW with offspring distribution $\bp$ and mark function $\bq$. Let us describe the distribution of $\F_\bq(\tau)$.
	\smallbreak
	We define the offspring distribution $\tilde{\bp}$ by: 
	\begin{equation}\label{probtilde}
		\tilde{\bp}(k) :=
		\left\{
		\begin{array}{ll}
			\bp(k)(1-\bq(k)),& \text{if}~k \neq 0,\\
			\bp(0)+\overset{+\infty}{\underset{j=1}{\sum}}\bp(j)\bq(j),& \text{otherwise}.
		\end{array}
		\right. 
	\end{equation}
	Note that $\tilde{\bp}(0)>0$ and its mean satisfies $\mu(\tilde{\bp})<+\infty$ as $\bp(0)>0$ and $\mu(\bp)<+\infty$.
	\smallbreak
	By construction, $\tilde{\mathcal{S}}^\bq_\emptyset(\tau^*)$ is distributed as a sub-critical MGW with offspring distribution $\tilde{\bp}$. On $\tilde{\mathcal{S}}^\bq_\emptyset(\tau^*)$, only leaves can be marked, and any unmarked leaf in this tree was already unmarked in the original tree $\tau$. We denote by $L$ the number of leaves of $\tilde{\mathcal{S}}^\bq_\emptyset(\tau^*)$ and we denote by $N$ the number of marked leaves. Let $X$ be a random variable with distribution $\bp$.
	\begin{prop}\label{propL} Let  $\bp$ satisfying \eqref{condp} and mark function $\bq$ satisfying \eqref{condq}. If $\mu(\bp)<1$ assume that $\bq$ satisfies \eqref{limitq} and \eqref{defqlim1}. Then, $L$ admits a moment of order 2 and:
\[\E[L]= \dfrac{\tilde{\bp}(0)}{1-\E[X(1-\bq(X))]}.\]
	\end{prop}
	\begin{proof}
		Let $Y$, given $X$, be a Bernoulli random variable with parameter $\bq(X)$.
		We determine $\E[L^i]$ thanks to the generating function of $L$.\\
		Let $s\in [0,1)$, and thanks to the branching property:
		\begin{align}\label{eqgenL}
			\E[s^L]&=\tilde{\bp}(0)s+\E\Big[s^{\overset{X}{\underset{i=1}{\sum}}L_i}\indic_{\{X>0,Y=0\}}\Big]=\tilde{\bp}(0)s+\E\left[\E[s^L]^X(1-\bq(X))\indic_{\{X>0\}}\right]\nonumber\\
			&=\tilde{\bp}(0)s+\overset{+\infty}{\underset{i=1}{\sum}}\bp(i)(1-\bq(i))\E[s^L]^i,
		\end{align}
		where $(L_i)_{i\geq1}$ are independent copies of $L$ and independent of $(X,Y)$.\\ 
		Differentiating the equation \eqref{eqgenL}, we obtain, for every $s\in[0,1)$:
		\begin{equation}\label{firstdergenL}
			\E[Ls^{L-1}]=\frac{ \tilde{\bp}(0)}{1-\sum_{i\ge 1}\bp(i)(1-\bq(i))i\E[s^L]^{i-1}} 
		\end{equation}
		Since we have $\E[X]=1$ or $\E[X]<1$ and $\bq$ satisfies \eqref{defqlim1}, the right member admits a limit when $s$ tends to 1, and consequently:
		\[\E[L]= \dfrac{\tilde{\bp}(0)}{1-\E[X(1-\bq(X))]}.\]
		The reasoning to prove that $L$ admits a moment of order 2 is very similar: we differentiate  \eqref{eqgenL} a second time and we have our result as $X$ admits a moment of order 2. 
	\end{proof}
%
	\begin{prop}\label{propN}
		Conditionally on $L$, the random variable $N:=\tilde{N}_\emptyset(\tau^*)$ has a binomial distribution with parameter $\left(L,\nicefrac{\E[\bq(X)]}{\tilde{\bp}(0)}\right)$, that we denote by $p_N$.
	\end{prop}
	\begin{proof}
	So, we have a sum of $L$ independent Bernoulli variables with unknown parameter and we can conclude if we show that this parameter is equal to $\nicefrac{\E[\bq(X)]}{\tilde{\bp}(0)}$:
		\begin{equation*}
		\p(N=1|L=1)=\dfrac{\p(N=1,L=1)}{\p(L=1)}=\dfrac{\displaystyle\sum_{k\in\N} \bp(k)\bq(k)}{\tilde{\bp}(0)}=\dfrac{\E[\bq(X)]}{\tilde{\bp}(0)}.
	\end{equation*}
	
	\end{proof}
	Let us define $Z^{(1)}$ by:
	\begin{equation}\label{hatX}
		Z^{(1)}:=\sum_{k=1}^{X^{(1)}}N_k,
	\end{equation}
	where $X^{(1)}$ is distributed as $X$ conditionally on $Y=1$ and $(N_i)_{i\in\N}$ is a sequence of independent copies of $N$, independent of $X^{(1)}$. Note that:
	\begin{equation}\label{meanX1}
	\p(X^{(1)}=k)=\frac{\bp(k)\bq(k)}{\E[\bq(X)]}\mbox{ and }\E[X^{(1)}]=\frac{\E[X\bq(X)]}{\E[\bq(X)]},
	\end{equation}
	and for typographical simplicity, we denote by $\bp^{(1)}$ the distribution of $Z^{(1)}$, and by $\mu^{(1)}$ its mean.
	\begin{lem}\label{prophatp}
		Assume that $\bp$ satisfies \eqref{condp} and $\bq$ satisfies \eqref{condq}. We have the three following properties: 
		\begin{enumerate}
		\item if $\mu(\bp)<1$, or $\mu(\bp)=1$, then $\mu^{(1)}<1$, or $\mu^{(1)}=1$, respectively;
		\item there exists $k\in\mathbb N^*$ such that $\bp^{(1)}(0)\bp^{(1)}(k)>0$;
		\item  if $\rho_l(\bp,\bq)=1$ then $\rho(\bp^{(1)})=1$.
		\end{enumerate}
	\end{lem}
	\begin{proof}
		\begin{enumerate}
		\item Using Propositions \ref{propL} and \ref{propN}, and \eqref{meanX1}, we have:
		\begin{align*}
			\E[Z^{(1)}]&
			=\E[X^{(1)}]\E[L]\dfrac{\E[\bq(X)]}{\tilde{\bp}(0)}
			=\dfrac{\E[X\bq(X)]}{1-\E[X(1-\bq(X))]}.
		\end{align*}
		Thus: 
		\begin{align*}
		&\E[Z^{(1)}]<1\Longleftrightarrow \E[X\bq(X)]<1-\E[X(1-\bq(X))]\Longleftrightarrow \E[X]<1,\\
		&\E[Z^{(1)}]=1\Longleftrightarrow \E[X\bq(X)]=1-\E[X(1-\bq(X))]\Longleftrightarrow \E[X]=1.
		\end{align*}
		\item \eqref{condq} implies that there exists $k\in \NN$ such that $\bp(k)\bq(k)>0$ and thus,  according to \eqref{meanX1}, $ \p(X^{(1)}=k)>0$. Consequently:
		\begin{align*}
			&\bp^{(1)}(0)=\p\Big(\sum_{i=1}^{X^{(1)}}N_i =0\Big)\geq \p(X^{(1)}=k)\p(N=0)^k>0,\\
			&\bp^{(1)}(k)=\p\Big(\sum_{i=1}^{X^{(1)}}N_i =k\Big) \geq \p(X^{(1)}=k)\p(N=1)^k>0
		\end{align*}
		\item Let $g^{(1)}$ and $g_N$ be the generating functions of $Z^{(1)}$ and $N$, respectively. For all $z\ge 0$, as $N\sim\mathscr B(L,\nicefrac{\E[\bq(X)]}{\tilde\bp(0)})$, conditionally on $L$:
		\begin{align*}
			g^{(1)}(z)=\E\left[g_N(z)^{X^{(1)}}\right]
			=\dfrac{\E\left[g_N(z)^X\bq(X)\right]}{\E[\bq(X)]}=\dfrac{\E\left[\E\left[\varphi(z)^L\right]^X\bq(X)\right]}{\E[\bq(X)]},
		\end{align*}
		where $\varphi(z):=\nicefrac{(\bp(0)\left(1-\bq(0)\right)+\E[\bq(X)]z)}{\tilde{\bp}(0)}$ and note that if $z>1$ then $\varphi(z)>1$.
		As $L$ is stochastically larger than $X\indic_{\{Y=0\}}$, we obtain: 
		\[g^{(1)}(z)\ge  \dfrac{\E\left[\E\left[\varphi(z)^X(1-\bq(X))\right]^X\bq(X)\right]}{\E[\bq(X)]}= \dfrac{\E\left[l(\varphi(z))^X\bq(X)\right]}{\E[\bq(X)]}\]
		Since $\rho_l(\bp,\bq)=1$, if $z>1$, we have that $l(\varphi(z))=+\infty$ implying that $g^{(1)}(z)=+\infty$. Consequently, $\rho(\bp^{(1)})=1$.
%
		\end{enumerate}
	\end{proof}
	Note that the forest $\F_\bq(\tau)$ is distributed as $N$ independent GW with offspring distribution $\bp^{(1)}$. 
	Let $(Z^{(1)}_i)_{i\in\NN}$ be a sequence of independent copies of $Z^{(1)}$, independent of $(N_i)_{i\in\NN}$. Then, we define $(W_n)_{n\in\mathbb N^*}$ and $(S_n)_{n\in\mathbb N^*}$ by 
$W_0=S_0=0$ a.s, and for $n\in\NN$, $W_n=\displaystyle\sum_{i=1}^nN_i$ and $S_n=\displaystyle\sum_{i=1}^n{Z^{(1)}_i}$.\\
	We use the decomposition of the GW with respect to the descendants of $\emptyset$ in $\F_\bq(\tau)$ and Dwass formula (see \cite{dwass_total_1969}): for every $j\in\NN$ and every $n\geq j$, we have 
	\begin{equation}\label{dwassformula}
		\p_j\left(\left|  \tau^{\F_\bq} \right|=n\right)=\dfrac{j}{n}\p(S_n=n-j).
	\end{equation}
	Thereby, for $n\in\NN$:
	\begin{align}\label{probamark}
		\p(M(\tau)=n)&=\underset{j\in\N}{\sum}\bp_N(j)\p_j\left(\left|  \tau^{\F_\bq} \right|=n\right)=\underset{j\in\N}{\sum}\bp_N(j)\dfrac{j}{n}\p(S_n=n-j)\nonumber\\
		&=\dfrac{1}{n}\E[N\indic_{\{S_n+N=n\}}].
	\end{align}
	More generally, using exchangeability, we have for $j,n\in\NN$:
	\begin{align}\label{probajmark}
		\p_j(M(\tau)=n)&=\sum_{i\in\N}\p(W_j=i)\p_i(\left|  \tau^{\F_\bq} \right|=n)=\dfrac{1}{n}\E[W_j\indic_{\{S_n+W_j=n\}}]\nonumber\\
		&=\dfrac{j}{n}\E[N\indic_{\{S_n+W_{j-1}+N=n\}}],
	\end{align}
	with $N$ independent of $S_n$ and $W_{j-1}$.
	\smallbreak
	We explain the difference between $\p(M(\bt^*)=n)$ and 	$\p_j(M(\bt^*)=n)$, with an example, for a marked tree $\bt^{*}$, $n=6$ and $j=4$.
\smallbreak
\begin{figure}[h!]
	\centering
	\begin{tikzpicture}[scale=0.8]
		\coordinate (0) at (0,0);
		\coordinate (1) at (-1,1);
		\coordinate (2) at (1,1);
		\coordinate (11) at (-1.5,2);
		\coordinate (12) at (-1,2);
		\coordinate (13) at (-0.5,2);
		\coordinate (111) at (-1.5,3);
		\coordinate (121) at (-1.25,3);
		\coordinate (122) at (-0.75,3);
		\coordinate (131) at (-0.5,3);
		\coordinate (1311) at (-1,4);
		\coordinate (1312) at (0,4);
		\coordinate (21) at (1,2);
		\coordinate (211) at (0.25,3);
		\coordinate (212) at (1.75,3);
		\coordinate (2111) at (0.25,4);
		\coordinate (2121) at (1.25,4);
		\coordinate (2122) at (1.75,4);
		\coordinate (2123) at (2.25,4);
		\draw [thick](0) -- (2) -- (21)--(211)--(2111);
		\draw [thick](0) -- (1)-- (13)--(131)--(1312);
		\draw [thick](21) -- (212)--(2121);
		\draw [thick](1) -- (12)--(121);
		\draw [thick](1) -- (11)-- (111);
		\draw [thick](1) -- (12)-- (122);
		\draw [thick](1) -- (13)-- (131)--(1311);
		\draw [thick](212) -- (2123);
		\draw [thick](212) -- (2122);
		\draw (12)node{\textcolor{blue}{\Large{$\bullet$}}} ;
		\draw (122)node{\textcolor{blue}{\Large{$\bullet$}}};
		\draw (1311)node{\textcolor{blue}{\Large{$\bullet$}}};
		\draw (211)node{\textcolor{blue}{\Large{$\bullet$}}} ;
		\draw (212)node{\textcolor{blue}{\Large{$\bullet$}}};
		\draw (2121)node{\textcolor{blue}{\Large{$\bullet$}}};
	\end{tikzpicture}
	\caption{A marked tree $\bt^{*}$, with 6 marks.}
	\label{exp(M(t)=n)}
\end{figure}
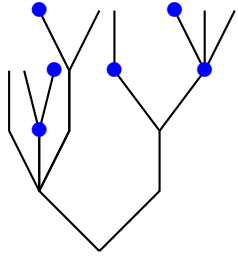
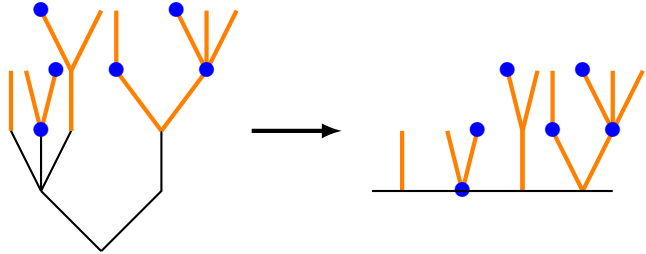
\begin{figure}[h!]
	\centering
	\begin{tikzpicture}[scale=0.8]
		\coordinate (0) at (0,0);
		\coordinate (1) at (-1,1);
		\coordinate (2) at (1,1);
		\coordinate (11) at (-1.5,2);
		\coordinate (12) at (-1,2);
		\coordinate (13) at (-0.5,2);
		\coordinate (111) at (-1.5,3);
		\coordinate (121) at (-1.25,3);
		\coordinate (122) at (-0.75,3);
		\coordinate (131) at (-0.5,3);
		\coordinate (1311) at (-1,4);
		\coordinate (1312) at (0,4);
		\coordinate (21) at (1,2);
		\coordinate (211) at (0.25,3);
		\coordinate (212) at (1.75,3);
		\coordinate (2111) at (0.25,4);
		\coordinate (2121) at (1.25,4);
		\coordinate (2122) at (1.75,4);
		\coordinate (2123) at (2.25,4);
		\coordinate (5) at (5,1);
		\coordinate (6) at (6,1);
		\coordinate (7) at (7,1);
		\coordinate (51) at (5,2);
		\coordinate (61) at (5.75,2);
		\coordinate (62) at (6.25,2);
		\coordinate (71) at (7,2);
		\coordinate (711) at (6.75,3);
		\coordinate (712) at (7.25,3);
		\coordinate (81) at (8,1);
		\coordinate (811) at (7.5,2);
		\coordinate (812) at (8.5,2);
		\coordinate (8111) at (7.5,3);
		\coordinate (8121) at (8,3);
		\coordinate (8122) at (8.5,3);
		\coordinate (8123) at (9,3);
		\draw [thick](0) -- (2)-- (21);
		\draw [ultra thick, orange](21)--(211)--(2111);
		\draw [thick](0) -- (1);
		\draw [ultra thick, orange](13)--(131)--(1312);
		\draw [ultra thick, orange](21) -- (212)--(2121);
		\draw [ultra thick, orange](11)--(111);
		\draw [thick](1) -- (12);
		\draw [ultra thick, orange](12)--(121);
		\draw [thick](1) -- (11);
		\draw [ultra thick, orange](12)-- (122);
		\draw [thick](1) -- (13);
		\draw [ultra thick, orange](13)-- (131)--(1311);
		\draw [ultra thick, orange](212) -- (2123);
		\draw [ultra thick, orange](212) -- (2122);
		\draw [ultra thick, orange](81)--(811)--(8111);
		\draw [ultra thick, orange](7)--(71)--(712);
		\draw [ultra thick, orange](81) -- (812)--(8121);
		\draw [ultra thick, orange](6)--(61);
		\draw [ultra thick, orange](5)--(51);
		\draw [ultra thick, orange](6)-- (62);
		\draw [ultra thick, orange](7)-- (71)--(711);
		\draw [ultra thick, orange](812) -- (8123);
		\draw [ultra thick, orange](812) -- (8122);
		\draw (12)node{\textcolor{blue}{\Large{$\bullet$}}} ;
		\draw (122)node{\textcolor{blue}{\Large{$\bullet$}}};
		\draw (1311)node{\textcolor{blue}{\Large{$\bullet$}}};
		\draw (211)node{\textcolor{blue}{\Large{$\bullet$}}} ;
		\draw (212)node{\textcolor{blue}{\Large{$\bullet$}}};
		\draw (2121)node{\textcolor{blue}{\Large{$\bullet$}}};
		\draw (6)node{\textcolor{blue}{\Large{$\bullet$}}} ;
		\draw (62)node{\textcolor{blue}{\Large{$\bullet$}}};
		\draw (711)node{\textcolor{blue}{\Large{$\bullet$}}};
		\draw (811)node{\textcolor{blue}{\Large{$\bullet$}}} ;
		\draw (812)node{\textcolor{blue}{\Large{$\bullet$}}};
		\draw (8121)node{\textcolor{blue}{\Large{$\bullet$}}};
		\draw [thick] (4.5,1)--(8.5,1);
		\draw [>=latex,->, ultra thick] (2.5,2)--(4,2);
	\end{tikzpicture}
	\caption{A marked tree $\bt^{*}$, with 6 marks and 4 sub-trees.}
	\label{exp4(M(t)=n)}
\end{figure}
When determining $\p(M(\bt^*)=6)$, we consider the entire tree as shown in Figure \ref{exp(M(t)=n)}. To find $\p_4(M(\bt^*)=6)$, we look at all possible decompositions of our tree $\bt^{*}$ into 4 subtrees such that the sum of the marks on these 4 subtrees equals 6. This can be illustrated in Figure \ref{exp4(M(t)=n)}.
	\section{Results with the marked conditioning}\label{condconvresult}
		We want to demonstrate that a MGW $\tau^{*}$ with offspring $\bp$ and mark function $\bq$ satisfies that, the conditional distribution of $\tau^{*}$ given $\{M(\tau)=n\}$ converges in distribution to a Kesten's tree $\tau_K^{*}=\tau_K^{*}(\bp,\bq)$, in the critical case. In the generic and sub-critical case, the conditional distribution of $\tau^{*}$ given $\{M(\tau^{*})=n\}$ converges in distribution to a Kesten's tree $\tau_K^{*}(\bp_{\theta_c}, \bq_{\theta_c})$,  and converges in distribution to a condensation tree $\tau_C^{*}(\bp,\bq)$, in the non-generic and sub-critical case.
	\smallbreak
	\subsection{Critical case}\label{critcase}
	In this section, we consider trees that belong to $\T^{*}$.
	\begin{theo}\label{cvgkesten}
		Let $\bp$ be a critical offspring distribution that satisfies \eqref{condp}. Let $\tau^{*}$ be a MGW with offspring distribution $\bp$ and mark function $\bq$ satisfying \eqref{condq}. 
		Then we have:
		\begin{equation}\label{limdis1}
		\mathrm{dist}(\tau^{*}|M(\tau^*)=n)\underset{n\to+\infty}{\longrightarrow} \mathrm{dist}(\tau_K^{*}(\bp,\bq)),
		\end{equation}
		where the limit has to be understood along a sub-sequence for which $\p(M(\tau)=n)>0$.
	\end{theo}
	\begin{proof}
		According to \eqref{critconvT}, \eqref{limdis1} is equivalent to  prove that for every finite marked tree $\bt^{*}\in\T_0^{*}$ and every leaf $x\in\mathcal{L}_0(\bt)$:
		\[\p(\tau^{*}\in \T(\bt^{*},x)|M(\tau^*)=n )\underset{n\to+\infty}{\longrightarrow}\p(\tau_K^{*}(\bp,\bq)\in \T(\bt^{*},x)).\]
		In fact, we have also to prove that $\lim_{n\rightarrow+\infty}\p(\tau^{*}=\bt^{*}|M(\tau)=n )=0$, 
		but it is direct, since if $n\geq |\bt|$, we have $\p(\tau^{*}=\bt^{*}|M(\tau^*)=n )=0.$
		\smallbreak
		We assume that $\bp^{(1)}$ is aperiodic. If $\bp^{(1)}$ is periodic, the results stay true if we consider a sub-sequence $(S_{\psi(n)})_{n\in\NN}$, with $\psi$ a growing, positive function, such that $\E[N\indic_{\{S_{\psi(n)}+N={\psi(n)}\}}]$  is positive, to apply lemmas obtained in the appendix.
		\smallbreak
		The event \textgravedbl$\tau^{*}\in \T_+(\bt^{*},x),M(\tau^*)=n$\textacutedbl ~is explained as following. 
		We use an example for $n=8$ : 
\begin{figure}[h!]
	\centering
	\begin{tikzpicture}[scale=0.8]
		\coordinate (0) at (0,0);
		\coordinate (1) at (-1,1);
		\coordinate (2) at (1,1);
		\coordinate (11) at (-1.5,2);
		\coordinate (12) at (-1,2);
		\coordinate (13) at (-0.5,2);
		\coordinate (111) at (-1.5,3);
		\coordinate (121) at (-1.25,3);
		\coordinate (122) at (-0.75,3);
		\coordinate (131) at (-0.5,3);
		\coordinate (1311) at (-1,4);
		\coordinate (1312) at (0,4);
		\coordinate (21) at (1,2);
		\coordinate (211) at (0.25,3);
		\coordinate (212) at (1.75,3);
		\coordinate (2111) at (0.25,4);
		\coordinate (2121) at (1.25,4);
		\coordinate (2122) at (1.75,4);
		\coordinate (2123) at (2.25,4);
		\draw [thick](0) -- (2) -- (21)--(211)--(2111);
		\draw [thick](0) -- (1)-- (13);
		\draw [thick](21) -- (212)--(2121);
		\draw [thick](1) -- (12)--(121);
		\draw [thick](1) -- (11)-- (111);
		\draw [thick](1) -- (12)-- (122);
		\draw [thick](1) -- (13);
		\draw [thick](212) -- (2123);
		\draw [thick](212) -- (2122);
		\draw (0) node[below left]{{$1$}};
		\draw (2)node{\textcolor{blue}{\Large{$\bullet$}}} ;
		\draw (12)node{\textcolor{blue}{\Large{$\bullet$}}} ;
		\draw (122)node{\textcolor{blue}{\Large{$\bullet$}}};
		\draw (13)node{\textcolor{blue}{\Large{$\bullet$}}} node[below right]{{$x$}};
		\draw (211)node{\textcolor{blue}{\Large{$\bullet$}}} ;
		\draw (212)node{\textcolor{blue}{\Large{$\bullet$}}};
		\draw (2121)node{\textcolor{blue}{\Large{$\bullet$}}};
	\end{tikzpicture}
	~~
	\begin{tikzpicture}[scale=0.8]
		\coordinate (0) at (0,0);
		\coordinate (1) at (-1,1);
		\coordinate (2) at (1,1);
		\coordinate (11) at (-1.5,2);
		\coordinate (12) at (-1,2);
		\coordinate (13) at (-0.5,2);
		\coordinate (111) at (-1.5,3);
		\coordinate (121) at (-1.25,3);
		\coordinate (122) at (-0.75,3);
		\coordinate (131) at (-0.5,3);
		\coordinate (1311) at (-1,4);
		\coordinate (1312) at (0,4);
		\coordinate (21) at (1,2);
		\coordinate (211) at (0.25,3);
		\coordinate (212) at (1.75,3);
		\coordinate (2111) at (0.25,4);
		\coordinate (2121) at (1.25,4);
		\coordinate (2122) at (1.75,4);
		\coordinate (2123) at (2.25,4);
		\draw [thick](0) -- (2) -- (21)--(211)--(2111);
		\draw [thick](0) -- (1)--(13);
		\draw [thick](21) -- (212)--(2121);
		\draw [thick](1) -- (12)--(121);
		\draw [thick](1) -- (11)-- (111);
		\draw [thick](1) -- (12)-- (122);
		\draw [thick](212) -- (2123);
		\draw [thick](212) -- (2122);
		\draw (0) node[below left]{{$2$}};
		\draw (2)node{\textcolor{blue}{\Large{$\bullet$}}} ;
		\draw (12)node{\textcolor{blue}{\Large{$\bullet$}}} ;
		\draw (122)node{\textcolor{blue}{\Large{$\bullet$}}};
		\draw (13)node[below right]{{$x$}};
		\draw (211)node{\textcolor{blue}{\Large{$\bullet$}}} ;
		\draw (212)node{\textcolor{blue}{\Large{$\bullet$}}};
		\draw (2121)node{\textcolor{blue}{\Large{$\bullet$}}};
	\end{tikzpicture}
	~~
	\begin{tikzpicture}[scale=0.8]
		\coordinate (0) at (0,0);
		\coordinate (1) at (-1,1);
		\coordinate (2) at (1,1);
		\coordinate (11) at (-1.5,2);
		\coordinate (12) at (-1,2);
		\coordinate (13) at (-0.5,2);
		\coordinate (111) at (-1.5,3);
		\coordinate (121) at (-1.25,3);
		\coordinate (122) at (-0.75,3);
		\coordinate (131) at (-0.5,3);
		\coordinate (1311) at (-1,4);
		\coordinate (1312) at (0,4);
		\coordinate (21) at (1,2);
		\coordinate (211) at (0.25,3);
		\coordinate (212) at (1.75,3);
		\coordinate (2111) at (0.25,4);
		\coordinate (2121) at (1.25,4);
		\coordinate (2122) at (1.75,4);
		\coordinate (2123) at (2.25,4);
		\draw [thick](0) -- (2) -- (21)--(211)--(2111);
		\draw [thick](0) -- (1)-- (13);
		\draw [ultra thick, orange](13)--(131)--(1312);
		\draw [thick](21) -- (212)--(2121);
		\draw [thick](1) -- (12)--(121);
		\draw [thick](1) -- (11)-- (111);
		\draw [thick](1) -- (12)-- (122);
		\draw [ultra thick, orange](13)-- (131)--(1311);
		\draw [thick](212) -- (2123);
		\draw [thick](212) -- (2122);
		\draw (0) node[below left]{{$3$}};
		\draw (1311)node{\textcolor{blue}{\Large{$\bullet$}}} ;
		\draw (2)node{\textcolor{blue}{\Large{$\bullet$}}} ;
		\draw (12)node{\textcolor{blue}{\Large{$\bullet$}}} ;
		\draw (122)node{\textcolor{blue}{\Large{$\bullet$}}};
		\draw (13)node{\textcolor{blue}{\Large{$\bullet$}}} node[below right]{{$x$}};
		\draw (211)node{\textcolor{blue}{\Large{$\bullet$}}} ;
		\draw (212)node{\textcolor{blue}{\Large{$\bullet$}}};
		\draw (2121)node{\textcolor{blue}{\Large{$\bullet$}}};
	\end{tikzpicture}
	\caption{A marked tree $\bt^{*}$ ($1$), the sub-tree $\mathcal{S}^x(\bt^{*})$ ($2$) and the tree obtained by grafting a tree at $x$ on $\bt^{*}$, such that $M(\tau)=8$ ($3$).}
	\label{exgreftleaf}
\end{figure}
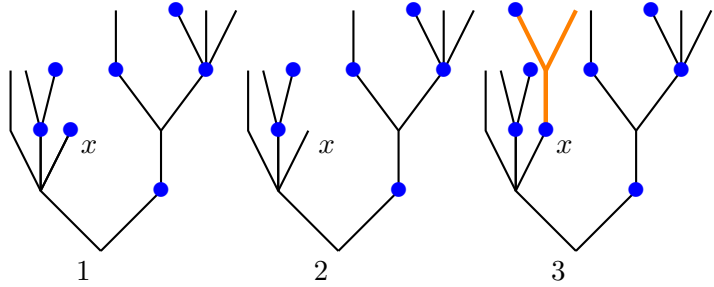
		\smallbreak
		Let $m:=M(\bt^*)$ and $D(\bt^{*},x):=\dfrac{\p(\tau^{*} = \mathcal{S}^x(\bt^{*}))}{\bp(0)(1-\bq(0))}$, thanks to \eqref{probajmark}, we have:
		{\small\begin{align*}
			&\p(\tau^{*}\in\T(\bt^{*},x),M(\tau^*)=n)=\dfrac{\p(\tau^{*} = \mathcal{S}^x(\bt^{*}))}{\bp(0)(1-\bq(0))}\underset{j\in\N}{\sum}\bp(j)\alpha_{j,x} \p_j(M(\tau^*)=n-m)\\
			&=D(\bt^{*},x)\underset{j\in\N}{\sum}\bp(j)\alpha_{j,x}\dfrac{j}{n-m}\E[N\indic_{\{S_{n-m}+W_{j-1}+N=n-m\}}],
		\end{align*}}
		with for all $j$, $\alpha_{j,x}:= \bq(j)\eta_x(\bt) +(1-\bq(j))(1-\eta_x(\bt))$.\\
		In particular, according to \eqref{probamark}, we obtain:
		\begin{align*}
			&\p(\tau^{*}\in\T(\bt^{*},x)|M(\tau^*)=n)\\
			&=D(\bt^{*},x)\dfrac{n}{n-m} \underset{j\in\N}{\sum}\bp(j)\alpha_{j,x} j \dfrac{\E[N\indic_{\{S_{n-m}+W_{j-1}+N=n-m\}}]}{\E\left[N\indic_{\{S_n+N=n\}}\right]}\\
			&=D(\bt^{*},x)\delta_{n,m}\underset{j\in\N}{\sum}\bp(j)\alpha_{j,x} j \dfrac{\E[N\indic_{\{S_{n-m}+W_{j-1}+N=n-m\}}]}{\E\left[N\indic_{\{S_{n-m}+N=n-m\}}\right]}\\
			&=D(\bt^{*},x)\delta_{n,m}B_{n-m,0},
		\end{align*}
		with $\delta_{n,m}:=\dfrac{n}{n-m}\dfrac{\E[N\indic_{\{S_{n-m}+N=n-m\}}]}{\E[N\indic_{\{S_n+N=n\}}]}$.\\
		Notice that \eqref{limTRSn2}, in Lemma \ref{limSnTR2} for $T=N$, $R=0$ and $k=0$, implies that $$\underset{n\to+\infty}{\lim}\delta_{n,m}=1.$$
		Moreover, according to Lemma \ref{limBn,lneg}~, for $l=0$, we have 
		\begin{align*}
			\underset{n\to+\infty}{\lim}B_{n,0}
			&=\E[X\bq(X)]\eta_x(\bt)+\E[X(1-\bq(X))](1-\eta_x(\bt))=\E[X\alpha(X)].
		\end{align*}
		Thereby, according to \eqref{exptau^*K}, we get:
		\begin{align*}
			\underset{n\to+\infty}{\lim}\p(\tau^{*}\in\T(\bt^{*},x)|M(\tau^*)=n)&=D(\bt^{*},x)\E[X\alpha_{X,x}]\\
			&=\p(\tau_K^{*}(\bp,\bq)\in \T(\bt^{*},x)).
		\end{align*}
	\end{proof}
	\subsection{Sub-critical case}\label{subcritcase}

	\subsubsection{Generic case}
	According to Definition \ref{dfgen}, if $\bp$ is generic, it exists $\theta_c\in I$ such that $\bp_{\theta_c}$ is critical, and recall that $\theta_c$ is unique. Thanks to the Proposition \ref{loicondid} and Theorem \ref{cvgkesten}, we have:
	\begin{cor}\label{corconvdist}
		Let $\tau^{*}$ be a sub-critical MGW with offspring distribution $\bp$ satisfying \eqref{condp} and mark function $\bq$ satisfying \eqref{condq}.
		If $\bp$ is generic and $\bp_{\theta_c}$ admits a moment of order $2$ (always true for $\theta_c<\rho$), then: 
		\[ \mathrm{dist}(\tau^{*} | M(\tau^*)=n)\underset{n\to+\infty}{\longrightarrow} \mathrm{dist}(\tau_{K}^{*}(\bp_{\theta_c},\bq_{\theta_c})).\]
	\end{cor}
	\begin{proof}
	As $\mathrm{dist}(\tau^{*} | M(\tau^*)=n)=\mathrm{dist}(\tau_{\theta_c}^{*} | M(\tau^*_{\theta_c})=n)$, according to Proposition \ref{propL}, we have to prove that $\bp_{\theta_c}$ and $\bq_{\theta_c}$ satisfy respectively \eqref{condp} and  \eqref{condq}. 
	As $\theta_c\in I$, see the beginning of subsection \ref{resultcond}, it remains to prove that 
	$\sum_{k\ge 0}k^2\bp_{\theta_c}(k)<+\infty$, in other words:
\[\sum_{k\ge0}k^2\theta_c^{k-1}\bp(k)(c_{\theta_c}\bq(k)+1-\bq(k))\le\left(c_{\theta_c}+1\right)\E[X^2\theta_c^{X-1}]<+\infty.\]Since $\theta_c\in I$, $\theta_c < \rho$ the convergence radius of $g$ which is $\mathcal{C}^{\infty}(]-\rho,\rho[)$, thereby $\E[X^2\theta_c^{X-1}]<+\infty$.
	\end{proof}
	\begin{cor}\label{corcondconvloc}
			Let $\tau$ be a sub-critical and generic Galton-Watson tree with offspring distribution $\bp$ satisfying \eqref{condp} and its associated mark function $\bq$ satisfying \eqref{condq}. \\		
			If $\bp_{\theta_c}$ admits a moment of order $2$ (always true for $\theta_c< \rho(\bp)$) we have 
			\[ \mathrm{dist}(\tau| M(\tau)=n)\underset{n\to +\infty}{\to}\mathrm{dist}(\tau_K(\bp_{\theta_c})).\]
		\end{cor}

	\begin{proof}
		This result is a natural consequence of Corollary \ref{corconvdist}.
			Indeed, using again \eqref{critconvT}, as in the proof of Theorem \ref{cvgkesten}, we just have to show that for every finite tree $\bt\in\T_0$ and every leaf $x\in\mathcal{L}_0(\bt)$:
		\[\p(\tau\in \T(\bt,x)|M(\tau)=n )\underset{n\to+\infty}{\longrightarrow}\p(\tau_K(\bp_{\theta_c}))\in \T(\bt,x)).\] 
			
First, if $\psi:\T^{*}\mapsto \T$, is defined by $\psi(\bt^*)=\bt$, Proposition \ref{loicondid} implies:
		\begin{align*}
			\p(\tau\in \T(\bt,x)|M(\tau)=n)&=\sum_{\bt^{*}\in\T_0^{*},\psi(\bt^{*})=\bt} \p(\tau^{*}\in\T(\bt^{*},x)|M(\tau)=n)\\
			&=\sum_{\bt^{*}\in\T_0^{*},\psi(\bt^{*})=\bt} \p(\tau_{\theta_c}^{*}\in\T(\bt^{*},x)|M(\tau_{\theta_c})=n).
		\end{align*}
		As the previous sum is finite ($|\{\bt^{*}\in\T_0^{*},\psi(\bt^{*})=\bt\}|=2^{| t|}$), Theorem \ref{cvgkesten} implies that:
		{\small\begin{align*}
			&\lim_{n\to+\infty}\p(\tau\in\T(\bt,x)|M(\tau)=n)=\sum_{\bt^{*}\in\T_0^{*},\psi(\bt^{*})=\bt}\lim_{n\to+\infty}\p(\tau^{*}\in\T(\bt^{*},x)|M(\tau)=n)\\
			&=\sum_{\bt^{*}\in\T_0^{*},\psi(\bt^{*})=\bt}\p(\tau_{K}^{*}(\bp_{\theta_c},\bq_{\theta_c})\in \T(\bt^{*},x))=\p(\tau_K(\bp_{\theta_c})\in \T(\bt,x)).
		\end{align*}}
	\end{proof}
%
	\subsubsection{Non-generic case}\label{Nongenericcase}
	We recall that if $\rho_l(\bp,\bq)=1$ then we have $\rho(\bp)=1$, since, $\rho(\bp)\leq \rho_l(\bp,\bq)$ and $g(1)=\mu<+\infty$.
	\begin{theo}\label{cvgcondtree}
		Let $\tau^{*}$ be a sub-critical MGW with offspring distribution $\bp$ satisfying \eqref{condp}, \eqref{probpalpha} and $\rho_l(\bp,\bq)=1$, and mark function $\bq$ satisfying \eqref{condq}, \eqref{limitq} and \eqref{defqlim1}. 
		We have that:
		\begin{equation}\label{convcondtreeeq}
			\mathrm{dist}(\tau^{*}|M(\tau^*)=n)\underset{n\to+\infty}{\longrightarrow} \mathrm{dist}(\tau_C^{*}(\bp,\bq)).
		\end{equation}
	\end{theo}
	\begin{proof}
		According to \eqref{critconvTinf}, it is enough to prove for all marked tree $\bt^{*}\in \T_0^{*}$, $x\in \bt$ and $k\in\N$:
		\[\lim_{n\to+\infty} \p(\tau^{*}\in \T_+(\bt^{*},x,k)|M(\tau^*)=n)=\p(\tau_C^{*}(\bp,\bq)\in\T_+(\bt^{*},x,k))\]
		Let $m:=M(\bt^*)$ and $l:=k_x(\bt)$. \smallbreak
		We assume that $\bp^{(1)}$ is aperiodic. If $\bp^{(1)}$ is periodic, the results stay true if we consider a sub-sequence $(S_{\psi(n)})_{n\in\NN}$, with $\psi$ a growing, positive function, such that $\E[N\indic_{\{S_{\psi(n)}+N={\psi(n)}\}}]$, $C_{\psi(n),0}(0)$ and $C_{\psi(n),l}(-1)$ (see \eqref{equalidad}) are positive, to apply lemmas obtained in the appendix.
		\smallbreak
		Using Dwass Formula \eqref{probajmark}, we have 
		{\small\begin{align*}
			&\p(\tau^{*}\in \T_+(\bt^{*},x,k),M(\tau^*)=n)
			= \underset{j\geq k-l}{\sum}C(\bt^{*},x)\alpha_{j+l,x}\bp(j+l)\p_j(M(\tau^*)=n-m)\\
			&= C(\bt^{*},x)\underset{j\geq \max(l+1,k)}{\sum}\alpha_{j,x}\bp(j)\p_{j-l}(M(\tau^*)=n-m)\\
			&=C(\bt^{*},x)\underset{j\geq\max(l+1,k)}{\sum}\bp(j)\alpha_{j,x}\dfrac{j-l}{n-m}\E[N\indic_{\{S_{n-m}+W_{j-1-l}+N=n-m\}}],
		\end{align*}}
		with for all $j\in\N$, $\alpha_{j,x}:=\bq(j)\eta_x(\bt)+(1-\bq(j))(1-\eta_x(\bt))$.
		We recall \eqref{dfC} that $$C(\bt^{*},x):=\dfrac{\p(\tau^{*} = \mathcal{S}^x(\bt^{*}))}{\bp(0)(1-\bq(0))}\p_{k_x(\bt)}(\tau^{*}=\F_x(\bt^{*})).$$
		The event \textgravedbl$\tau^{*}\in \T_+(\bt^{*},x,k),M(\tau^*)=n$\textacutedbl ~is explained as following. In the same way of Figure \ref{exgreftleaf}, we use an example for $n=9$ and $k=3$ :
		\smallbreak
		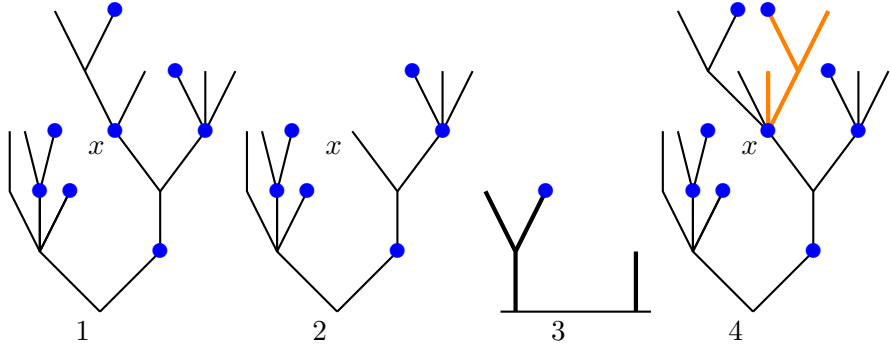
\begin{figure}[h!]
		\centering
		\begin{tikzpicture}[scale=0.8]
			\coordinate (0) at (0,0);
			\coordinate (1) at (-1,1);
			\coordinate (2) at (1,1);
			\coordinate (11) at (-1.5,2);
			\coordinate (12) at (-1,2);
			\coordinate (13) at (-0.5,2);
			\coordinate (111) at (-1.5,3);
			\coordinate (121) at (-1.25,3);
			\coordinate (122) at (-0.75,3);
			\coordinate (131) at (-0.5,3);
			\coordinate (1311) at (-1,4);
			\coordinate (1312) at (0,4);
			\coordinate (21) at (1,2);
			\coordinate (211) at (0.25,3);
			\coordinate (212) at (1.75,3);
			\coordinate (2111) at (-0.25,4);
			\coordinate (2112) at (0.75,4);
			\coordinate (2121) at (1.25,4);
			\coordinate (2122) at (1.75,4);
			\coordinate (2123) at (2.25,4);
			\coordinate (21111) at (-0.75,5);
			\coordinate (21112) at (0.25,5);
			\draw [thick](0) -- (2) -- (21)--(211)--(2112);
			\draw [thick](0) -- (1)-- (13);
			\draw [thick](21) -- (212)--(2121);
			\draw [thick](1) -- (12)--(121);
			\draw [thick](1) -- (11)-- (111);
			\draw [thick](1) -- (12)-- (122);
			\draw [thick](1) -- (13);
			\draw [thick](212) -- (2123);
			\draw [thick](212) -- (2122);
			\draw [thick](211) -- (2111)--(21111);
			\draw [thick](2111) -- (21112);
			\draw (0) node[below left]{{$1$}};
			\draw (2)node{\textcolor{blue}{\Large{$\bullet$}}} ;
			\draw (12)node{\textcolor{blue}{\Large{$\bullet$}}} ;
			\draw (122)node{\textcolor{blue}{\Large{$\bullet$}}};
			\draw (13)node{\textcolor{blue}{\Large{$\bullet$}}} ;
			\draw (211)node{\textcolor{blue}{\Large{$\bullet$}}} node[below left]{{$x$}} ;
			\draw (212)node{\textcolor{blue}{\Large{$\bullet$}}};
			\draw (2121)node{\textcolor{blue}{\Large{$\bullet$}}};
			\draw (21112)node{\textcolor{blue}{\Large{$\bullet$}}};
		\end{tikzpicture}
		~~
		\begin{tikzpicture}[scale=0.8]
			\coordinate (0) at (0,0);
			\coordinate (1) at (-1,1);
			\coordinate (2) at (1,1);
			\coordinate (11) at (-1.5,2);
			\coordinate (12) at (-1,2);
			\coordinate (13) at (-0.5,2);
			\coordinate (111) at (-1.5,3);
			\coordinate (121) at (-1.25,3);
			\coordinate (122) at (-0.75,3);
			\coordinate (131) at (-0.5,3);
			\coordinate (1311) at (-1,4);
			\coordinate (1312) at (0,4);
			\coordinate (21) at (1,2);
			\coordinate (211) at (0.25,3);
			\coordinate (212) at (1.75,3);
			\coordinate (2111) at (0.25,4);
			\coordinate (2121) at (1.25,4);
			\coordinate (2122) at (1.75,4);
			\coordinate (2123) at (2.25,4);
			\draw [thick](0) -- (2) -- (21)--(211);
			\draw [thick](0) -- (1)--(13);
			\draw [thick](21) -- (212)--(2121);
			\draw [thick](1) -- (12)--(121);
			\draw [thick](1) -- (11)-- (111);
			\draw [thick](1) -- (12)-- (122);
			\draw [thick](212) -- (2123);
			\draw [thick](212) -- (2122);
			\draw (0) node[below left]{{$2$}};
			\draw (2)node{\textcolor{blue}{\Large{$\bullet$}}} ;
			\draw (12)node{\textcolor{blue}{\Large{$\bullet$}}} ;
			\draw (122)node{\textcolor{blue}{\Large{$\bullet$}}};
			\draw (211)node[below left]{{$x$}};
			\draw (13)node{\textcolor{blue}{\Large{$\bullet$}}} ;
			\draw (212)node{\textcolor{blue}{\Large{$\bullet$}}};
			\draw (2121)node{\textcolor{blue}{\Large{$\bullet$}}};
		\end{tikzpicture}
		~~
		\begin{tikzpicture}[scale=0.8]
			\coordinate (0) at (0,0);
			\draw [ultra thick] (1,0)--(1,1);
			\draw [ultra thick] (-1,0)--(-1,1)--(-1.5,2);
			\draw [ultra thick] (-1,1)--(-0.5,2);
			\draw [thick] (-1.25,0)--(1.25,0);
			\draw (0) node[below left]{{$3$}};
			\draw (-0.5,2)node{\textcolor{blue}{\Large{$\bullet$}}};
		\end{tikzpicture}
		~~
		\begin{tikzpicture}[scale=0.8]
			\coordinate (0) at (0,0);
			\coordinate (1) at (-1,1);
			\coordinate (2) at (1,1);
			\coordinate (11) at (-1.5,2);
			\coordinate (12) at (-1,2);
			\coordinate (13) at (-0.5,2);
			\coordinate (111) at (-1.5,3);
			\coordinate (121) at (-1.25,3);
			\coordinate (122) at (-0.75,3);
			\coordinate (21) at (1,2);
			\coordinate (211) at (0.25,3);
			\coordinate (212) at (1.75,3);
			\coordinate (2111) at (-0.75,4);
			\coordinate (2112) at (-0.25,4);
			\coordinate (2113) at (0.25,4);
			\coordinate (2114) at (0.75,4);
			\coordinate (2121) at (1.25,4);
			\coordinate (2122) at (1.75,4);
			\coordinate (2123) at (2.25,4);
			\coordinate (21111) at (-1.25,5);
			\coordinate (21112) at (-0.25,5);
			\coordinate (21141) at (0.25,5);
			\coordinate (21142) at (1.25,5);
			\draw [thick](0) -- (2) -- (21)--(211)--(2112);
			\draw [thick](0) -- (1)-- (13);
			\draw [thick](21) -- (212)--(2121);
			\draw [thick](1) -- (12)--(121);
			\draw [thick](1) -- (11)-- (111);
			\draw [thick](1) -- (12)-- (122);
			\draw [thick](1) -- (13);
			\draw [thick](212) -- (2123);
			\draw [thick](212) -- (2122);
			\draw [thick](211) -- (2111)--(21111);
			\draw [thick](2111) -- (21112);
			\draw [ultra thick, orange](2114)--(21142);
			\draw [ultra thick, orange](211)-- (2114)--(21141);
			\draw [ultra thick, orange](211)-- (2113);
			\draw (0) node[below left]{{$4$}};
			\draw (21141)node{\textcolor{blue}{\Large{$\bullet$}}} ;
			\draw (2)node{\textcolor{blue}{\Large{$\bullet$}}} ;
			\draw (12)node{\textcolor{blue}{\Large{$\bullet$}}} ;
			\draw (122)node{\textcolor{blue}{\Large{$\bullet$}}};
			\draw (13)node{\textcolor{blue}{\Large{$\bullet$}}} ;
			\draw (211)node{\textcolor{blue}{\Large{$\bullet$}}} node[below left]{{$x$}};
			\draw (212)node{\textcolor{blue}{\Large{$\bullet$}}};
			\draw (2121)node{\textcolor{blue}{\Large{$\bullet$}}};
			\draw (21112)node{\textcolor{blue}{\Large{$\bullet$}}};
		\end{tikzpicture}
		\caption{A marked tree $\bt^{*}$ ($1$), the sub-tree $\mathcal{S}^x(\bt^{*})$ ($2$), the forest $\F_x(\bt^{*})$ ($3$) and the tree obtained by grafting a tree at $x$ on the right of $\bt^{*}$, such that $M(\tau^*)=9$ and $k_x(t)\geq 3$ ($4$).}
		\label{exgreft}
	\end{figure}
		\smallbreak
		Henceforth, we have:
		\begin{multline}
			\p(\tau^{*}\in \T_+(\bt^{*},x,k)|M(\tau^*)=n)= C(\bt^{*},x)\delta_{n,m}\\
			\times\left(B_{n-m,l}(x)-\overset{k-1}{\underset{j=l+1}{\sum}}\bp(j)\alpha_{j,x}(j-l)a_{n-m,j}\right),
		\end{multline}
		with for $i\in\Z$ and $n\in\NN$:
		\begin{equation}\label{Bni}
			B_{n,i}(x):=\underset{j>i}{\sum}\bp(j)\alpha_{j,x}(j-i)a_{n,j},~a_{n,j}:=\dfrac{\E[N\indic_{\{S_{n}+W_{j-1-i}+N=n\}}]}{\E[N\indic_{\{S_n+N=n\}}]},
		\end{equation}
		and $\delta_{n,m}:=\dfrac{n}{n-m}\dfrac{\E[N\indic_{\{S_{n+m}+N=n-m\}}]}{\E[N\indic_{\{S_n+N=n\}}]}$.
		
		Notice that Lemma \ref{limSnTR2}, more precisely \eqref{limTRSn2} for $T=N$ and $R=0$, implies that:
		\[\lim_{n\to+\infty}\delta_{n,m}=1,\] 
		and Lemma \ref{limSnN2} implies that 
		\[\lim_{n\to+\infty}a_{n,j}=1.\]
		Then, using Lemma \ref{limBn,lpos}:
		\begin{align*}
			\underset{n\to+\infty}{\lim}	&\p(\tau^{*}\in \T_+(\bt^{*},x,k)|M(\tau^*)=n)\\
			&=C(\bt^{*},x)\left((1-\mu(\bp))\left(\ell_\bq\eta_x(\bt)+\left(1-\ell_\bq\right)\left(1-\eta_x(\bt)\right)\right)\right.\\
			&+\left.\E[\alpha_{X,x}(X-l)_+\indic_{\{X\geq k\}}]\right).
		\end{align*}
	\end{proof}
Note that, mimicking the proof of   the first point of Corollary \ref{corcondconvloc}, we can easily obtain its second corollary saying:
			\begin{cor}
			Let $\tau$ be a sub-critical and non-generic Galton-Watson tree with offspring distribution $\bp$ satisfying \eqref{condp} and its associated mark function $\bq$ satisfying \eqref{condq}. \\		
				If $\bq$ satisfied \eqref{limitq} and, if $\ell_\bq=1$, \eqref{defqlim1} we have 
				\[ \mathrm{dist}(\tau| M(\tau)=n)\underset{n\to +\infty}{\to}\mathrm{dist}(\tau_C(\bp)).\]
		\end{cor}

	\section{Appendix}\label{appendix}
	\subsection{Preliminaries and Key Lemmas}
	 Recall the definition of $N$ in Section \ref{Model} (see \ref{propN}), $N=\mathscr{B}\left(L,\frac{\E[\bq(X)]}{\tilde{\bp}(0)}\right)$, where $L$ is the number of leaves of $\tilde{\mathcal{S}}^\bq_\emptyset(\tau^*)$ with offspring distribution $\tilde{\bp}$ \eqref{probtilde}. $X$ be a random variable with distribution $\bp$, $Y$, given $X$, be a Bernoulli random variable with parameter $\bq(X)$, and introduce $Z^{(0)}$ whose definition is very similar to $Z^{(1)}$ \eqref{hatX}, $	Z^{(1)}=\displaystyle\overset{X^{(1)}}{\underset{k=1}{\sum}}N_k,$ where $X^{(1)}$ is distributed as $X$ conditionally to be marked, and $(N_i)_{i\in\N}$ is a sequence of independent copies of $N$, independent of $X^{(1)}$.\\ In other words:
	\begin{equation}\label{tildeX}
		Z^{(0)}:=\sum_{k=1}^{X^{(0)}}N_k,
	\end{equation}
	where $X^{(0)}$ is distributed as $X$ conditionally on $Y=0$ and $(N_i)_{i\in\N}$ is a sequence of independent copies of $N$, independent of $X^{(0)}$. Moreover, throughout this section, for $j\in\lbrace 0,1\rbrace$, $\bp^{(j)}$ denotes the distribution of $Z^{(j)}$ and $\mu^{(j)}$, its mean.\\
Note that the existence of $k\in\mathbb N^*$ such that $0<\bp(k)\bq(k)<1$, ensures that $X^{(0)}$ is well defined and one have:
\[\p(X^{(0)}=k)=\frac{\bp(k)(1-\bq(k))}{1-\E[\bq(X)]}.\]
\begin{lem}\label{espXbar}
Assume that $\bp$ and $\bq$ satisfy \eqref{condp} and \eqref{condq}. If $\mu(\bp)\le1$, $Z^{(0)}\in L^2$ and moreover:
\[\E[Z^{(0)}]=\dfrac{\E[X(1-\bq(X))]}{1-\E[X(1-\bq(X))]}\dfrac{\E[\bq(X)]}{1-\E[\bq(X)]}.\]
\end{lem}

\begin{proof}
According to Lemma \ref{prophatp}, $N\in L^2$ and \eqref{condp} implies that $X^{(0)}$ is in $L^2$ too. Consequently, using for instance the generating function, $Z^{(0)}\in L^2$ and we have: 
\begin{align*}
\E[Z^{(0)}]=\E[X^{(0)}]\E[N]
&=\dfrac{\E[X(1-\bq(X))]}{1-\E[\bq(X)]}\dfrac{\E[\bq(X)]}{1-\E[X(1-\bq(X))]}.
\end{align*} 
\end{proof}
Let $(Z^{(1)}_i)_{i\in\NN}$ be a sequence of independent copies of $Z^{(1)}$, independent of $(N_i)_{i\in\NN}$. Then, we define $(W_n)_{n\in\mathbb N^*}$ and $(S_n)_{n\in\mathbb N^*}$ by 
$W_0=S_0=0$ a.s, and for $n\in\NN$, $W_n=;\displaystyle\sum_{i=1}^nN_i$ and $S_n=\displaystyle\sum_{i=1}^n{Z^{(1)}_i}$.\\
	We assume that:
	\begin{equation}\label{condhatp}
		\mu^{(1)}=1\text{ or }(\mu^{(1)}<1 \text{ and, for all }\theta>0,\E[e^{\theta Z^{(1)}}]=+\infty),
	\end{equation}
	and that $\bp^{(1)}$ is aperiodic (that is $\p(S_n=n)>0$ for all $n$ large enough). According to \cite{kemeny_probability_1959} or \cite{neveu_sur_1963}, we have the following strong ratio limit property for all $m,u\in\Z$:
	\begin{equation}\label{strongratio}
		\underset{n\to+\infty}{\lim}\dfrac{\p(S_{n+m}=n-u)}{\p(S_n=n)}=1.
	\end{equation}
	\begin{prop}\label{momexpo}
		If $\rho(\bp^{(1)})=1$, then, for $R$ a random variable equals to $Z^{(1)}$, $X$, $L$ or $N$, we have for all $\theta>0$, $\E[e^{\theta R}]=+\infty$ .
	\end{prop}
	\begin{proof}
		Since $\rho(\bp^{(1)})=1$, then $\rho_l(\bp,\bq)=1$ and $\rho(\bp)=1$ according to  \eqref{functionl} and Lemma \ref{prophatp}~. Thus the proposition is true for $Z^{(1)}$ and $X$. As $N$, conditionally on $L$, has a binomial distribution with parameters $\left(L,\nicefrac{\E[\bq(X)]}{\tilde{\bp}(0)}\right)$, if the proposition is true for $L$, it is also true for $N$.\\
		Recalling \eqref{eqgenL}, the generating function of $L$ satisfies for all $s\in\mathbb{R}$:
		\[\E[s^L]=\tilde{\bp}(0)s+\sum_{i=1}^{+\infty}\bp(i)(1-\bq(i))\E[s^L]^i.\]
		Thereby, $\E[s^L]$ converges if $\sum_{i=1}^{+\infty}\bp(i)(1-\bq(i))\E[s^L]^i=l(\E[s^L])$ converges. Since, $\rho_l(\bp,\bq)=1$ and $s\mapsto\E[s^L]$ is increasing, we need to have $\E[s^L]\leq1$, so $s\leq1$. Thus, the convergence radius of the generating function of $L$ is 1, the condition is satisfied.
	\end{proof}
	\begin{lem}\label{limSnTR2}
		Assume that $\bp^{(1)}$ is aperiodic and satisfies \eqref{condhatp}. 
Let $T$ and $R$ two random variables, independent of $(Z_i^{(1)})_{i\in\mathbb N}$,  such that $T\in L^1$ and is positive a.s.
		For all $m,k\in\Z$, we have:
		\begin{equation}\label{limTRSn2}
			\underset{n\to+\infty}{\lim}\dfrac{\E[T\indic_{\{S_{n+m}+R+T=n-k\}}]}{\E[T\indic_{\{S_{n}+R+T=n\}}]}=1,
		\end{equation}
		If $\bp^{(1)}$ is periodic, then \eqref{limTRSn2} still holds along the sub-sequence for which the denominator is positive.
	\end{lem}
	\begin{proof}
		We shall mimic the proof of Lemma 8.3 of \cite{abraham_local_2014}. 
		Since $\bp^{(1)}$ is aperiodic, the denominator of \eqref{limTRSn2} is positive for $n$ large enough and it is enough to prove the result for $m=1$ and $k$ such that $\bp^{(1)}(k)>0$.\\
		With $ \bp^{(1)}_n(k):=\frac{1}{n}\sum_{i=1}^n{\indic_{\{Z^{(1)}_i =k\}}}$, using exchangeability:
		\begin{align}\label{limTRSn2bis}
			\E[T\bp^{(1)}_n(k)\indic_{\{S_{n}+R+T=n\}}]
			&=\bp^{(1)}(k)\E\left[T \indic_{\{S_{n-1}+R+T=n-k\}}\right].
		\end{align}
		Thus, for $\varepsilon>0$:
		\begin{align*}
			\left|  \dfrac{\E[T\indic_{\{S_{n+1}+R+T=n-k\}}]}{\E[T\indic_{\{S_{n}+R+T=n\}}]}-1 \right| 
			\leq\dfrac{\E\left[T\left|  \bp^{(1)}_n(k)-\bp^{(1)}(k)\right| \indic_{\{S_{n}+R+T=n\}} \right]}{\bp^{(1)}(k)\E[T\indic_{\{S_{n}+R+T=n\}}]}\\
			\leq\dfrac{\varepsilon}{\bp^{(1)}(k)}+\dfrac{\E\left[T\right]\p(|  \bp^{(1)}_n(k)-\bp^{(1)}(k)|>\varepsilon)}{\bp^{(1)}(k)\E[T\indic_{\{S_{n}+R+T=n\}}]}=:\dfrac{\varepsilon}{\bp^{(1)}(k)}+\frac{\E[T]}{\bp^{(1)}(k)}J_{n,\varepsilon}.
		\end{align*}
		 We have our result if we prove that $J_{n,\varepsilon}$ tends to 0 when $n$ goes to infinity. Writing:
		\begin{align*}
			J_{n,\varepsilon}=
			\dfrac{\p(S_n=n)}{\E[T\indic_{\{S_{n}+R+T=n\}}]}\dfrac{\p(|  \bp^{(1)}_n(k)-\bp^{(1)}(k)|>\varepsilon)}{\p(S_n=n)}=:\alpha_n\beta_{n,\varepsilon},
		\end{align*}
		we already have that $\lim_{n\rightarrow+\infty}\beta_{n,\varepsilon}=0$, according to \cite{neveu_sur_1963}, pp.2954 and we conclude applying Fatou's lemma and the strong ratio limit property \eqref{strongratio} to $\alpha_n$:
		\[\lim_{n\rightarrow+\infty}\alpha_n\le \left(\sum_{k\in\mathbb N}\E[T\mathds{1}_{\lbrace R+T=k\rbrace}]\liminf_{n\rightarrow+\infty}\frac{\p(S_n=n-k)}{\p(S_n=n)}\right)^{-1}=\E[T]^{-1}.\]
	\end{proof}
	\begin{lem}\label{limSnN2}
		Assume that $\bp^{(1)}$ is aperiodic , with $\mu^{(1)}=1$ or ($\mu^{(1)}<1$ and $\rho(\bp^{(1)})=1$). 
		For all $m\in\N,$ $k\in\Z$, we have:
		\begin{equation}\label{limNSn2}
			\lim_{n\to+\infty}\dfrac{\E[N\indic_{\{S_{n}+W_m+N=n-k\}}]}{\E[N\indic_{\{S_n+N=n\}}]}=1=\lim_{n\to+\infty}\dfrac{\E\left[\indic_{\{S_{n+m}+N=n-k\}}\right]}{\p(S_n=n)}.
		\end{equation}
		If $\bp^{(1)}$ is periodic, then \eqref{limNSn2} still holds along the sub-sequence for which the denominator is positive.
	\end{lem}
	\begin{proof}The left equality follows from Lemma 8.6 of \cite{abraham_local_2014}, and to prove the right one, we mimic the proof of Lemma  8.4 of \cite{abraham_local_2014}.
		We define for all $j\in\N$:
		\begin{equation*}
			b_n(j)=\bp_N(j)\dfrac{\p(S_{n+m}=n-k-j)}{\p(S_n=n)},\,d_n(j)=\bp^{(1)}(j)\dfrac{\p(S_{n+m}=n-k-j)}{\p(S_n=n)}.
		\end{equation*}
		Thanks to the strong ratio limit \eqref{strongratio} we have $\lim_{n\to+\infty}b_n(j)=\bp_N(j)$ and $\lim_{n\to+\infty}d_n(j)=\bp^{(1)}(j)$.
		\smallbreak
		Moreover, using \ref{limTRSn2bis} with $T=k$ and $R=0$, we have:
		\begin{equation*}
			\sum_{j\in\mathbb N}d_n(j)=\dfrac{\p(S_{n+1+m}=n-k)}{\p(S_n=n)},
		\end{equation*}
		and $\lim_{n\to+\infty}\sum_{j\in \mathbb N}d_n(j)=1=	\sum_{j\in\mathbb N}\bp^{(1)}(j)=\sum_{j\in\mathbb N}\lim_{n\to+\infty}d_n(j)$.\\
		Then, since it exists $j_0\in\NN$ such that $\bp(j_0)\bq(j_0)>0$, we have:
		\begin{align*}
			\bp^{(1)}(j)=\p\left(\sum_{i=1}^{X^{(1)}}N_i =j\right)&\geq \p(X^{(1)}=j_0)\p\left(\sum_{i=1}^{j_0}N_i =j\right)\\
			&\geq \p(X^{(1)}=j_0){\bp_N(0)}^{j_0-1}\bp_N(j)=:C \bp_N(j) .
		\end{align*}
		Thereby, for all $j\in\N$, $b_n(j) \leq C^{-1}d_n(j)$, and the dominated convergence theorem in \cite{kallenberg_foundations_1997} (Theorem 1.21) implies that:
		\begin{align*}
			\lim_{n\to+\infty}\dfrac{\E\left[\indic_{\{S_{n+m}+N=n-k\}}\right]}{\p(S_n=n)}&=\lim_{n\to+\infty}\underset{j\in\N}{\sum}b_n(j)=\sum_{j\in\mathbb N}\lim_{n\to+\infty}b_n(j)\\
			&=\sum_{j\in\mathbb N}\bp_N(j)=1.
		\end{align*}
	\end{proof}
	\subsection{The sub-critical case: some properties}
	When $\mu^{(1)}<1$, let us recall that we assume in this case that:
	\begin{equation}
		\bp(k)=\Le(k)k^{-(1+\alpha)},
	\end{equation}
	where $\alpha>2$ and $\Le$ is a SV function, and that $\bq$ admits a finite limit in infinity:
	\begin{equation}
		\lim_{n\rightarrow+\infty}\bq(n)=:\ell_\bq\in[0,1].
	\end{equation}
A useful lemma concerning SV functions is the following, Lemma 2 in \cite{feller_introduction_1971}, pp.282:
\begin{lem}\label{SV}
If $\Le$ is a SV function, for every $0<a<b<+\infty$, uniformly in for $c\in[a,b]$:
\begin{equation}
\lim_{x\rightarrow+\infty}\frac{\Le(cx)}{\Le(x)}=1.
\end{equation}
\end{lem} 
This result can be easily seen using the representation Theorem in \cite{feller_introduction_1971}, pp.282, saying that if $\Le$ is a SV function, it can be written: 
\begin{equation}
\Le(x)=a(x)\exp\left(\int_{1}^{x}\frac{\nu(y)}{y}\mathrm{d}y\right), 
\end{equation}
 where $\nu(x)\to 0$ and $a(x) \to \mathfrak{a}>0$ when $x\rightarrow+\infty$.\\
The following corollary is a direct consequence of the previous lemma in the case of regularly varying functions:
 \begin{cor}\label{RV}
 Assume that $f:\mathbb R^{+,*}\rightarrow\mathbb R^{+,*}$ is defined by:
 \[f(x)=\frac{\Le(x)}{x^{\alpha}},\]
 where $\alpha>0$ and $\Le$ is a SV function.\\
 Moreover, let us consider two positive functions $u,v$ satisfying:
 \[\lim_{x\rightarrow+\infty}\frac{u(x)}{x}=A\le  \lim_{x\rightarrow+\infty}\frac{v(x)}{x}=B.\]
Then, for any $\varepsilon>0$, for $x$ large enough and for all $y\in[u(x),v(x)]$:
 \begin{equation}\label{corequiv}
 \frac{(1-\varepsilon)}{B^{\alpha}}\le \frac{f(y)}{f(x)}\le  \frac{(1+\varepsilon)}{A^{\alpha}}
 \end{equation}   
 \end{cor}
 
 \begin{rqqq}\label{rqRV}
 Note that we can easily adapt the previous result if $\nicefrac{u(x)}{x}\sim A(x)$ when $x$ goes to infinity.
 In this case, for any $\varepsilon>0$, for $x$ large enough and for all $y\in[u(x),v(x)]$:
 \begin{equation}\label{rqquiv}
\frac{f(y)}{f(x)}\le  \frac{(1+\varepsilon)}{A(x)^{\alpha}}.
 \end{equation}   

 \end{rqqq}
	According to Section \ref{Model}, $L$ is the number of leaves of $\tilde{\mathcal{S}}^\bq_\emptyset(\tau^*)$ with reproduction law \eqref{probtilde}:
	\begin{equation*}
		\tilde{\bp}(k) :=
		\left\{
		\begin{array}{ll}
			\bp(k)(1-\bq(k)),& \text{if}~k \neq 0,\\
			\bp(0)+\overset{+\infty}{\underset{j=1}{\sum}}\bp(j)\bq(j),& \text{otherwise}.
		\end{array}
		\right. 
	\end{equation*}
	We denote by $\tilde{X}$ a random variable with distribution $\tilde{\bp}$. 
	\subsubsection{Case $\ell_\bq\in[0,1[$}\label{caselneq1}
	According to Remark 4.8 of \cite{abraham_local_2013}, $L$ is also distributed as the total size of
	a GW tree with offspring distribution given by:
	\begin{equation}\label{tildeYdef}
		\tilde{Y}=\overset{\G-1}{\underset{k=1}{\sum}}\left(\tilde{X}_{k,+}-1\right),
	\end{equation}
	where $\G$ is a geometric law with parameter $\tilde{\bp}(0)$ independent of $(\tilde{X}_{k,+})_{k\in\NN}$, a sequence of i.i.d. random variables, distributed as $\tilde X$ given $\lbrace \tilde X>0\rbrace$. Thus, for all $n\in\NN$:
	\begin{equation}\label{Xpos}
	\p(\tilde{X}_{+}=n)=\dfrac{\Le(n)(1-\bq(n))}{n^{1+\alpha}(1-\tilde{\bp}(0))},
	\end{equation}
	implying that $\p(\G=n )=o(\p(\tilde{X}_{+}=n))$, when $n$ goes to infinity. Thus, applying Theorem 2 of \cite{bloznelis_local_2019} and Lemma \ref{SV}:
	\begin{equation*}
		\p(\tilde{Y}=n)\sim\E\left[\G-1\right]\p(\tilde{X}_{+}-1=n)
		\sim\dfrac{\Le(n)(1-\ell_\bq)}{n^{1+\alpha}\tilde{\bp}(0)}.
	\end{equation*}
	And one can write that for all $n\in\NN$:
	\begin{equation}\label{probYalpha}
	\p(\tilde{Y}=n)=\dfrac{b(n)\Le(n)}{n^{1+\alpha}}
	\end{equation}
	where $b$ is a positive function satisfying $\lim_{x\rightarrow+\infty}b(x)=\nicefrac{(1-\ell_\bq)}{\tilde\bp(0)}$.\\ 
	Using Dwass formula \cite{dwass_total_1969}, we have that,
	\begin{equation*}
		n\p(L=n)=\p\left(\tilde{S}_n=n-1\right),
	\end{equation*}
	with $\tilde{S}_n=\sum_{k=1}^n\tilde{Y}_k$, and $(\tilde{Y}_k)_{k\in\NN}$ independent copies of $\tilde{Y}$.\\
	As $\alpha>2$ in \eqref{probYalpha}, we can apply Theorem 1 of \cite{doney_large_1989};  as $\E[\tilde{Y}]<1$, there exists $\varepsilon>0$ small enough such that $\tilde{\varepsilon}:=1-\E[\tilde{Y}]-\varepsilon>0$, and as  $n-1 \ge n \E[\tilde{Y}]+n\tilde{\varepsilon}$, when $n$ goes to infinity:
	\begin{equation*}
		\p\left(\tilde{S}_n=n-1\right)=n\p\left(\tilde{Y}=\left\lfloor n(1-\E[\tilde{Y}])\right\rfloor\right)(1+o(1)).
	\end{equation*}
	Moreover, \eqref{tildeYdef} ensures that:
	\begin{equation*}
		\E[\tilde{Y}]=\E[\G-1]\E[\tilde{X}_{+}-1]
		=1-\E[L]^{-1}.
	\end{equation*}
	Finally, when $n$ goes to infinity, using Lemma \ref{SV}:
	\begin{equation*}
		\p(L=n)\sim\p\left(\tilde{Y}=\left\lfloor n(1-\E[\tilde{Y}])\right\rfloor\right)\sim \p(\tilde{Y}=n)\E[L]^{1+\alpha}.
	\end{equation*}
%
	So, we can write:
	\begin{equation}\label{probLalpha}
	\p(L=n)=\dfrac{c(n)\Le(n)}{n^{1+\alpha}}
	\end{equation}
	where $c$ is a positive function satisfying $\lim_{x\rightarrow+\infty}c(x)=\nicefrac{(1-\ell_\bq)\E[L]^{1+\alpha}}{\tilde\bp(0)}$.
	We recall, that the random variable $N$ defines the number of marked leaves among the $L$ leaves, thus $N$, conditionally on $L$, has a binomial distribution with parameter $\left(L,\tilde c\right)$ where $\tilde{c}=\nicefrac{\E[\bq(X)]}{\tilde{\bp}(0)}$. We have the following result: 
		\begin{lem}\label{lemprobNalpha}
		For all $n\in\NN$:
		\begin{equation}\label{probNalpha}
			\bp_N(n)=\p(N=n)=\dfrac{d(n)\Le(n)}{n^{1+\alpha}},
		\end{equation}
	where $d$ is a positive function satisfying $\displaystyle\lim_{x\rightarrow+\infty} d(x)=(1-\ell_\bq)\frac{\E[N]^{\alpha}}{1-\E[X(1-\bq(X))]}=:\mathfrak{c}_N.$
	\end{lem}
	\begin{proof} 
	If $(B_i)_{i\in\mathbb N}$ is a family of iid random variables with Bernoulli distribution $\mathscr B(\tilde c)$, and assuming that this sequence is independent of $L$,  $\sum^{L}_{i=1}B_i$ equals $N$ in law and we define for all $k\in \mathbb N, \,\tilde{W}_k=\sum_{i=1}^kB_i$.\\ 
	Thus, for $\varepsilon>0$, we have:
		\begin{align}\label{decomppN}
		\bp_N(n)
			&=	\sum_{k=n}^{\left\lfloor\frac{n}{\tilde{c}}(1-\varepsilon)\right\rfloor}\displaystyle+\sum^{\left\lfloor\frac{n}{\tilde{c}}(1+\varepsilon)\right\rfloor}_{k=\left\lfloor\frac{n}{\tilde{c}}(1-\varepsilon)\right\rfloor+1}
			+\sum_{k>\left\lfloor\frac{n}{\tilde{c}}(1+\varepsilon)\right\rfloor}\p(L=k)\p(\tilde{W}_k=n)\nonumber\\
			&=:I_1+I_2+I_3.
		\end{align}
		Using Hoeffding's inequality:
		\begin{align*}
			I_1&\le  \p(\tilde{W}_{\left\lfloor\frac{n}{\tilde{c}}(1-\varepsilon)\right\rfloor}\ge n)\le \p(\tilde{W}_{\left\lfloor\frac{n}{\tilde{c}}(1-\varepsilon)\right\rfloor}-\left\lfloor\frac{n}{\tilde{c}}(1-\varepsilon)\right\rfloor \tilde c\ge n\varepsilon)\le e^{-2\tilde cn\varepsilon^2},
		\end{align*}
		and with a similar reasoning, we can show the existence of a positive constant $C$ such that for $n$ large enough:
		\begin{equation*}
			I_3\le Ce^{-2\tilde cn\varepsilon^2}.
		\end{equation*}
		In view of what we want to prove, $I_1$ and $I_3$ are negligible.\\
		Corollary \ref{RV} gives that for $n$ large enough and for all $k\in \llbracket \left\lfloor\frac{n}{\tilde{c}}(1-\varepsilon)\right\rfloor+1 ,\left\lfloor\frac{n}{\tilde{c}}(1+\varepsilon)\right\rfloor \rrbracket$: 
		\begin{align*}
		(1-\varepsilon)\left(\frac{\tilde c}{1+\varepsilon}\right)^{1+\alpha}
		\le\frac{\p(L=k)}{\p(L=n)}
		\le(1+\varepsilon)\left(\frac{\tilde c}{1-\varepsilon}\right)^{1+\alpha},
		\end{align*}
		and consequently, to study the asymptotic behavior of $I_2$, we just have to look after:
		\[\p(L=n)\tilde c^{\alpha+1}\sum^{\left\lfloor\frac{n}{\tilde{c}}(1+\varepsilon)\right\rfloor}_{k=\left\lfloor\frac{n}{\tilde{c}}(1-\varepsilon)\right\rfloor+1}\p(\tilde{W}_k=n)
		=:\p(L=n)\tilde c^{\alpha+1}\tilde I_2.\]
		The proof of the negligibility of $I_1$ and $I_3$ implies that:
		{\small\begin{align*}
		\lim_{n\rightarrow+\infty}\tilde I_2&= \lim_{n\rightarrow+\infty}\sum_{k\ge n}\p(\tilde W_k=n)- \sum^{\left\lfloor\frac{n}{\tilde{c}}(1-\varepsilon)\right\rfloor}_{k=n}\p(\tilde{W}_k=n)-\sum_{k>\left\lfloor\frac{n}{\tilde{c}}(1+\varepsilon)\right\rfloor}\p(\tilde{W}_k=n)\\
		&= \lim_{n\rightarrow+\infty}\sum_{k\ge n}\p(\tilde W_k=n)=\lim_{n\rightarrow+\infty}\sum_{k\ge n}\binom{k}{n}\tilde{c}^n\left(1-\tilde{c}\right)^{k-n} =\tilde{c}^{-1}.
		\end{align*}}
	%
		Finally, when $n$ goes to infinity
		\begin{equation*}
			\bp_N(n)
			\sim \tilde{c}^{\alpha}\p(L=n),
		\end{equation*}
		and we have:
		\begin{equation}
		\bp_N(n)=\dfrac{g(n)c(n)\Le(n)}{n^{1+\alpha}}
	\end{equation}
	where $c$ is defined in \eqref{probLalpha} and $g$ is a positive function satisfying $\displaystyle\lim_{x\rightarrow+\infty}g(x)=c^{\alpha}$. And we obtain our claim result as:
	\[\lim_{x\rightarrow+\infty}g(x)c(x)=\frac{(1-\ell_\bq)\E[L]^{1+\alpha}\tilde{c}^{\alpha}}{\tilde\bp(0)}=\frac{(1-\ell_\bq)\E[N]^{\alpha}}{1-\E[X(1-\bq(X))]}.\]
	\end{proof}

	Now, we can determine the asymptotic behavior of $\bp^{(1)}(n)=\p(Z^{(1)}=n)$ where $Z^{(1)}=\sum_{k=1}^{X^{(1)}}N_k$. Recalling that:
	\begin{equation*}
	\p(X^{(1)}=n)=\dfrac{ \Le(n)\bq(n)}{\E[\bq(X)]n^{1+\alpha}},
	\end{equation*}
	according to Theorem 1, (iii) of \cite{bloznelis_local_2019}, Lemmata \ref{SV} and \ref{lemprobNalpha}:
	\begin{align}\label{aleph}
		\p(Z^{(1)}=n)&\sim\E\left[X^{(1)}\right]\p(N=n)+\E\left[N\right]^{-1}\p(X^{(1)}=\left\lfloor \nicefrac{n}{\E\left[N\right]}  \right\rfloor )\nonumber\\
		&\sim \frac{\Le(n)\E[N]^{\alpha}}{n^{1+\alpha}\E[\bq[X]]}\left(\frac{\E[X\bq(X)]}{1-\E[X(1-\bq(X))]}(1-\ell_\bq)+\ell_\bq\right)\nonumber\\
		&\sim \frac{\Le(n)\E[N]^{\alpha}}{n^{1+\alpha}\E[\bq[X]]}\frac{\E[X\bq(X)]+\ell_\bq(1-\E[X])}{1-\E[X(1-\bq(X))]}=:\frac{\Le(n)}{n^{1+\alpha}}\aleph
	\end{align}
	and we have:
	\begin{equation}\label{probXhatalpha}
		\bp^{(1)}(n)=\p(Z^{(1)}=n)=\frac{h(n)\Le(n)}{n^{1+\alpha}}
	\end{equation}
	where $h$ is a positive function satisfying $\lim_{x\rightarrow+\infty}h(x)=\aleph$.\\
	Similarly, recalling that $Z^{(0)}=\sum_{k=1}^{X^{(0)}}N_k$ and that for $n\in\NN$,
	\[\p(X^{(0)}=n)=\dfrac{ \Le(n)(1-\bq(n))}{(1-\E[\bq(X)])n^{1+\alpha}},\] 
	when $n$ goes to infinity:
	{\small\begin{align}\label{cz0}
		\p(Z^{(0)}=n)&\sim\E[X^{(0)}] \p(N=n)+\E\left[N\right]^{-1}\p(X^{(0)}=\left\lfloor \nicefrac{n}{\E\left[N\right]} \right\rfloor )\nonumber\\
		&\sim \frac{\Le(n)}{n^{1+\alpha}}\frac{(1-\ell_\bq)\E[N]^{\alpha}}{(1-\E[\bq(X)])(1-\E[X(1-\bq(X))])}=: \frac{\Le(n)}{n^{1+\alpha}}(1-\ell_\bq)\mathfrak{c}_{Z^{(0)}}.
			\end{align}}
	and we have:
	\begin{equation}\label{probXbar}
		\bp^{(0)}(n)=\p(Z^{(0)}=n)=\frac{\tilde h(n)\Le(n)}{n^{1+\alpha}}
	\end{equation}
	where $\tilde h$ is a positive function satisfying $\lim_{x\rightarrow+\infty}\tilde h(x)=(1-\ell_\bq)\mathfrak{c}_{Z^{(0)}}$.
	
	\subsubsection{Case $\ell_\bq=1$}
	In that case, we suppose that, the mark function $\bq$ satisfies \eqref{defqlim1} for $k\in\NN$ , 
	\begin{equation}
		1-\bq(k)=k^{-\beta}\Lq(k) 
	\end{equation}
	with $\beta\ge 2$ and $\Lq$ is a SV function.\\
	In this case, we obtain:
	\[\p(\tilde{X}_{+}=n)=\dfrac{\Le(n)\Lq(n)}{n^{1+\alpha+\beta}(1-\tilde{\bp}(0))},\]
	and as the product of SV functions is a SV function, all the results obtained in subsection 5.2.1 can be adapted taking respectively 0 for $\ell_\bq$ and $\alpha+\beta$ for $\alpha$. Consequently, for all $n\in\NN$:
		\begin{equation}\label{probYalphazero}
			\bp_N(n)=\p(N=n)=\dfrac{\tilde d(n)\Le(n)\Lq(n)}{n^{1+\alpha+\beta}},
		\end{equation}
	where $d$ is a positive function satisfying $\lim_{x\rightarrow+\infty}\tilde d(x)=\frac{\E[N]^{\alpha+\beta}}{1-\E[X(1-\bq(X))]}.$

	To determine the asymptotic behavior of  $\bp^{(1)}(n)$, when $n$ goes to infinity, note that:
	\[\p(X^{(1)}=n)\sim\frac{\Le(n)}{n^{1+\alpha}\E[\bq(X)]},\]
	implying that $\p(N=n )\underset{n \to \infty}{=}o\left(\p(X^{(1)}=n)\right)$ and according to Theorem 6 of \cite{bloznelis_local_2019}:
	\begin{equation*}
		\p(Z^{(1)}=n)\sim\E\left[N\right]^{-1}\p(X^{(1)}=\left\lfloor \nicefrac{n}{\E\left[N\right]} \right\rfloor )
		\sim\dfrac{\Le(n)}{n^{1+\alpha}}\frac{\E[N]^{\alpha}}{\E[\bq(X)]}.
	\end{equation*}
	As a result:
	\begin{equation}\label{probXhatalphacase0}
		\bp^{(1)}(n)=\p(Z^{(1)}=n)=\frac{f(n)\Le(n)}{n^{1+\alpha}},
	\end{equation}
	where $f$ is positive function satisfying $\lim_{x\rightarrow+\infty}f(x)=\nicefrac{\E[N]^{\alpha}}{\E[\bq(X)]}$.\\
	For $Z^{(0)}$, note that 
	\[\p(X^{(0)}=n)=\dfrac{ \Le(n)\Lq(n)}{n^{1+\alpha+\beta}(1-\E[\bq(X)])},\]
	and the results is the same as the one in Sub-subsection \ref{caselneq1}, taking respectively 0 for $\ell_\bq$ and $\alpha+\beta$ for $\alpha$, in other words:
	\begin{equation}\label{probXbarcase0}
		\bp^{(0)}(n)=\p(Z^{(0)}=n)=\frac{ \tilde f(n)\Le(n)\Lq(n)}{n^{1+\alpha+\beta}}
	\end{equation}
	where $\tilde f$ is a positive function such that {\small$\displaystyle\lim_{x\rightarrow+\infty}f(x)$ $ =\frac{\E[N]^{\alpha+\beta}}{(1-\E[\bq(X)])(1-\E[X(1-\bq(X))])}$}. 
	
	\subsubsection{General results}
	For $S_n=\sum_{i=1}^n{Z^{(1)}_i}$, according to Theorem 2 in \cite{doney_local_2001}, since $\mu^{(1)}<+\infty$, $\sigma^2:=\mathrm{Var }[Z^{(1)}]<+\infty$ and $\bp^{(1)}$ being regularly varying at infinity with index $\kappa<-3$, uniformly in $k$ such that $\nicefrac{(k-n\mu^{(1)})}{\sqrt{n}}\to +\infty$, when $n$ goes to infinity, we have:
	\begin{equation}\label{expS_nequiv2}
		\p(S_n=k)=\dfrac{e^{-\frac{(k-n\mu^{(1)})^2}{2n\sigma^2}}}{\sigma\sqrt{2\pi n}}\left(1+o(1)\right)+n\bp^{(1)}(\lfloor k-n\mu^{(1)}\rfloor)(1+o(1)).
	\end{equation}
	
	\begin{lem}\label{limSnT1}
		Assume that $\bp^{(1)}$ is aperiodic , with $\mu^{(1)}<1$, $\rho(\bp^{(1)})=1$, $\mathrm{Var}[Z^{(1)}]<+\infty$ and $\bq$ satisfies \eqref{defqlim1} if $\ell_\bq=1$. If $T$ is a nondegenerate positive random variable in $L^1$, independent of $(S_n)_{n\in\mathbb N}$, such that:
		\[\forall k\in\mathbb N,\,\p(T=k)=\frac{\Le(k)}{k^{1+\alpha}}(1-\bq(k))\gamma(k),\] 
		where $\gamma$ is a positive function such $\lim_{x\rightarrow +\infty}\gamma(x)=\mathfrak{c}>0$, we have
		\begin{equation}\label{limTSn1}
			\underset{n\to+\infty}{\lim} \dfrac{\E\left[T\indic_{\{S_{n}+T=n\}}\right]}{\p(S_n=n)}= \E[T]+\left(1-\mu^{(1)}\right)\frac{(1-\ell_\bq)\mathfrak{c}}{\aleph},
		\end{equation}
		where $\aleph:=\frac{\E[N]^{\alpha}}{\E[\bq(X)]}\frac{\E[X\bq(X)]+\ell_\bq(1-\E[X])}{1-\E[X(1-\bq(X))]}$.
	\end{lem}	
	\begin{rqqq}
	Note that if $\ell_\bq=1$, the ratio in \eqref{limTSn1} tends to $\E[T]$.
	\end{rqqq}
	\begin{proof}
		Let $0<\varepsilon<1-\mu^{(1)}$, $\nicefrac{2}{3}<d<\nicefrac{3}{4}$, and for $n$ large enough, 
		we introduce these four intervals:
		\begin{align*}
			J_{1}&=\left[0, n(1-\mu^{(1)}-\varepsilon)\right],~J_2=\left] n(1-\mu^{(1)}-\varepsilon), n\left(1-\mu^{(1)}\right)-n^d\right],\\
			J_{3}&= \left]n\left(1-\mu^{(1)}\right)-n^d,n\left(1-\mu^{(1)}\right)+n^d\right],
			J_{4}=\left]n\left(1-\mu^{(1)}\right)+n^d,n\right],
		\end{align*}
		and for typographical simplicity, we introduce for $i\in\llbracket1,4\rrbracket$: 
		\[I_i:=\dfrac{\E\left[T\indic_{\{S_{n}+T=n, T\in J_i\}}\right]}{\p(S_n=n)}=\sum_{k\in J_i}k\p(T=k)\frac{\p(S_n=n-k)}{\p(S_n=n)}.\]
		On $J_1$, note that the exponential term is negligible in \eqref{expS_nequiv2}, and it is also the case for $k=n$ as $\mu^{(1)}<1$, as a result for $n$ large enough for all $k\in I_1$:
		\begin{equation}\label{ratio1}
		\frac{\p(S_n=n-k)}{\p(S_n=n)}\le 2\frac{\bp^{(1)}(\lfloor n\left(1-\mu^{(1)}\right)-k\rfloor)}{\bp^{(1)}(\lfloor n\left(1-\mu^{(1)}\right)\rfloor)}.
		\end{equation}
		Thus, using formula \eqref{probXhatalpha} or \eqref{probXhatalphacase0} and Corollary \ref{RV}, for $n$ large enough and all $k\in J_1$:
		\begin{align}\label{borne1}
		\frac{\p(S_n=n-k)}{\p(S_n=n)}
		\le 4\left(\frac{1-\mu^{(1)}}{\varepsilon}\right)^{\alpha+1}.
		\end{align}
		The strong ratio limit Theorem gives that for fixed $k$:
		\[\lim_{n\rightarrow+\infty}k\p(T=k)\frac{\p(S_n=n-k)}{\p(S_n=n)}=k\p(T=k).\]
		Consequently, as $T\in L^1$, Lebesgue’s dominated convergence theorem implies:
		\[\lim_{n\rightarrow+\infty}I_1=\lim_{n\rightarrow+\infty}\sum_{k\in J_1}k\p(T=k)\frac{\p(S_n=n-k)}{\p(S_n=n)}=\sum_{k\in \mathbb N}k\p(T=k)=\E[T].\]
		Note that on $J_2$, \eqref{ratio1} remains valid and using Remark \ref{rqRV}, there exists $C>0$ such that for $n$ large enough and all $k\in J_2$: 
		 \begin{align}
		\frac{\p(S_n=n-k)}{\p(S_n=n)}
		\le Cn^{(1-d)(\alpha+1)}\label{borne2}, 
		\end{align}
		and as $\p(T=k)=k^{-1-\alpha}\tilde \Le(k)$ where $\tilde \Le$ is a SV function, for any $\delta>0$, for $k$ large enough $\p(T=k)\le k^{\delta-1-\alpha}$, consequently for $n$ large enough and $0<\delta<\alpha$:
		 \begin{align*}
		I_2
		 \le  Cn^{(1-d)(\alpha+1)}\sum_{k\in J_2}\frac{1}{k^{\alpha-\delta}}
		 \le \frac{\varepsilon Cn^{2+\delta-d(\alpha+1)}}{(1-\mu^{(1)}-\varepsilon)^{\alpha-\delta}}=:\tilde Cn^{2+\delta-d(\alpha+1)}.
		 \end{align*}
		As $\alpha>2$ and $d>\nicefrac{2}{3}$, we can take $\delta$ small enough such that $2+\delta-d(\alpha+1)<0$ and it implies that
		$\lim_{n\rightarrow+\infty}I_2=0$.\\
In order to determine the limit of $I_3$, we have to distinguish if $\ell_\bq=1$ or not. 
		 \underline{Case  $\ell_\bq=1$}: Recall that when $n$ goes to infinity, there exists $ C_0>0$ such that $\p(S_n=n)\sim \nicefrac{C_0\Le(n)}{n^\alpha}$. 
		As a result, there exists two positive constants $C_1$ and $C_2$, such that for $n$ large enough, using Lemma \ref{SV}, for all $k\in J_3$:
		\begin{align*}
		\frac{k\p(T=k)}{\p(S_n=n)}&\le \frac{C_1\Lq(k)}{k^\beta}\frac{n^\alpha\Le(k)}{\Le(n)k^\alpha}
		\le C_2\frac{\Lq(k)}{k^\beta}.
		\end{align*}
		Consequently, for $n$ large enough:
		 \begin{equation*}
		 I_3\le \sum_{k\in J_3}\frac{k\p(T=k)}{\p(S_n=n)}\le C_2\sum_{k\in J_3}\frac{\Lq(k)}{k^{\beta}}=:C_2R_n,
		 \end{equation*}
		 and  $\lim_{n\rightarrow+\infty}R_n=0$, $\nicefrac{\Lq(k)}{k^{\beta}}$ being the general  term of a convergent series as $\Lq$ is a SV function and $\beta\ge 2$.\\
		 \underline{Case $\ell_\bq\neq1$:}
		We postpone the proof that $I_3$ has the same limit in infinity as:
		\[\tilde I_3:=\frac{\E\left[n\left(1-\mu^{(1)}\right)\indic_{\{S_{n}+T=n,T\in J_3\}}\right]}{\p(S_n=n)}.\]
		Note that, according to Corollary  \ref{RV} and  \eqref{probXhatalpha},  for $n$ large enough and all $k\in J_3$:
		\begin{equation}\label{ratio38}
		 (1-\varepsilon)\frac{(1-\ell_\bq)\mathfrak{c}}{n\aleph}	\le \frac{\p(T=k)}{\p(S_n=n)}\le (1+\varepsilon)\frac{(1-\ell_\bq)\mathfrak{c}}{n\aleph}.
		\end{equation}
		Moreover, as $d>\nicefrac{2}{3}$, thanks to the central limit theorem, when $n$ goes to infinity:
		\begin{equation}\label{TCL38}
		\sum_{k\in J_3}\p(S_n=n-k)=\p\left(\dfrac{S_n-n\mu^{(1)}}{\sqrt{n}}\in \left[-n^{d-\nicefrac{1}{2}},n^{d-\nicefrac{1}{2}}\right]\right)\rightarrow 1.
		\end{equation}
		This gives an upper bound for $\lim_{n\rightarrow+\infty}\tilde I_3$. Indeed: 
		 \begin{align*}
		\limsup_{n\rightarrow+\infty} \tilde{I_3}&=\limsup_{n\rightarrow+\infty}n\left(1-\mu^{(1)}\right)\sum_{k\in J_3}\frac{{\p(T=k)}}{\p(S_n=n)}\p(S_n=n-k)\\
		& \le (1+\varepsilon)\left(1-\mu^{(1)}\right)\frac{(1-\ell_\bq)\mathfrak{c}}{\aleph}\lim_{n\rightarrow+\infty}\p\left(S_n-n\mu^{(1)}\in [-n^d,n^d]\right)\\
		&=(1+\varepsilon)\left(1-\mu^{(1)}\right)\frac{(1-\ell_\bq)\mathfrak{c}}{\aleph}
		 \end{align*}
		And we can obtain the lower bound with the same reasoning. As a result:
		\[\lim_{n\rightarrow+\infty} \tilde{I_3}=\left(1-\mu^{(1)}\right)\frac{(1-\ell_\bq)\mathfrak{c}}{\aleph}\]
		Now, to prove that $I_3$ and $\tilde I_3$ have the same limit, note that \eqref{ratio38} ensures that there exists $C>0$ such that for all $k\in J_3$ and $n$ large enough:
		\begin{equation}\label{truc}
		\frac{\p(T=k)}{\p(S_n=n)}\le \frac{ C}{n},
		\end{equation}
		and as a result, as $d<\nicefrac{3}{4}$ and using \eqref{TCL38}, when $n$ goes to infinity:
		\begin{align*}
		|I_3-\tilde I_3|&\le \dfrac{\E\left[\left|  T-n\left(1-\mu^{(1)}\right)\right| \indic_{\{T\in J_3,S_{n}+T=n\}}\right]}{\p(S_n=n)}\\
		&\le \frac{2 C}{n^{1-d}}\p\left(S_n-n\mu^{(1)}\in \left[-n^d,n^d\right]\right)\le \frac{2C}{n^{1-d}}\rightarrow0. 
		\end{align*}
 It remains to prove that $I_4$ tends to 0 to obtain our result. Note that for $n$ large enough and $k\in J_4$, \eqref{truc} still true and we have
\[\frac{k\p(T=k)}{\p(S_n=n)}\le C, \] 
and thus, using again the central limit theorem, when $n$ goes to infinity:
\begin{align*}
I_4\le C\p\left(S_n-n\mu^{(1)}\le-n^d\right)\rightarrow0.
\end{align*}
	\end{proof}
	
	\subsection{Principal lemmas}
	In this subsection, we consider lemmata used in Theorems \ref{cvgkesten} and \ref{cvgcondtree}. \\
	Recall that $\mathfrak{c}_N$ and $\aleph$ are described in Lemma \ref{lemprobNalpha} and \eqref{aleph}. 
	\begin{lem}\label{limSnN3}
		Assume that $\bp^{(1)}$ is aperiodic , with $\mu^{(1)}=1$ or ($\mu^{(1)}<1$ and $\rho(\bp^{(1)})=1$), $\mathrm{Var }[Z^{(1)}]<+\infty$.  
		For all $m\in\N,$ $k\in\Z$, we have:
		\begin{equation}\label{limNSn3}
			\lim_{n\to+\infty}\dfrac{\E[N\indic_{\{S_n+N=n\}}]}{\E[\indic_{\{S_{n}+W_m+N=n-k\}}]}=
				{\E[N]+\left(1-\mu^{(1)}\right)\frac{(1-\ell_\bq)\mathfrak{c}_N}{\aleph}}.
		\end{equation}
		If $\bp^{(1)}$ is periodic, then \eqref{limNSn3} still holds along the sub-sequence for which the denominator is positive.
	\end{lem}
	\begin{rqqq}\label{rqqq1}
	\begin{enumerate}
	\item If $\mu^{(1)}=1$ or $\ell_\bq=1$, the ratio in \eqref{limNSn3} tends to $\E[N]$.
	\item Note that the way the limit is written in formula \eqref{limNSn3} helps streamline the proof. Nevertheless, in the remainder of this section, it will be more convenient to simplify the expression further by using that:
	\begin{align*}
	\frac{\mathfrak{c}_N}{\aleph}&=\frac{\E[\bq(X)]}{\E[X\bq(X)]+\ell_\bq(1-\mu(\bp))]},\,
			\E[N]=
			\frac{\E[\bq(X)]}{1-\mu(\bp)+\E[X\bq(X)]},\\
			\mu^{(1)}&=
			\frac{\E[X\bq(X)]}{1-\mu(\bp)+\E[X\bq(X)]},
		\end{align*}
		and we obtain: 
		\begin{equation}\label{limNSn3bis}
			\lim_{n\to+\infty}\dfrac{\E[N\indic_{\{S_n+N=n\}}]}{\E[\indic_{\{S_{n}+W_m+N=n-k\}}]}=
				\dfrac{\E[\bq(X)]}{\E[X\bq(X)]+\ell_\bq(1-\mu(\bp))}.
		\end{equation}
	\end{enumerate}
	\end{rqqq}
	\begin{proof}
	First, assume that there exists $a>0$, such that for all $k,m$ in $\mathbb Z$:
	\begin{equation}\label{assumption}
	\lim_{n\rightarrow+\infty}\dfrac{\E\left[\indic_{\{S_{n+m}+N=n-k\}}\right]}{\E\left[N\indic_{\{S_{n}+N=n\}}\right]}=a.
	\end{equation}
	 To mimic the proof of Lemma \ref{limSnN2} (see Lemma 8.6 of \cite{abraham_local_2014}), we introduce for all $j\in\mathbb Z$:
		 \begin{equation}\label{expcnjlimnN3}
			c_{n,j}:=\dfrac{\E\left[\indic_{\{S_{n}+N=n-k-j\}}\right]}{\E\left[N\indic_{\{S_{n}+N=n\}}\right]},
		\end{equation}
		and denote by $t:=(t(j))_{j\in\N}$ and $r:=(r(j))_{j\in\N}$ respectively the distributions of $W_m$ and $S_m$. Thanks to \eqref{assumption}, $\lim_{n\to+\infty}c_{n,j}=a$ and consequently:
		\[\lim_{n\to+\infty}\sum_{j\in\N}r(j)c_{n,j}=\lim_{n\to+\infty}\dfrac{\E\left[\indic_{\{S_{n+m}+N=n-k\}}\right]}{\E\left[N\indic_{\{S_{n}+N=n\}}\right]}=a=\sum_{j\in\N}r(j)\lim_{n\to+\infty}c_{n,j}.\]
		According to \eqref{condq}, there exists $j_0\in\NN$ such that $\bp(j_0)\bq(j_0)>0$ implying that $\p(X^{(1)}=j_0)>0$. As $Z^{(1)}_1=\displaystyle\sum_{i=1}^{X^{(1)}_1}N_i$, for $j\in\N$:
		\begin{align*}
			r(j)=\p(S_m=j)
			&\geq \p\left(\sum_{i=1}^{m}X^{(1)}_i=mj_0,\sum_{i=1}^{m_{j_0}}N_i=j\right)\\
			&\geq \p\left(\sum_{i=1}^{m}X^{(1)}_i=mj_0,W_m=j,\sum_{i=m+1}^{m_{j_0}}N_i=0\right)\geq C t(j),
		\end{align*}
		with $C>0$, independent of $j$. Thus, as $t(j)\leq \nicefrac{r(j)}{C}$ for all $j\in\N$, by dominated convergence theorem in \cite{kallenberg_foundations_1997} (Theorem 1.21), we deduce that 
		\[\lim_{n\to+\infty}\dfrac{\E\left[\indic_{\{S_{n}+W_m+N=n-k\}}\right]}{\E\left[N\indic_{\{S_{n}+N=n\}}\right]}=\lim_{n\to+\infty}\sum_{j\in\N}t(j)c_{n,j}=\sum_{j\in\N}t(j)\lim_{n\to+\infty}c_{n,j}=a.\]
	As a result, in order to prove  \eqref{limNSn3}, we have to obtain \eqref{assumption} but,  according to Lemma \ref{limSnN2}:
	\begin{align*}
			\lim_{n\rightarrow+\infty}\dfrac{\E\left[N\indic_{\{S_{n}+N=n\}}\right]}{\E\left[\indic_{\{S_{n+m}+N=n-k\}}\right]}&=\lim_{n\rightarrow+\infty}\dfrac{\E\left[N\indic_{\{S_{n}+N=n\}}\right]}{\p(S_n=n)}\dfrac{\p(S_n=n)}{\E\left[\indic_{\{S_{n+m}+N=n-k\}}\right]}\\
			&=\lim_{n\rightarrow+\infty}\dfrac{\E\left[N\indic_{\{S_{n}+N=n\}}\right]}{\p(S_n=n)}=:\lim_{n\rightarrow+\infty}R_n,
	\end{align*}
	and it suffices to study the behavior at infinity of $R_n$.\\
	The case $\mu^{(1)}<1$ is straightforward by applying Lemma \ref{limSnT1} with $T=N$.\\
		To obtain the case $\mu^{(1)}=1$, 
		we denote by $Z_n:= \displaystyle S_n-n =\sum_{i=1}^n{(Z^{(1)}_i-1)}$, a centered random walk such that $\E[Z_n]=0$. We have:
		\begin{equation*}
			R_n=\dfrac{\E\left[N\indic_{\{Z_n=-N\}}\right]}{\p(Z_n=0)}= \sum_{k\ge0}k\bp_N(k)\dfrac{\p(Z_n=-k)}{\p(Z_n=0)}.
		\end{equation*}
		According to 
		\cite{durrett_probability_1996} (Theorem 5.2 p.132),
		\[\sup_{k\in \mathbb Z}\left|  \sqrt{n}\p(Z_n=k)-\frac{e^{-\frac{k^2}{2n\sigma^2}}}{\sigma\sqrt{2\pi}}\right|\] 
		tends to zero when $n$ goes to infinity.
		Thus for all $k\in\mathbb Z$ we have:
		\begin{equation}\label{eqvcentradmwalk}
			\p(Z_n=k)=\dfrac{e^{-\frac{k^2}{2n\sigma^2}}}{\sigma\sqrt{2\pi n}}+o\left(\dfrac{1}{\sqrt{n}}\right),
		\end{equation}
		with a uniform error term for all $k\in \mathbb Z$. Consequently, there exists a positive constant $C$ such that: 
		\begin{align}\label{bornecritique}
		\sup_{k\in\mathbb Z}\frac{\p(Z_n=k)}{\p(Z_n=0)}\le C\mbox{ and } \lim_{n\rightarrow+\infty}\frac{\p(Z_n=k)}{\p(Z_n=0)}=1
		\end{align}
		for all $k\in\mathbb Z$, and with Lebesgue’s dominated convergence theorem:
		\begin{equation*}\label{limSnN31}
			\lim_{n\rightarrow+\infty}R_n=\sum_{k\ge 0}k\bp_N(k)\lim_{n\rightarrow+\infty}\dfrac{\p(Z_n=-k)}{\p(Z_n=0)}=\E[N].
		\end{equation*}
	\end{proof}
	
	\begin{lem}\label{limSnN4}
		Assume that $\bp^{(1)}$ is aperiodic , with $\mu^{(1)}=1$ or ($\mu^{(1)}<1$ and $\rho(\bp^{(1)})=1$), $\mathrm{Var}[Z^{(1)}]<+\infty$. 
		For all $m\in\N,$ $k\in\Z$, we have:
		\begin{equation}\label{limNSn4}
			\underset{n\to+\infty}{\lim}\dfrac{\E[N\indic_{\{S_{n}+W_{m}+Z^{(0)}+N=n-k\}}]}{\E[N\indic_{\{S_{n}+N=n\}}]}=1
		\end{equation}
		If $\bp^{(1)}$ is periodic, then \eqref{limNSn4} still holds along the sub-sequence for which the denominator is positive.
	\end{lem}
	\begin{proof}First, writing the ratio from \eqref{limNSn4} in the form: 
		\begin{equation*} 
			\dfrac{\E[N\indic_{\{S_{n}+W_{m}+Z^{(0)}+N=n-k\}}]}{\p(S_n=n)}\dfrac{\p(S_n=n)}{\E\left[N\indic_{\{S_{n}+N=n\}}\right]}=:v_nw_n,
		\end{equation*}
one can note that Lemma \ref{limSnN3} gives that $w_n$ tends to an explicit constant  $w$, as a result we just have to show that $\lim_{n\rightarrow+\infty}v_n=w^{-1}$.\\
		\underline{Case $\mu^{(1)}=1$:}  		
		Here, it suffices to prove that $\lim_{n\rightarrow+\infty}v_n=\E[N]$. \\
		Using the same notation as Lemma \ref{limSnN3} for $Z_n$, we have:
		\begin{align*}
			v_n=\dfrac{\E[N\indic_{\{Z_n+W_{m}+Z^{(0)}+N=-k\}}]}{\p(Z_n=0)}
			=\sum_{i\in\mathbb Z}\E[N\indic_{\{W_{m}+Z^{(0)}+N=i-k\}}]\dfrac{\p(Z_n=-i)}{\p(Z_n=0)}.
			\end{align*}
		Consequently, \eqref{bornecritique} permits us to apply Lebesgue’s dominated convergence theorem, so:
		\[\lim_{n\rightarrow+\infty}v_n=\sum_{i\in\mathbb Z}\E[N\indic_{\{W_{m}+Z^{(0)}+N=i-k\}}]=\E[N].\]

		\noindent\underline{Case $\mu^{(1)}<1$:}
		Like in Lemma \ref{limSnN3}, we can prove that for all $k,m\in\Z$:
		\begin{equation*}
		\lim_{n\rightarrow+\infty}v_n=\lim_{n\rightarrow+\infty}\dfrac{\E[N\indic_{\{S_{n+m}+Z^{(0)}+N=n-k\}}]}{\p(S_n=n)}=:\lim_{n\rightarrow+\infty}z_n,
		\end{equation*}
		and  we  write:
		\begin{equation*}
		z_n=\dfrac{\E[N\indic_{\{S_{n+m}+Z^{(0)}+N=n-k\}}]}{\E[N\indic_{\{S_{n}+Z^{(0)}+N=n\}}]} \dfrac{\E[N\indic_{\{S_{n}+Z^{(0)}+N=n\}}]}{\p(S_{n}=n)}=:a_nb_n.
		\end{equation*}
		According to Lemma \ref{limSnTR2}, with $T=N$ and $R=Z^{(0)}$, $\lim_{n\rightarrow+\infty}a_n=1$ and as a result, we need to determine the limit of $b_n$ and distinguish between the cases depending on whether $\ell_\bq=0$ or not.\\ 
		\underline{\textbf{Case $\ell_\bq\neq0$:}} 
		Our strategy here is to mimic the proof of the Lemma \ref{limSnN2}, in other words, we search $(\beta_j)_{j\in \mathbb N}$ such that for all $j\in\mathbb N$: 
		\begin{equation*}
		\p(Z^{(0)}=j)\le \beta_j, 
		\end{equation*}
		and satisfying:
		\begin{equation}\label{TCD}
		\lim_{n\rightarrow+\infty}\sum_{j\in\mathbb N}\beta_j c_{n,j}=\sum_{j\in\mathbb N}\beta_j\lim_{n\rightarrow+\infty}c_{n,j},
		\end{equation}
		where:
		\[c_{n,j}=\dfrac{\E\left[N\indic_{\{S_{n}+N=n-j\}}\right]}{\p(S_{n}=n)}.\]
		Thanks to Lemma \ref{limSnT1} with $T=N$, for all $j\in\mathbb N$, $\lim_{n\rightarrow+\infty}c_{n,j}=w^{-1}$, thus, by applying the dominated convergence theorem in \cite{kallenberg_foundations_1997} (Theorem 1.21), we obtain the claimed result as:
		\begin{align*}
		\lim_{n\rightarrow+\infty}b_n=	\lim_{n\rightarrow+\infty}\sum_{j\in\mathbb N}\p(Z^{(0)}=j)c_{n,j}=\sum_{j\in\mathbb N}\p(Z^{(0)}=j)\underset{n\to+\infty}{\lim}c_{n,j}=w^{-1}.
		\end{align*}
		To find $(\beta_j)_{j\in\mathbb N}$, first note that since $\lim_{n\to+\infty}\bq(n)=\ell_\bq> 0$, there exists $k_0$ such that for all $k\ge k_0$, $\bq(k)\ne0$, and $c$ defined by:
		\[c= \frac{\E[\bq(X)]}{1-\E[\bq(X)]}\underset{k\in \N, \bq(k)>0}{\max}\left\{\frac{1-\bq(k)}{\bq(k)}\right\},\]
		is finite, and one can write for all $k\in\N,~\bq(k)>0$:
		\[\p(X^{(0)}=k)=\frac{\bp(k)(1-\bq(k))}{1-\E[\bq(X)]}\leq c \frac{\bp(k)\bq(k)}{\E[\bq(X)]}=c\p(X^{(1)}=k).\]
		Consequently, for all $j\in\N$:
		\begin{align*}
			\p(Z^{(0)}=j)&=\p\left(\sum_{i=1}^{X^{(0)}}N_i=j\right)
			= \sum_{k=1}^{+\infty}\p\left(W_k=j\right)\p(X^{(0)}=k)\\
			&\le \sum_{k=1,\bq(k)=0}^{k_0}\p\left(W_k=j\right)\dfrac{\bp(k)}{1-\E[\bq(X)]}+c\p(Z^{(1)}=j)\\
			&\le\dfrac{1}{1-\E[\bq(X)]} \sum_{k=1,\bq(k)=0}^{k_0}\p\left(W_k=j\right)+c\bp^{(1)}(j)=:\beta_j.
		\end{align*}
In one hand, thanks to Lemmata \ref{limSnTR2} and \ref{limSnT1} with $T=N$ and $R=0$, we have for all $j\in\mathbb N$:
		\[\lim_{n\rightarrow +\infty}\sum_{j\in\mathbb N}\bp^{(1)}(j)c_{n,j}=\lim_{n\rightarrow +\infty}\dfrac{\E\left[N\indic_{\{S_{n+1}+N=n\}}\right]}{\p(S_{n}=n)}=w^{-1},\]
		implying that, as  $\lim_{n\rightarrow+\infty}c_{n,j}=w^{-1}$:
		\begin{equation}\label{first}
			\lim_{n\rightarrow +\infty}\sum_{j\in\mathbb N}\bp^{(1)}(j)c_{n,j}
			=w^{-1}=\sum_{j\in\mathbb N}\bp^{(1)}(j)\lim_{n\rightarrow +\infty}c_{n,j}.
		\end{equation}
		On the other hand, for all $k\in\llbracket0,k_0\rrbracket$ such that $\bq(k)=0$, we write:
		\begin{align*}
		\sum_{j\in\mathbb N}\p\left(W_k=j\right)c_{n,j}&=\dfrac{\E\left[N\indic_{\{S_{n}+W_k+N=n\}}\right]}{\p(S_{n}=n)}\\
                                         &=\dfrac{\E\left[N\indic_{\{S_{n}+W_k+N=n\}}\right]}{\E\left[N\indic_{\{S_{n}+N=n\}}\right]}\dfrac{\E\left[N\indic_{\{S_{n}+N=n\}}\right]}{\p(S_{n}=n)}=:r_ns_n.
		\end{align*}
		Thanks to Lemma \ref{limSnN2}, $\lim_{n\rightarrow+\infty}r_n=1$  and according to Lemma \ref{limSnT1} with $T=N$, $\lim_{n\rightarrow+\infty}s_n=w^{-1}$. Then:
		\begin{equation}\label{second}
			\lim_{n\rightarrow+\infty}\sum_{j\in\mathbb N}\p\left(W_k=j\right)c_{n,j}=w^{-1}=\sum_{j\in\mathbb N}\p\left(W_k=j\right)\lim_{n\rightarrow+\infty}c_{n,j}.
		\end{equation}
		\eqref{first} and \eqref{second} imply that $(\beta_j)_{j\in\mathbb N}$ satisfies \eqref{TCD} and we can conclude.  
		\underline{\textbf{Case $\ell_\bq=0$:}}
		As in Lemma \ref{limSnN3}, to obtain $\lim_{n\rightarrow+\infty} b_n$, we need to distinguish between several cases. To that end, we introduce $0<\varepsilon<1-\mu^{(1)}$, $\max\left(\nicefrac{2}{3}, \nicefrac{(1+\alpha)}{(2\alpha)}\right)<d<\nicefrac{3}{4}$, and for $n$ large enough, we introduce the same four  intervals:
		\begin{align*}
			J_{1}&=\left[0, n(1-\mu^{(1)}-\varepsilon)\right],~J_2=\left] n(1-\mu^{(1)}-\varepsilon), n\left(1-\mu^{(1)}\right)-n^d\right],\\
			J_{3}&= \left]n\left(1-\mu^{(1)}\right)-n^d,n\left(1-\mu^{(1)}\right)+n^d\right],
			J_{4}=\left]n\left(1-\mu^{(1)}\right)+n^d,n\right],
		\end{align*}
		and for typographical simplicity, we introduce for $i\in\llbracket1,4\rrbracket$: 
		\[I_i:=\dfrac{\E\left[N\indic_{\{S_{n}+Z^{(0)}+N=n, Z^{(0)}+N\in J_i\}}\right]}{\p(S_n=n)}.\]
		Writing the equality in distribution 
		\[N+Z^{(0)}=\sum_{i=1}^{X^{(0)}+1}N_i,\] 
		we can apply the same method as \eqref{probXbar}, and we obtain:
		\begin{equation}\label{probN+Z0}
			\p(N+Z^{(0)}=n)=h(n)\frac{\Le(n)(1-\bq(n))}{n^{1+\alpha}},
		\end{equation}
		where $h$ is a positive function such that 
		\[\lim_{x\rightarrow+\infty}h(x)=\frac{\E[N]^{\alpha}(2-\E[\bq(X)])}{(1-\E[\bq(X)])(1-\E[X(1-\bq(X))])}.\] 
		Then, with a similar reasoning as the one of the proof of Lemma \ref{limSnT1} with $T=N+Z^{(0)}$, we obtain for $i\in\lbrace 2,4\rbrace$ :
		\begin{equation*}
			I_i\le \dfrac{\E[\left(N+Z^{(0)}\right)\indic_{\{S_{n}+Z^{(0)}+N=n, Z^{(0)}+N\in J_i\}}]}{\p(S_{n}=n)}\underset{n\to+\infty}{\longrightarrow} 0,
		\end{equation*}
                 and as $N\in L^{1}$, \eqref{borne1} and the strong ratio limit Theorem allow us to apply Lebesgue’s dominated convergence theorem and:
                 \begin{align*}
			\lim_{n\rightarrow+\infty}I_1&=\sum_{k\in J_1}\E[N\mathds{1}_{\{Z^{(0)}+N=k\}}]\frac{\p(S_{n}=n-k)}{\p(S_n=n)}\\
			&=\sum_{k\in \mathbb N}\E[N\mathds{1}_{\{Z^{(0)}+N=k\}}]=\E[N].
					\end{align*}
		To study the limit of $I_3$, we need more information on $N$ and, for this purpose, we introduce three new intervals: 
		\begin{align*}
		K_1&=\left[0,n^d\right], K_2=\left[n^d+1,n\left(1-\mu^{(1)}\right)-n^d-1\right], \\
		K_3&=\left[n\left(1-\mu^{(1)}\right)-n^d,n\left(1-\mu^{(1)}\right)+n^d\right],
		\end{align*}
		and we define for all $i\in \lbrace 1,2,3\rbrace$:
		\[I_{3,i}=\frac{\E\left[N\indic_{\{S_{n}+Z^{(0)}+N=n, Z^{(0)}+N\in J_3,N\in K_i\}}\right]}{\p(S_n=n)}.\]
		Recalling \eqref{probXbar}
		\begin{equation}
		\bp^{(0)}(n)=\p(Z^{(0)}=n)=\frac{\tilde h(n)\Le(n)}{n^{1+\alpha}}
	\end{equation}
	where $\tilde h$ satisfies that $\lim_{x\rightarrow+\infty}\tilde h(x)=\mathfrak b$. 
		If $N\le n^d$, according to Corollary \ref{RV}, there exists $C>0$, such that for $n$ large enouch and  for all $k\in J_3-N$: 
		\[\frac{\p(Z^{(0)}=k)}{\p(S_n=n)}\le \frac{C}{n}.\]
		As a result, for $n$ large enough: 
		\begin{equation*}
			I_{3,1}\le \frac{C}{n} \sum_{j\le n^d}j\p(N=j)\sum_{k\in J_3-j}\p(S_n=n-k-j)
						\le\frac{C}{n}\E[N]\rightarrow0.
			\end{equation*}
		

		To study $I_{3,2}$, we still need further refinement, in other words we decompose $J_3$ in three intervals:
		\begin{align*}
		J_{3}^1&=\left[n\left(1-\mu^{(1)}\right)-n^d,n\left(1-\mu^{(1)}\right)-\varepsilon n^d\right]\\
		J_{3}^2&=\left]n\left(1-\mu^{(1)}\right)-\varepsilon n^d,n\left(1-\mu^{(1)}\right)+\varepsilon n^d\right]\\
		J_{3}^3&=\left]n\left(1-\mu^{(1)}\right)+\varepsilon n^d,n\left(1-\mu^{(1)}\right)+ n^d\right]\\
		\end{align*}
		and we define for all $i\in \lbrace 1,2,3\rbrace$:
		\[I_{3,2}^i=\frac{\E\left[N\indic_{\{S_{n}+Z^{(0)}+N=n, Z^{(0)}+N\in J_3^i,N\in K_2\}}\right]}{\p(S_n=n)}.\]
		First, note that if $k\in J_3^1$ then $n-k\in [n\mu^{(1)}+\varepsilon n^d, n\mu^{(1)}+n^d]$, and as result for any $\delta>0$, according to Corollary \ref{RV}, there exists $C>0$ such that for $n$ large enough and all $k\in J_3^1$: 
		\begin{align*}
		\frac{\p(S_n=n-k)}{\p(S_n=n)}\le \frac{C}{n^{(d-1)(\alpha+1+\delta)}}.
		\end{align*}
		As a result, for $n$ large enough: 
		\begin{align*}
		I_{3,2}^1&\le  \frac{C}{n^{(d-1)(\alpha+1+\delta)}}\sum_{k\in J_3^1}\E\left[N\indic_{\{ Z^{(0)}+N=k,N\in K_2\}}\right]\\
		&\le\frac{C}{n^{(d-1)(\alpha+1+\delta)}}\sum_{j\in K_2}j\bp_N(j)\le\frac{C}{n^{(d-1)(\alpha+1+\delta)}}\sum_{j\ge n^d}\frac{\Le(j)}{j^{\alpha}}\\
		&\le \frac{C}{n^{(d-1)(\alpha+1+\delta)}}\sum_{j\ge n^d}\frac{1}{j^{\alpha-\delta}}\le \frac{C}{n^{(d-1)(\alpha+1+\delta)}}\times \frac{1}{n^{d(\alpha-\delta-1)}}\\
		&\le \frac{C}{n^{\alpha(2d-1)-1-\delta}}\underset{n\rightarrow+\infty}{\rightarrow}0, 
		\end{align*}
for $\delta$ small enough, as $\alpha(2d-1)-1>0$ since $\alpha>2$ and $\max\left(\nicefrac{2}{3}, \nicefrac{(1+\alpha)}{(2\alpha)}\right)<d$.\\
In the second case, the reasoning is very similar; indeed for any $\delta>0$, there exists $C>0$ such that for $n$ large enough and all \\$k\in \left[n^d(1-\varepsilon),n\left(1-\mu^{(1)}\right)-n^d(1-\varepsilon)\right]$:
\begin{align*}
	\dfrac{\bp^{(0)}(k)}{\p(S_n=n)}
	&\le \frac{C}{n^{d(\alpha+1)-\alpha+\delta(d-1)}}.
\end{align*}
Then, as $n^d(1-\varepsilon)\le Z^{(0)}\le n\left(1-\mu^{(1)}\right)-n^d(1-\varepsilon)$ a.s.,  we obtain:
\begin{align*}
I_{3,2}^2\le  \frac{C}{n^{d(\alpha+1)-\alpha+\delta(d-1)}}\sum_{j\in K_2}j\bp_N(j)\le \frac{C}{n^{\alpha(2d-1)-\delta}}\underset{n\rightarrow+\infty}{\rightarrow}0.
\end{align*}
In the last case, note that for a $n$ large enough, for all $k$ satisfying \\$- n^d\le k-n\mu^{(1)}\le  -\varepsilon n^d$, according to \eqref{expS_nequiv2} 
\begin{equation*}
\p(S_n=k)\le C \frac{e^{-\frac{\varepsilon^2n^{2d-1}}{2\sigma^2}}}{\sqrt{n}}.
\end{equation*}
Consequently, as $- n^d\le S_n-n\mu^{(1)}\le  -\varepsilon n^d$ in $I_{3,2}^3$, we have:
\begin{align*}
I_{3,2}^3\le C \frac{e^{-\frac{\varepsilon^2n^{2d-1}}{2\sigma^2}}}{\sqrt{n}\p(S_n=n)}\E[N]\underset{n\rightarrow+\infty}{\rightarrow}0.
\end{align*}

		Finally, for $I_{3,3}$, we use the same method as Lemma \ref{limSnT1} for $J_3$ in the case where $\ell_\bq\neq1$. 
	We first study the limit of:
		\[\tilde I_{3,3}:=\dfrac{\E[n\left(1-\mu^{(1)}\right)\indic_{\{S_{n}+Z^{(0)}+N=n, Z^{(0)}+N\in J_3, N\in K_3\}}]}{\p(S_{n}=n)}.\]
		Let $\varepsilon>0$, according to Lemma \ref{lemprobNalpha} and Corollary \ref{RV}, for $n$ large enough and every $k\in K_3$: 
		\begin{equation}\label{BorneN}(1-\varepsilon)\frac{\mathfrak{c}_N}{n\aleph}\le  \frac{\p(N=k)}{\p(S_n=n)}\le (1+\varepsilon)\frac{\mathfrak{c}_N}{n\aleph}.
		\end{equation}
		Thus, for $n$ large enough:
		\begin{align*}
		\tilde I_{3,3}
		&\le  (1+\varepsilon)\frac{(1-\mu^{(1)})\mathfrak{c}_N}{n\aleph}\sum_{j=0}^{2n^d}\p(Z^{(0)}=j)\sum_{k\in K_3}\p(S_n=n-k-j)\\
		&\le (1+\varepsilon)\frac{(1-\mu^{(1)})\mathfrak{c}_N}{n\aleph}\p(Z^{(0)}\in[0,2n^d])\p(S_n-n\mu^{(1)}\in [-3n^d,n^d])
		\end{align*}
		and:
		\begin{align*}
		\tilde I_{3,3}
		&\ge  (1-\varepsilon)\frac{(1-\mu^{(1)})\mathfrak{c}_N}{n\aleph}\sum_{j=0}^{\frac{n^d}{2}}\p(Z^{(0)}=j)\sum_{k\in K_3}\p(S_n=n-k-j)\\
		&\ge (1-\varepsilon)\frac{(1-\mu^{(1)})\mathfrak{c}_N}{n\aleph}\p(Z^{(0)}\in [0,\nicefrac{n^d}{2}])\p(S_n-n\mu^{(1)}\in [-n^d,\nicefrac{n^d}{2}]).
		\end{align*}
		As our upper and lower bounds are true for every $\varepsilon>0$, using again the central limit Theorem, we obtain:
		 \[\lim_{n\rightarrow+\infty}\tilde I_{3,3}=\left(1-\mu^{(1)}\right)\frac{\mathfrak{c}_N}{\aleph}.\]
		 We can conclude as $\lim_{n\rightarrow+\infty}I_{3,3}=\lim_{n\rightarrow+\infty}\tilde I_{3,3}$, following the same reasoning for $I_3$ in  Lemma \ref{limSnT1}, thanks to the upper bound in \eqref{BorneN}.\\
		Finally when $\ell_\bq=0$:
		\begin{align*}
			\underset{n\to+\infty}{\lim}\dfrac{\E[N\indic_{\{S_{n}+Z^{(0)}+N=n\}}]}{\p(S_{n}=n)}
			&=\E[N]+\left(1-\mu^{(1)}\right)\frac{\mathfrak{c}_N}{\aleph}.
		\end{align*}
		We can write for all $\ell_\bq\in[0,1]$:
		\begin{equation*}
			\underset{n\to+\infty}{\lim}\dfrac{\E[N\indic_{\{S_{n}+Z^{(0)}+N=n\}}]}{\p(S_{n}=n)}=\E[N]+\left(1-\mu^{(1)}\right)\frac{(1-\ell_\bq)\mathfrak{c}_N}{\aleph}.
		\end{equation*}
		That concludes the proof of the lemma in the second case.
	\end{proof}
	Recall that $\mathfrak{c}_{Z^{(0)}}$ is explicit and given in \eqref{cz0}.
	\begin{lem}\label{limSnN5}
		Assume that $\bp^{(1)}$ is aperiodic , with $\mu^{(1)}=1$ or ($\mu^{(1)}<1$ and $\rho(\bp^{(1)})=1$), $\mathrm{Var}[Z^{(1)}]<+\infty$. 
		For all $m\in\N,$ $k\in\Z$, we have:
		\begin{equation}\label{limNSn5}
			\lim_{n\to+\infty}	\dfrac{\E[Z^{(0)}\indic_{\{S_{n}+W_{m}+Z^{(0)}+N=n-k\}}]}{\E[N\indic_{\{S_{n}+N=n\}}]}=\dfrac{\E[Z^{(0)}]+\left(1-\mu^{(1)}\right)(1-\ell_\bq)\frac{\mathfrak{c}_{Z^{(0)}}}{\aleph}}{\E[N]+\left(1-\mu^{(1)}\right)(1-\ell_\bq)\frac{\mathfrak{c}_{N}}{\aleph}}.
		\end{equation}
		If $\bp^{(1)}$ is periodic, then \eqref{limNSn5} still holds along the sub-sequence for which the denominator is positive.
	\end{lem}
	\begin{rqqq}\label{rqqq2}
	\begin{enumerate}
	\item If $\mu^{(1)}=1$ or $\ell_\bq=1$, the ratio in \eqref{limNSn5} tends to $\nicefrac{\E[Z^{(0)}]}{\E[N]}$.
	\item  As in Remark \ref{rqqq1}, we can considerably simplify \eqref{limNSn5} by noting that:
	\begin{align*}
	\frac{\mathfrak{c}_{Z^{(0)}}}{\aleph}=\frac{1}{1-\E[\bq(X)]},\,\E[Z^{(0)}]=\frac{\E[X(1-\bq(X))]}{1-\E[\bq(X)]}\E[N],
	\end{align*}
	and we obtain:
	\begin{equation}\label{limNSn5bis}
			\lim_{n\to+\infty}	\dfrac{\E[Z^{(0)}\indic_{\{S_{n}+W_{m}+Z^{(0)}+N=n-k\}}]}{\E[N\indic_{\{S_{n}+N=n\}}]}=\frac{\E[X(1-\bq(X))]+(1-\ell_\bq)\left(1-\mu(\bp)\right)}{1-\E[\bq(X)]}.
		\end{equation}

	\end{enumerate}
	\end{rqqq}
	\begin{proof}This proof is very similar to those of Lemmata \ref{limSnN3} and \ref{limSnN4}, and as a result, we begin by writing the ratio on the left-hand side of \eqref{limNSn5} in the form: 
		\begin{equation*} 
			\dfrac{\E[Z^{(0)}\indic_{\{S_{n}+W_{m}+Z^{(0)}+N=n-k\}}]}{\p(S_n=n)}\dfrac{\p(S_n=n)}{\E\left[N\indic_{\{S_{n}+N=n\}}\right]}=:v_nw_n.
		\end{equation*}
Lemma \ref{limSnN3} gives that $w_n$ tends to the denominator of the right-hand side of the equation \eqref{limNSn5}, and consequently we just have to show that $v_n$ tends to the numerator that we denote by $\mathscr N$ for typographical simplicity.\\
		The roles of $N$ and $Z^{(0)}$ are completely symmetric, thus following the proof of Lemma \ref{limSnN4}, we easily obtain that $\lim_{n\rightarrow+\infty}v_n=\E[Z^{(0)}]$, if $\mu^{(1)}=1$,which completes this case, and if $\mu^{(1)}<1$, for all $k,m\in\Z$, we have:
		\begin{align*}
			\lim_{n\rightarrow+\infty}v_n&=\lim_{n\rightarrow+\infty}\dfrac{\E[Z^{(0)}\indic_{\{S_{n}+Z^{(0)}+N=n\}}]}{\p(S_n=n)}\\
			&=\lim_{n\rightarrow+\infty}\sum_{j\in\mathbb N}\bp_N(j)\dfrac{\E[Z^{(0)}\indic_{\{S_{n}+Z^{(0)}=n-j\}}]}{\p(S_n=n)}\\
			&=\lim_{n\rightarrow+\infty}\sum_{j\in\mathbb N}\bp_N(j)d_{n,j}.
		\end{align*}
		Thanks to \eqref{limTRSn2} and \eqref{limTSn1}, with $T=Z^{(0)}$ and $R=0$, we have $\lim_{n\rightarrow+\infty}d_{n,j}=\mathscr N$ and:
		\begin{align*}
			\lim_{n\rightarrow+\infty}\sum_{j\in\N}\bp^{(1)}(j)d_{n,j}&=\lim_{n\rightarrow+\infty}\dfrac{\E[Z^{(0)}\indic_{\{S_{n+1}+Z^{(0)}=n\}}]}{\p(S_{n}=n)}
			\\&=\mathscr N=\sum_{j\in\N}\bp^{(1)}(j)\lim_{n\rightarrow+\infty}d_{n,j}.
		\end{align*}
		And we have already proved in Lemma \ref{limSnN2}, that there exists $C>0$ such that for all $j\in\mathbb N,\, C\bp_N(j)\le \bp^{(1)}(j)$.
Using again Theorem 1.21 in \cite{kallenberg_foundations_1997}:
\[\lim_{n\rightarrow+\infty}\sum_{j\in\mathbb N}\bp_N(j)d_{n,j}=\sum_{j\in\mathbb N}\bp_N(j)\lim_{n\rightarrow+\infty}d_{n,j}=\mathscr N.\]
		That concludes the proof in the case $\mu^{(1)}<1$.
	\end{proof}
	Let a marked tree $\bt^{*}\in \T_0^*$, $x\in \bt$ and  recall the expression \eqref{Bni}, for all $i\in\Z$ and $n\in\NN$:
	\[B_{n,i}(x)=\underset{j>i}{\sum}\bp(j)\alpha_{j,x}(j-i)a_{n,j},\]
	where:
	\[\alpha_{j,x}=\bq(j)\eta_x(\bt)+(1-\bq(j))(1-\eta_x(\bt)) \mbox{ and }a_{n,j}=\dfrac{\E[N\indic_{\{S_{n}+W_{j-1-i}+N=n\}}]}{\E[N\indic_{\{S_n+N=n\}}]}.\]
	\begin{lem}\label{limBn,lneg}
		Assume that $\bp^{(1)}$ is aperiodic, with $\mu^{(1)}=1$ or ($\mu^{(1)}<1$ and $\rho(\bp^{(1)})=1$), $\mathrm{Var}[Z^{(1)}]<+\infty$. 
		For all nonpositive $l$, we have:
		\[\lim_{n\to+\infty}B_{n,l}(x)=\E[(X-l)\alpha_{X,x}]+(1-\mu(\bp))\left(\ell_\bq\eta_x(\bt)+\left(1-\ell_\bq\right)\left(1-\eta_x(\bt)\right)\right).\]
	\end{lem}
	
	\begin{proof}
		 Recall that for $j,n\in\NN,~j\leq n$:
		\begin{align*}
		&\p\left(M(\tau^*)=n\right)=\dfrac{1}{n}\E[N\indic_{\{S_n+N=n\}}],\\
		&\p_j(M(\tau^*)=n)=\dfrac{j}{n}\E[N\indic_{\{S_n+W_{j-1}+N=n\}}].
		\end{align*}
		For $i\geq 0$:
		\begin{align*}
			B_{n,-i}(x)&=\underset{j>-i}{\sum}\bp(j)\dfrac{\E[N\indic_{\{S_n+W_{j-1+i}+N=n\}}]}{\E[N\indic_{\{S_n+N=n\}}]}(j+i)\alpha_{j,x}\\
			&=\dfrac{n}{\E[N\indic_{\{S_n+N=n\}}]}\underset{j>-i}{\sum}\bp(j)\alpha_{j,x}\p_{j+i}(M(\tau^*)=n).
		\end{align*}
		We suppose that $x$ is marked on $\tau$, i.e. $\eta_x=1$. By decomposing $\tau$ under $\p_{i+1}$ with respect to the number of children of the root of the first tree in the forest ($\emptyset_1$). We consider that this root is associated to $x$, we get:
		\begin{equation}\label{probaforestmark}
			\p_{i+1}\left(M(\tau^*)=n+1,\eta_{\emptyset_1}=1\right) =\underset{j\in\N}{\sum}\bp(j)\bq(j)\p_{j+i}(M(\tau^*)=n).
		\end{equation}
		We explain this result with an example, for $j=5$, and $n=8$.
		\smallbreak
		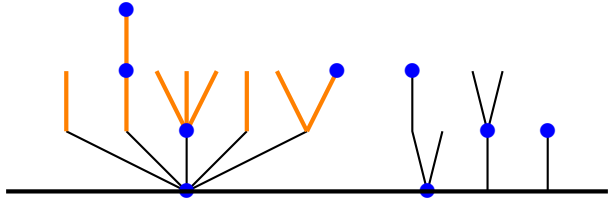
\begin{figure}[h!]
			\centering
			\begin{tikzpicture}[scale=0.8]
				\coordinate (1) at (-3,0);
				\coordinate (2) at (1,0);
				\coordinate (3) at (2,0);
				\coordinate (4) at (3,0);
				\coordinate (11) at (-5,1);
				\coordinate (12) at (-4,1);
				\coordinate (13) at (-3,1);
				\coordinate (14) at (-2,1);
				\coordinate (15) at (-1,1);
				\coordinate (111) at (-5,2);
				\coordinate (121) at (-4,2);
				\coordinate (131) at (-3.5,2);
				\coordinate (132) at (-3,2);
				\coordinate (133) at (-2.5,2);
				\coordinate (141) at (-2,2);
				\coordinate (151) at (-1.5,2);
				\coordinate (152) at (-0.5,2);
				\coordinate (1211) at (-4,3);
				\coordinate (21) at (0.75,1);
				\coordinate (22) at (1.25,1);
				\coordinate (211) at (0.75,2);
				\coordinate (31) at (2,1);
				\coordinate (311) at (1.75,2);
				\coordinate (312) at (2.25,2);
				\coordinate (41) at (3,1);
				\draw [thick](1) -- (11) ;
				\draw [ultra thick, orange](11)--(111);
				\draw [thick](1) -- (12);
				\draw [ultra thick, orange](12)--(121)--(1211);
				\draw [ultra thick, orange](13) -- (131);
				\draw [ultra thick, orange](13)--(132);
				\draw [thick](1) -- (13);
				\draw [ultra thick, orange](13)--(133);
				\draw [thick](1) -- (14);
				\draw [ultra thick, orange](14)-- (141);
				\draw [thick](1) -- (15);
				\draw [ultra thick, orange](15)-- (151);
				\draw [ultra thick, orange](15) -- (152);
				\draw [thick](2) -- (21) -- (211);
				\draw [thick](2) -- (22) ;
				\draw [thick](3) -- (31)--(311) ;
				\draw [thick](31) -- (312) ;
				\draw [thick](4) -- (41) ;
				\draw (1)node{\textcolor{blue}{\Large{$\bullet$}}} ;
				\draw (13)node{\textcolor{blue}{\Large{$\bullet$}}};
				\draw (121)node{\textcolor{blue}{\Large{$\bullet$}}};
				\draw (1211)node{\textcolor{blue}{\Large{$\bullet$}}} ;
				\draw (152)node{\textcolor{blue}{\Large{$\bullet$}}};
				\draw (2)node{\textcolor{blue}{\Large{$\bullet$}}};
				\draw (211)node{\textcolor{blue}{\Large{$\bullet$}}} ;
				\draw (31)node{\textcolor{blue}{\Large{$\bullet$}}};
				\draw (41)node{\textcolor{blue}{\Large{$\bullet$}}};
				\draw [ ultra thick] (-6,0)--(4,0);
				
			\end{tikzpicture}
			\smallbreak
			\caption{The forest of $\tau$ by decomposing under $\p_{i+1}$.}
			\label{exP_i+1}
		\end{figure}
		\begin{figure}[!h]
			\centering
			\centering
			\begin{tikzpicture}[scale=0.8]
				\coordinate (1) at (-3,0);
				\coordinate (2) at (1,0);
				\coordinate (3) at (2,0);
				\coordinate (4) at (3,0);
				\coordinate (11) at (-5,0);
				\coordinate (12) at (-4,0);
				\coordinate (13) at (-3,0);
				\coordinate (14) at (-2,0);
				\coordinate (15) at (-1,0);
				\coordinate (111) at (-5,1);
				\coordinate (121) at (-4,1);
				\coordinate (131) at (-3.5,1);
				\coordinate (132) at (-3,1);
				\coordinate (133) at (-2.5,1);
				\coordinate (141) at (-2,1);
				\coordinate (151) at (-1.5,1);
				\coordinate (152) at (-0.5,1);
				\coordinate (1211) at (-4,2);
				\coordinate (21) at (0.75,1);
				\coordinate (22) at (1.25,1);
				\coordinate (211) at (0.75,2);
				\coordinate (31) at (2,1);
				\coordinate (311) at (1.75,2);
				\coordinate (312) at (2.25,2);
				\coordinate (41) at (3,1);
				\draw [ultra thick, orange](11)--(111);
				\draw [ultra thick, orange](12)--(121)--(1211);
				\draw [ultra thick, orange](13) -- (131);
				\draw [ultra thick, orange](13)--(132);
				\draw [ultra thick, orange](13)--(133);
				\draw [ultra thick, orange](14)-- (141);
				\draw [ultra thick, orange](15)-- (151);
				\draw [ultra thick, orange](15) -- (152);
				\draw [thick](2) -- (21) -- (211);
				\draw [thick](2) -- (22) ;
				\draw [thick](3) -- (31)--(311) ;
				\draw [thick](31) -- (312) ;
				\draw [thick](4) -- (41) ;
				\draw (13)node{\textcolor{blue}{\Large{$\bullet$}}};
				\draw (121)node{\textcolor{blue}{\Large{$\bullet$}}};
				\draw (1211)node{\textcolor{blue}{\Large{$\bullet$}}} ;
				\draw (152)node{\textcolor{blue}{\Large{$\bullet$}}};
				\draw (2)node{\textcolor{blue}{\Large{$\bullet$}}};
				\draw (211)node{\textcolor{blue}{\Large{$\bullet$}}} ;
				\draw (31)node{\textcolor{blue}{\Large{$\bullet$}}};
				\draw (41)node{\textcolor{blue}{\Large{$\bullet$}}};
				\draw [ ultra thick] (-6,0)--(4,0);
				
			\end{tikzpicture}
			\caption{The forest of $\tau$ by decomposing under $\p_{i+j}$, with the first $j$ trees represent the trees resulting from the children of the root of the first tree on the figure \ref{exP_i+1}~.}
			\label{exdecP_i+1}
		\end{figure}
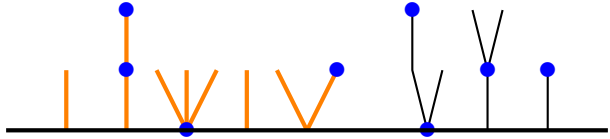
		\smallbreak
		On the figure \ref{exP_i+1}~, the first tree represents the tree with root $x$ on $\tau$. Figure \ref{exdecP_i+1} allows to explain our formula \eqref{probaforestmark}. In the sum we have the term $\bq(j)$ because we know $x$ is marked on $\tau$, since there are $n$ marks on the $j+i$ trees obtained (as the figure \ref{exdecP_i+1}), we have the term $\p_{j+i}(M(\tau^*)=n)$.
		\smallbreak
		Thereby, if $\eta_x=1$, we have:
		\[B_{n,-i}(x)
		=\dfrac{n\p_{i+1}(M(\tau^*)=n+1,\eta_{\emptyset_1}=1)}{\E[N\indic_{\{S_n+N=n\}}]}.\]
		Moreover, using Dwass formula like in \eqref{probajmark} and the fact that $\eta_{\emptyset_1}$ is independent of $(W_n)_{n\in\N}$ and $(S_n)_{n\in\N}$, we have:
		\begin{align*}
			\p_{i+1}(M(\tau^*)&=n+1,\eta_{\emptyset_1}=1)
			=\sum_{j\in\N}\p(W_{i}=j)\p_{j+1}\left(\left|  \tau^{\F_q}\right| = n+1\right)\p(\eta_{\emptyset_1}=1)\\
			&=\E[\bq(X)]\underset{j\in\N}{\sum}\p(W_{i}=j)\dfrac{j+1}{n+1}\p(S_{n+1}=n-j)\\
			&=\dfrac{\E[\bq(X)]}{n+1}\left(i\E[N\indic_{\{S_{n+1}+W_{i-1}+N=n\}}]+\E[\indic_{\{S_{n+1}+W_{i-1}+N=n\}}]\right).
		\end{align*}
		Then, thanks to Lemma \ref{limSnTR2}, for $T=N$ and $R=0$, Lemmata \ref{limSnN2} and \eqref{limNSn3bis}:
		\begin{align*}
		&\lim_{n\rightarrow+\infty}\frac{\E[N\indic_{\{S_{n+1}+W_{i-1}+N=n\}}]}{\E[N\indic_{\{S_n+N=n\}}]}=1\\
		&\lim_{n\rightarrow+\infty}\frac{\E[\indic_{\{S_{n+1}+W_{i-1}+N=n\}}]}{\E[N\indic_{\{S_n+N=n\}}]}=\dfrac{\E[X\bq(X)]+\ell_\bq(1-\mu(\bp))}{\E[\bq(X)]}.
		\end{align*}
		and as a result: 
		\begin{align*}
			\underset{n\to+\infty}{\lim}B_{n,-i}(x)= i\E[\bq(X)]+\E[X\bq(X)]+\ell_\bq\left(1-\mu(\bp)\right).
		\end{align*}
		Now, we suppose that $x$ is not marked on $\tau$, i.e. $\eta_x=0$. In the same way of \eqref{probaforestmark}, we obtain:
		$$\p_{i+1}(M(\tau^*)=n,\eta_{\emptyset_1}=0)=\underset{j\in\N}{\sum}\bp(j)(1-\bq(j))\p_{j+i}(M(\tau^*)=n),$$
		and we have:
		\begin{align*}
			&\p_{i+1}(M(\tau^*)=n,\eta_{\emptyset_1}=0)	
			=(1-\E[\bq(X)])\underset{j\in\N}{\sum}\p(W_{i}+Z^{(0)}=j)\dfrac{j}{n}\p(S_{n}=n-j)\\
			&=\dfrac{1-\E[\bq(X)]}{n}\left(i\E[N\indic_{\{S_{n}+W_{i-1}+Z^{(0)}+N=n\}}]+\E[Z^{(0)}\indic_{\{S_{n}+W_{i}+Z^{(0)}=n\}}]\right).
		\end{align*}
		Thereby, thanks to Lemma \ref{limSnN4} and \eqref{limNSn5bis}, we have:
		\begin{align*}
		&\lim_{n\rightarrow+\infty}\dfrac{\E[N\indic_{\{S_{n}+W_{i-1}+Z^{(0)}+N=n\}}]}{\E[N\indic_{\{S_{n}+N=n\}}]}=1\\
		&\lim_{n\rightarrow+\infty}\dfrac{\E[Z^{(0)}\indic_{\{S_{n}+W_{i-1}+Z^{(0)}+N=n\}}]}{\E[N\indic_{\{S_{n}+N=n\}}]}=\frac{\E[X(1-\bq(X))]+(1-\ell_\bq)\left(1-\mu(\bp)\right)}{1-\E[\bq(X)]}.
		\end{align*}
		Implying:
		\begin{align*}
			\lim_{n\rightarrow+\infty}B_{n,-i}(x)=i(1-\E[\bq(X)])+\E[X(1-\bq(X))]+(1-\ell_\bq)\left(1-\mu(\bp)\right).
		\end{align*}
		When we regroup these two results, we obtain the lemma.
	\end{proof}
	If $\bp^{(1)}$ is periodic, then Lemma \ref{limBn,lneg}~ still holds along the sub-sequence for which the denominator is positive when we use Lemma \ref{limSnN2}~.
	\bigbreak
	We adapt Lemmata 8.8 and 8.9 of \cite{abraham_local_2014}.
	In order to extend Lemma \ref{limBn,lneg} for $l>0$, we give a preliminary lemma and introduce for all $l,k\in\Z$ such that $l\geq k$:
	\[C_{n,l,x}(k):=\E\left[\alpha_{X,x} N (X-l)_+\indic_{\{S_{n}+W_{X-1-k}+N=n\}}\right].\]
	We recall that $z_+=\max{(z,0)}$. 
	\begin{lem}\label{limCn,k}
		Assume $\bp^{(1)}$ is aperiodic, with $\mu^{(1)}<1$, $\rho({\bp^{(1)}})=1$ and $0<\E[N]<+\infty$. We have for $k,l\in\Z$ such that $k\leq l$:
		\begin{equation}\label{limCn,l(k)}
			\underset{n\to+\infty}{\lim}\dfrac{C_{n,l,x}(k)}{C_{n,l,x}(l)}=1
		\end{equation}
		If $\bp^{(1)}$ is periodic, then \eqref{limCn,l(k)} still holds along the sub-sequence for which the denominator is positive.
	\end{lem}
	\begin{proof}
		Notice that $\alpha_{X,x} N (X-l)_+$ is integrable and, consequently, if we mimic the proof of Lemma \ref{limSnTR2}, we obtain:
				\begin{equation} \label{limCn,l(k)1}
			\underset{n\to+\infty}{\lim}\dfrac{\E\left[\alpha_{X,x} N (X-l)_+\indic_{\{S_{n+m}+W_{X-1-l}+N=n\}}\right]}{\E\left[\alpha_{X,x} N (X-l)_+\indic_{\{S_{n}+W_{X-1-l}+N=n\}}\right]}=1.
		\end{equation}
Now, to obtain our result, we use a reasoning similar to that of Lemma \ref{limSnN3}. Here, the quantity $c_{n,j}$ is defined by:
\[c_{n,j}=\dfrac{\E\left[\alpha_{X,x} N (X-l)_+\indic_{\{S_{n}+W_{X-1-l}+N=n-j\}}\right]}{\E\left[\alpha_{X,x} N (X-l)_+\indic_{\{S_{n}+W_{X-1-l}+N=n\}}\right]},\]
and using \eqref{limCn,l(k)1}, we can show that for $m\in\mathbb N$:
\begin{equation} \label{limCn,l(k)2}
			\lim_{n\rightarrow+\infty}\dfrac{\E\left[\alpha_{X,x} N (X-l)_+\indic_{\{S_{n}+W_{X-1-l+m}+N=n\}}\right]}{\E\left[\alpha_{X,x} N (X-l)_+\indic_{\{S_{n}+W_{X-1-l}+N=n\}}\right]}=1.
		\end{equation}
		Finally, with $m=l-k$ in \eqref{limCn,l(k)2}, we get the result.
	\end{proof}
	\begin{lem}\label{limBn,lpos}
		Assume that $\bp^{(1)}$ is aperiodic , with $\mu^{(1)}<1$, $\rho(\bp^{(1)})=1$, $\mathrm{Var}[Z^{(1)}]<+\infty$ and $0<\E[N]<+\infty$. 
		For $l>0$, we have 
	{\small	\begin{align*}
			\lim_{n\to+\infty}B_{n,l}(x)&=\E\left[\alpha_{X,x}(X-l)_+\right]+(1-\mu(\bp))\left(\ell_\bq\eta_x(\bt)+\left(1-\ell_\bq\right)\left(1-\eta_x(\bt)\right)\right).
		\end{align*}}
	\end{lem}
	\begin{proof}
		Let $l\geq -1$, We have:
		\begin{align}\label{expCn,l(-1)}
			C_{n,l,x}(-1)&=\E\left[\alpha_{X,x}(X-l)_+N \indic_{\{S_{n}+W_{X}+N=n\}}\right]\nonumber\\
			&=\sum_{j\ge l}\bp(j)\alpha_{j,x}(j-l)\E\left[N\indic_{\{S_{n}+W_{j}+N=n\}}\right]\nonumber\\
			&=C_{n,0,x}(-1)-\sum_{j=0}^{l-1}\bp(j)\alpha_{j,x}(j-l)\E\left[N\indic_{\{S_{n}+W_{j}+N=n\}}\right]\nonumber\\
			&-l\E\left[\alpha_{X,x}N \indic_{\{S_{n}+W_{X}+N=n\}}\right].
		\end{align}
		According to Lemma \ref{limBn,lneg}~, we have: 
		{\small\begin{align*}
		&\lim_{n\rightarrow+\infty}B_{n,0}(x)=\E[X\alpha_{X,x}]+(1-\mu(\bp))\left(\ell_\bq\eta_x(\bt)+\left(1-\ell_\bq\right)\left(1-\eta_x(\bt)\right)\right)=:\xi_X\\
		&\lim_{n\rightarrow+\infty}B_{n,-1}(x)=\xi_X+\E[\alpha_{X,x}]
		\end{align*}}
		Note that for all $l\in\mathbb Z$:
	\begin{equation}\label{equalidad}
		C_{n,l,x}(l)=
		B_{n,l}(x)\E[N\indic_{\{S_n+N=n\}}],
	\end{equation}
 as a result: 
		\begin{align}
		C_{n,-1,x}(-1)
		&\underset{n\to+\infty}{\sim}(\xi_X+\E[\alpha_{X,x}])\E\left[N\indic_{\{S_n+N=n\}}\right]\label{equiv1}\\
		 C_{n,0,x}(0)&\underset{n\to+\infty}{\sim}\xi_X\E\left[N\indic_{\{S_n+N=n\}}\right]\label{equiv2}.
		 \end{align}
		Thereby, using \eqref{equiv2} and  Lemma \ref{limCn,k}:
		\begin{equation}\label{firstlimBn,lpos}
			\lim_{n\to+\infty}\dfrac{C_{n,0,x}(-1)}{\E\left[N\indic_{\{S_n+N=n\}}\right]}=\underset{n\to+\infty}{\lim}\xi_X \dfrac{C_{n,0,x}(-1)}{C_{n,0,x}(0)}=\xi_X.
		\end{equation}
		Using \eqref{expCn,l(-1)} for $l=0$, $C_{n,-1,x}(-1)=C_{n,0,x}(-1)+\E\left[\alpha_{X,x}N \indic_{\{S_{n}+W_{X}+N=n\}}\right]$, thereby, with \eqref{equiv1} and \eqref{firstlimBn,lpos}, we deduce that:
		{\small\begin{align}\label{scdlimBn,lpos}
			\lim_{n\to+\infty}\dfrac{\E\left[\alpha_{X,x}N \indic_{\{S_{n}+W_{X}+N=n\}}\right]}{\E\left[N\indic_{\{S_n+N=n\}}\right]}&=\lim_{n\to+\infty}\dfrac{C_{n,-1,x}(-1)-C_{n,0,x}(-1)}{\E\left[N\indic_{\{S_n+N=n\}}\right]}
			=\E[\alpha_{X,x}].
		\end{align}}
		\smallbreak
		Let $l\geq1$, according to Lemma \ref{limSnN2}, \eqref{expCn,l(-1)}, \eqref{firstlimBn,lpos} and \eqref{scdlimBn,lpos} we obtain:
		\begin{align}\label{thirdlimBn,lpos}
			\dfrac{C_{n,l,x}(-1)}{\E\left[N\indic_{\{S_n+N=n\}}\right]}
			&\underset{n\to+\infty}{\longrightarrow} \xi_X-\overset{l-1}{\underset{j=0}{\sum}}\bp(j)\alpha_{j,x}(j-l) -l\E[\alpha_{X,x}].
		\end{align}
		Now, using successively \eqref{equalidad},  Lemma \ref{limCn,k} and \eqref{thirdlimBn,lpos}, we obtain:
		{\small\begin{align*}
			\lim_{n\to+\infty} B_{n,l}&(x) = \lim_{n\to+\infty} \dfrac{C_{n,l,x}(-1)}{\E\left[N\indic_{\{S_n+N=n\}}\right]} \dfrac{C_{n,l,x}(l)}{C_{n,l,x}(-1)}= \lim_{n\to+\infty} \dfrac{C_{n,l,x}(-1)}{\E\left[N\indic_{\{S_n+N=n\}}\right]}\\
			&=\E\left[\alpha_{X,x}(X-l)_+\right]+(1-\mu(\bp))\left(\ell_\bq\eta_x(\bt)+\left(1-\ell_\bq\right)\left(1-\eta_x(\bt)\right)\right).
		\end{align*}}
	\end{proof}
	If $\bp^{(1)}$ is periodic, then \ref{limBn,lpos}~ still holds along the sub-sequence for which the denominators are positive when we use \ref{limSnN2} and the expression \eqref{limCn,l(k)}.

		\section{Prospect}
			Unfortunately, we do not have a result in the general case where $\bp$ is sub-critical and non-generic. However, in this case, we conjecture the following results:
		\begin{conj}\label{conject}
			Let $\tau^{*}$ be a sub-critical and non-generic MGW with offspring distribution $\bp$ satisfying \eqref{condp}, $\rho_l(\bp,\bq)=1$, and mark function $\bq$ satisfying \eqref{condq}. 
			We have that:
			\begin{equation}
				\mathrm{dist}(\tau^{*}|M(\tau^*)=n)\underset{n\to+\infty}{\longrightarrow} \mathrm{dist}(\tau_C^{*}(\bp,\bq)).
			\end{equation}
		\end{conj}
		To obtain this result we need to have the equivalents for the randomly stopped sum like in \cite{bloznelis_local_2019} but in a general case. Thanks to this result we obtain:
		\begin{cor}\label{corsubcritnongen}
			Assume that $\bp$ satisfies \eqref{condp}, $\mu(\bp)<1$, is non-generic, $1\le\rho(\bp)=\rho_l(\bp,\bq)=\theta_s$ satisfies \eqref{condnongensup}, $\bp_{\rho(\bp)}$ has a moment of order $2$, and $\bq$ satisfies \eqref{condq}. We have:
			\begin{equation*}
				\lim_{n\to+\infty}\mathrm{dist}(\tau^{*}|M(\tau^*)=n)= \mathrm{dist}(\tau_C^{*}(\bp_{\rho(\bp)},\bq_{\rho(\bp)})),
			\end{equation*}	
			\[ \lim_{n\to +\infty}\mathrm{dist}(\tau| M(\tau)=n)=\mathrm{dist}(\tau_C(\bp_{\rho(\bp)})).\]
		\end{cor}
		
		\begin{proof}
			{We consider the case $1<\rho(\bp)<+\infty$.} 
			Thanks to Proposition \ref{loicondid},  $\mathrm{dist}(\tau^{*} | M(\tau^*)=n)=\mathrm{dist}(\tau_{\theta_s}^{*} | M(\tau^*_{\theta_s})=n)$, then, in order to apply Conjecture \ref{conject} to $(\bp_{\theta_s},\bq_{\theta_s})$ we have to check if we have $\rho_l(\bp_{\theta_s},\bq_{\theta_s})=1$ and $\E[X^2\rho(\bp)^X]$ is finite.\\
			
			For $x\ge0$, 
			we have: 
			\begin{align*}
				l_{\bp_{\theta_s},\bq_{\theta_s}}(x)&=\sum_{j\ge0}x^j\bp_{\theta_s}(j)(1-\bq_\theta(j))=\sum_{j\ge0}x^j{\theta_s}^{j-1}\bp(j)(1-\bq(j))\\
				&=\dfrac{1}{\theta_s}\E\left[(x\theta_s)^{X}\left(1-\bq(X)\right)\right].
			\end{align*}
			Consequently, $\rho_l(\bp_{\theta_s},\bq_{\theta_s})=\dfrac{\rho_l(\bp,\bq)}{\theta_s}=1$. 
		\end{proof}
		\begin{rqqq}
			Recall that:
			\[l(\theta)=\sum_{k\ge 0}\theta^k  \bp(k)(1-\bq(k)),\]
			and if $(\bq(k))_{k\ge0}$ admits a finite limit $\ell_\bq$, we have to distinguish two cases:
			\begin{itemize}
				\item $\ell_\bq\ne 1$ and clearly $\rho(\bp)=\rho_l(\bp,\bq)$; 
				\item otherwise, as \eqref{defqlim1}, there exists $\beta\ge 2$ such that:
				\[\bp(k)(1-\bq(k))=\frac{\bp(k)\Lq(k)}{k^\beta}, \]
				implying again that $\rho(\bp)=\rho_l(\bp,\bq)$.
			\end{itemize}
		\end{rqqq}
	\section*{Acknowledgements}
	
	The authors want to thank Patrick Maheux for several helpful discussions and his indications
	to obtain Lemma \ref{SV} or Corollary \ref{RV}.
	\bibliographystyle{abbrv}
	\bibliography{biblio21}
	
\end{document}